\documentclass[a4paper,11pt]{article}
\usepackage{mattsstyle}
\usepackage{tcolorbox}
\usepackage{soul}

\def\TLp#1{\mathrm{TL}^{#1}}

\def\Lp#1{\mathrm{L}^{#1}}

\def\Ck#1{\mathrm{C}^{#1}}
\def\Ckc#1{\Ck{#1}_{\mathrm{c}}}
\def\Wkp#1#2{\mathrm{W}^{#1,#2}}

\def\spaceBar{\, | \,}

\def\uni{\mathrm{uni}}

\DeclareMathOperator{\graph}{gph}
\DeclareMathOperator{\range}{ran}
\DeclareMathOperator{\support}{supp}
\DeclareMathOperator{\domain}{dom}
\DeclareMathOperator{\signFunc}{sign}

\DeclareMathOperator{\diver}{div}
\DeclareMathOperator{\closure}{cl}

\DeclareMathOperator{\uTimeInter}{u_{\mathrm{TimeInt}}}
\DeclareMathOperator{\uTimeInject}{u_{\mathrm{TimeInj}}}

\DeclareMathOperator{\vTimeInter}{v_{\text{TimeInt}}}
\DeclareMathOperator{\vTimeInject}{v_{\text{TimeInj}}}

\newcommand{\bs}[1]{{\boldsymbol{#1}}}
\newcommand{\bi}{{\bs{i}}}
\newcommand{\bx}{{\bs{x}}}
\newcommand{\bz}{{\bs{z}}}
\newcommand{\pa}[1]{\left(#1\right)}
\newcommand{\abs}[1]{\left|#1\right|}

\DeclareMathOperator{\expW}{expW}
\DeclareMathOperator{\llSgt}{\overset{\mathrm{sgt}}{\ll}}

\def\commentOut#1{}

\allowdisplaybreaks

\title{Discrete-to-Continuum Rates of Convergence for $p$-Laplacian Regularization}
\author[1]{Adrien Weihs}
\author[2]{Jalal Fadili}
\author[3]{Matthew Thorpe}
\date{\add{July} \removetest{March} 2023}
\affil[1]{School of Mathematics,\protect\\ University of Manchester,\protect\\ Manchester, M13 9PL, UK. \vspace{\baselineskip}}
\affil[2]{Normandie Universit\'e, ENSICAEN, UNICAEN, CNRS, GREYC,\protect\\ Caen, France \vspace{\baselineskip}}

\affil[3]{Department of Statistics,\protect\\ University of Warwick,\protect\\ Coventry, CV4 7AL, UK.}

\date{\today}

\begin{document}

\maketitle

\begin{abstract}
Higher-order regularization problem formulations are popular frameworks used in machine learning, inverse problems and image/signal processing. In this paper, we consider the computational problem of finding the minimizer of the Sobolev $\mathrm{W}^{1,p}$ semi-norm with a data-fidelity term. We propose a discretization procedure and prove convergence rates between our numerical solution and the target function. Our approach consists of discretizing an appropriate gradient flow problem in space and time. The space discretization is a nonlocal approximation of the $p$-Laplacian operator and our rates directly depend on the localization parameter $\eps_n$ and the time mesh-size $\tau_n$. We precisely characterize the asymptotic behaviour of $\eps_n$ and $\tau_n$ in order to ensure convergence to the considered minimizer. Finally, we apply our results to the setting of random graph models.
\end{abstract}



\section{Introduction} 

In machine learning, inverse problems and image/signal processing, one is often confronted with finding smooth solutions to regression problems. This leads to the introduction of regularization terms in the problem formulation (see for example \cite{groetsch1984theory}, \cite{benning_burger_2018},  \cite{engl2000regularization}). In this paper, we investigate the computation of the solution to the following regularization problem:
\begin{equation}\label{eq:intro:regularizationProblem}
    u_{\infty} \in \argmin_{v \in \Wkp{1}{p}(\Omega)} \cF(v) :=  \frac{\mu}{p}\Vert \nabla v \Vert_{\Lp{p}(\Omega)}^p + \frac{1}{2}\Vert \cA v - \ell \Vert_{\Lp{2}(\Omega)}^2
\end{equation}
for some $\mu > 0$, linear operator $\cA$, data function $\ell$, $\Omega \subseteq \bbR^d$ and $p \geq 1$. 
The first term on the right-hand side of the latter enforces some regularization upon the functions $v$ while the second term is a data-fidelity term. For $\cA = \Id$, this is also called a fully supervised machine learning problem or a denoising problem in image processing. More precisely, in this paper, we prove convergence rates between our numerical discrete solution and $u_\infty$. This type of regularization has been considered for example in \cite{1333694}, \cite{liu2018p}, \cite{doi:10.1137/15M1022793} and \cite{10.1007/11550518_45}; for examples involving nonlocal versions of the $\Wkp{1}{p}$ semi-norm we refer to \cite{aubert2009nonlocal}, \cite{HafieneEvolution2} and references therein; for other higher-order regularization problems we refer for example to \cite{https://doi.org/10.48550/arxiv.2209.02305} and references therein.

The aim of this paper can be seen as the converse to a large part of the discrete-to-continuum work in recent years: indeed, in \cite{Slepcev}, \cite{Stuart}, \cite{Trillos3}, \cite{https://doi.org/10.48550/arxiv.2303.07818} and \cite{weihs2023consistency} for example, one starts with a discrete problem and analyses the large data limit of the latter -- a point of interest in machine learning settings.
Similarly to \cite{ElBouchairi}, \cite{HafieneEvolution} and \cite{HafieneEvolution2}, we start from the continuum and explain how the appropriate discretization should be designed as is usually the case in numerical analysis.  In particular, this means that in the former case, emphasis is partly placed on having as few constraints as possible on different parameters as the latter are usually inferred from the data at hand: in our case, as we are choosing all the parameters, this concern is less relevant. 

The approach we choose in this paper is to discretize the gradient flow \eqref{eq:main:notation:localProblem:localProblem} associated to \eqref{eq:intro:regularizationProblem}
which contains the $p$-Laplacian operator $\Delta_p u = \diver(\vert \nabla u \vert^{p-2} \nabla u )$ (see \cite{lindqvist} and references therein). We note two important consequences from this reformulation of the task at hand: adding a time dependence to our problem will allow us to leverage the theory of semigroups in Banach spaces (see Section \ref{sec:background:nonlinear}) in order to derive rates; the main problem is to obtain convergence rates between the $p$-Laplacian operator and its discrete counterpart. 

The plan for the discretization follows loosely the strategy in \cite{Trillos3} and \cite{Slepcev}: after choosing an appropriate kernel $K$, we start by introducing a nonlocal version of the $p$-Laplacian operator in the continuum and then discretize the latter. The former is inspired by \cite{Bourgain01anotherlook}, \cite{ponce2004}, \cite{leoni2009first} where finite-difference approximations of Sobolev norms are discussed in the continuum, while the latter step allows us to use existing results on convergence results in the nonlocal setting as in \cite{ElBouchairi}, \cite{HafieneEvolution}, \cite{HafieneEvolution2}. 

Passing from the nonlocal setting to the local one requires a kernel that is appropriately scaled. For some length-scale $\eps_n \to 0$, it is shown for example in \cite{Slepcev}, \cite{Calder_2018} and \cite{ANDREU2008201} that, ignoring regularity assumptions,  
\[
\frac{1}{\eps_n^{d+p}} \int K\l\frac{\vert y - x\vert}{\eps_n} \r \vert u(y) - u(x) \vert^{p-2} (u(y) - u(x) ) \, \dd y \to \Delta_{p}u (x).
\]
The finite-difference structure of the above nonlocal approximation is essential for obtaining rates in the continuum: indeed, as in \cite{Calder_2018}, for smooth enough functions, we pass from finite-differences to derivatives and the rates follow from a conceptually basic Taylor expansion.

The discrete-to-continuum step requires more subtle techniques. In fact, one needs a way to compare functions $\bar{u}_n \in \bbR^{D_n}$ and $u:\Omega \mapsto \bbR$. We discuss various alternatives in Section \ref{subsec:relatedWorks}. In this paper, we choose to partition our space $\Omega$ into $D_n$ cells and elements of $\bbR^{D_n}$ are injected through the operator $\cI:\bbR^{D_n} \to \Lp{1}(\Omega)$ in the continuum through piecewise constant approximations while continuum functions are projected through the operator $\cP:\Lp{1}(\Omega) \to \bbR^{D_n}$ onto our cells by averaging on each cell. Using this method, establishing rates between a continuum function and its injected discrete approximation relies on tools from approximation theory, depending on the partition chosen and the regularity assumption of the continuum function. This topic is discussed in greater detail in Section \ref{subsec:approximations}. The other central tool for the discrete-to-continuum rates, where the discretization is both in space and time, is the semigroup structure of our solutions to the nonlocal gradient flows. Indeed, relying on some favourable properties of the nonlocal $p$-Laplacian, we are able to deduce strong contraction properties as discussed in Section \ref{sec:background:nonlinear}.

Combining the discrete-to-continuum rates in the nonlocal setting to the nonlocal-to-local rates in the continuum, we obtain in Corollary \ref{cor:proofs:rates:discreteNonlocalContinuum:final} that for $p \geq 3$,
\begin{align} 
        \Vert \cI_n \bar{u}_n^N - u_\infty \Vert_{\Lp{2}}& \leq C \Bigg( \eps_n\log(\eps_n^{-\kappa}) + \eps_n^{\kappa/4}(\cF(u_0)-\cF(u_\infty))^{1/2} \label{eq:intro:rates} \\
  &+ \eps_n^{-\kappa}\ls n^{-\alpha_1} + n^{-\alpha_2} + \frac{\log(\eps_n^{-\kappa})^{(p-1)}}{\eps_n^{d+p + \alpha_3} n^{\alpha_3}} + \tau_n \frac{\log(\eps_n^{-\kappa})^{2p-3}}{\eps_n^{2(d+p)}} \rs  \Biggr) \notag
\end{align}
where $u_\infty$ is the solution to \eqref{eq:intro:regularizationProblem}, $\cI_n \bar{u}_n^N$ is the injected discretized solution on a partition indexed by $n$ in space and $N$ in time, $\tau_n$ is the maximum step-size of the forward Euler time-discretization (we pick $0 = t^0 < t^1 < \dots < t^N \approx \log(\eps_n^{-\kappa})$ so that $N \geq \log(\eps_n^{-\kappa})/\tau_n $), $C>0$ is a constant that is independent of $n$, $\kappa >0$ and $\alpha_i >0$ are chosen numerical constants depending on the regularity of the initial condition of the gradient flow problem \eqref{eq:main:notation:localProblem:localProblem} $u_0$, the kernel $K$ and the data $\ell$. 

First, we note that each term in the right-hand side of \eqref{eq:intro:rates} corresponds to a specific error source, namely (from left to right) the continuum nonlocal-to-local approximation, the gradient flow convergence, the discrete-to-continuum approximation of $u_0$, $\cA^* \ell$ and $K$ as well as a general discretization error. Furthermore, as the choices of $\alpha_i$ and $\kappa$ are left to the practitioner, the rates can be enhanced upon implementation. 

Second, while we give the precise statement of this result in Section \ref{sec:main:main}, we note that the right-hand side of \eqref{eq:intro:rates} tending to $0$ involves finding the right interplay between $\tau_n \to 0$ and $\eps_n \to 0$: this is similar to Courant-Friedrichs-Lewy (CFL) conditions \cite{5391985} and \cite{de2012courant}. In particular, while the classical CFL conditions correspond to $\tau_n \ll \eps_n^2$ for the heat equation with the forward Euler time-discretization,
we will show in Corollary \ref{cor:proofs:rates:discreteNonlocalContinuumLocal:simplified} that our requirement is roughly $\tau_n \ll \eps_n^{2(d+p)}$.
We also find that $\eps_n$ admits a lower bound and this is analogous to results in semi-supervised learning discrete-to-continuum phenomena in \cite{Slepcev}, \cite{Stuart} and \cite{weihs2023consistency}. 

Third, we note that \eqref{eq:intro:rates} does not cover the linear case of $p = 2$. This is due to a technicality and indeed, our well-posedness results both in the nonlocal case (Theorem \ref{thm:proofs:wellPosedness:nonlocalProblem:existenceUniqueness}) and the local case (Theorem \ref{thm:proofs:wellPosedness:localProblem:existenceUniqueness}) only require $p \geq 2$. The $p \geq 2$ assumption is particularly helpful since it allows one to have $\Lp{p/(p-1)}(\Omega) \subseteq \Lp{p}(\Omega)$ and $\Wkp{1}{p}(\Omega) \subseteq \Lp{2}(\Omega)$. When establishing continuum rates in Theorem \ref{thm:proofs:continuumRates:continuumNonlocalLocal} however, similarly to what is presented in \cite{Calder_2018}, we will have to consider a third-order Taylor expansion of the function $x \mapsto \vert x \vert^{p-2} x$, hence requiring $p \geq 3$. We nevertheless remark that the choice of $p$ is left to the practitioner and, as explained in Remark \ref{rem:proofs:rates:continuumRates:embeddings}, should be made in accordance with the dimension of $\Omega$ in order to have a small ratio $d/p$.

Lastly, coming back to a more data-centric approach, we apply the above-mentioned results to a random graph model. The graph models appear in several applications and we obtain results equivalent to the ones displayed above in Corollary \ref{cor:application:final}: the main difference is an additional term accounting for the discrete random-to-deterministic approximation error.

\subsection{Contributions}

Our main contributions in the paper which we discuss in greater detail in Section \ref{sec:main} are: \begin{enumerate}
    \vspace{-2mm}
    \setlength\itemsep{-1mm}
    \item A rigourous proof of the well-posedness of the nonlocal continuum gradient flow defined in \eqref{eq:main:notation:nonlocal:nonlocalProblem};
    \item The establishment of rates of convergence between the discrete gradient flow and $u_\infty$ through a precise characterization of the discretization parameters in Corollary \ref{cor:proofs:rates:discreteNonlocalContinuum:final};
    \item An application to random graph models in Section \ref{subsec:application}.
\end{enumerate}

\subsection{Related works} \label{subsec:relatedWorks}

\paragraph{$p$-Laplacian operator approximations.}
Approximating the $p$-Laplacian operator has been explored in \cite{10.2307/2158008},\cite{M2AN_1975__9_2_41_0},\cite{10.2307/2158136},\cite{10.2307/2153239} but always in the context of finite elements. This simplifies some part of the analysis as described in Remark \ref{rem:proofs;rates:continuumRates:increasedRegularity} but has the disadvantage of being difficult to apply in higher dimensions. 

In the continuum setting, nonlocal-to-local convergence of gradient flows involving the $p$-Laplacian operator is shown in \cite{ANDREU2008201}. This relies heavily on \cite{Bourgain01anotherlook} and is a consistency result, without rates. Some rates are established between the nonlocal and local operators in \cite{Calder_2018}.

Much of the discrete-to-continuum work in recent years has dealt with similar problems in many ways. The closest results are to be found in \cite{ElBouchairi}, \cite{HafieneEvolution} and \cite{HafieneEvolution2} where rates are established for some discrete-to-continuum problems involving the $p$-Laplacian in the nonlocal case using the same discretization procedures.

\paragraph{Discrete-to-continuum work.}
Other energies and operators have also been studied under the discrete-to-continuum lens, for example the eikonal equation \cite{eikonalJalal}, total variation \cite{Trillos3}, the Ginzburg-Landau functional \cite{diffuseInterfaceBertozzi}, \cite{Gennip}, \cite{cristoferi_thorpe_2020}, \cite{thorpe2019asymptotic}, the Mumford-Shah functional \cite{Caroccia_2020}, an application in empirical risk minimization \cite{garcia_trillos_murray_2017}, various Sobolev semi-norms \cite{Slepcev}, \cite{GARCIATRILLOS2018239}, \cite{Stuart}, \cite{weihs2023consistency} or variants of the Laplacian operator \cite{calderLipschitzLearning}, \cite{bungertRates}, \cite{Bungert}, \cite{calder2020rates}.

For alternatives approaches to the discrete-to-continuum, we refer to \cite{convergenceEigenmaps}, \cite{COIFMAN20065}, \cite{Gine}, \cite{hein1}, \cite{hein2}, \cite{singer}, \cite{ting} for pointwise convergence where one restricts a continuum function to a discrete domain or to \cite{Trillos}, \cite{CALDER2022123}, \cite{GARCIATRILLOS2018239}, \cite{consistencySpectral}, \cite{Pelletier}, \cite{wang}, \cite{Davydov} for spectral convergence where one considers the convergence of eigenvalues and eigenfunctions of the operators. Of course, as eigenfunctions are themselves functions, the papers on eigenfunctions rely on pointwise convergence or $\TLp{p}$-convergence which was introduced in \cite{Trillos3}.

\paragraph{}
The rest of the paper is organized as follows: in Section \ref{sec:background} we introduce the main theoretical tools required for the proofs; in Section \ref{sec:main} we present our main results; in Section \ref{sec:proofs} we prove our results.

\section{Background} \label{sec:background}

\subsection{General notation}

For $\Omega \subseteq \bbR^d$, we denote by $\overrightarrow{n}$ the outward normal vector to the boundary $\partial \Omega$ of $\Omega$. We denote the identity operator by $\Id$ and the indicator function of a set $A$ by $\chi_A$. We will write $\closure(\Omega)$ to denote the closure of $\Omega$. For $T > 0$, let $\lambda_t$ and $\lambda_x$ respectively be the Lebesgue measure on $[0,T]$ and $\Omega$. Elements of a discrete space will be over-lined, for example $\bar{u} \in \bbR^d$. The $i$-th component of $\bar{u}$ is denoted by $(\bar{u})_i$. We will denote Lebesgue spaces by $\Lp{p}$, the space of functions with $k$-th continuous derivatives by $\Ck{k}$, H\"older spaces by $\mathrm{C}^{k,\alpha}$ and Sobolev spaces by $\Wkp{k}{p}$. We write $\Vert \cdot \Vert_{\Lp{p}(A)}$ for the $\Lp{p}$-norm over a space $A$ and $\Vert \cdot \Vert_{\Lp{p}}$ when we take the norm over the whole space $\Omega$ or $\Omega \times \Omega$ depending on the domain of the function. We also write that $u_n \rightharpoonup u$ if $u_n$ converges weakly to $u$.

\paragraph{Functions}
We define Lambert's $W:[0,\infty) \mapsto [0,\infty)$ function \cite{lambertBook} as the inverse of the function $x e^x$: for every $y > 0$, we have $W(y)e^{W(y)} = y$. It is clear that $W$ is increasing on $[0,\infty)$. From this, it follows that the function $\expW:[0,\infty) \mapsto [1,\infty)$ defined as $\expW(y) = e^{W(y)}$ is also increasing and solves $\expW(y) \log(\expW(y)) = y$ for every $y > 0$. 

\paragraph{Asymptotics}
For two functions $f:\bbN \mapsto [0,\infty)$ and $g:\bbN \mapsto [0,\infty)$, we will write $f(n) \gg g(n)$ if
$\lim_{n \to \infty} \frac{g(n)}{f(n)} = 0.$
Therefore, $1 \gg f$ means that $\lim_{n \to \infty} f(n) = 0$.

\subsection{Nonlinear semigroup theory} \label{sec:background:nonlinear}

We now introduce a few elements of nonlinear semigroup theory that will be useful in the rest of the paper. However, we stress that a thorough and proper treatment of the subject can, for example, be found in \cite{andreu2004parabolic}, \cite{barbuNonlinear}, \cite{CRANDALL1976131}, \cite{Brezis:1663074}, \cite{BenilanCrandall},  \cite{pavel1987nonlinear} and references therein. 

For a Banach space $V$ with dual space $V^*$ and norm $\Vert \cdot \Vert$, we call a map $A:V \mapsto 2^{V}$ an operator. The domain and range of $A$, respectively denoted by $\domain(A)$ and $\range(A)$ are defined as $\domain(A) = \{ v \in V \spaceBar Av \neq \emptyset  \}$ and $\range(A) = \{y \spaceBar y\in Av \text{ for some } v \in \domain(A) \}$. The graph of an operator is defined as $\graph(A) = \{ (v,w) \in V \times V \spaceBar w \in Av \}$. 

Let us consider the general Cauchy problem for an operator $A$ on a Banach space $V$:
\begin{equation}
    \begin{cases} \label{eq:background:orderRelation:abstractCauchyProblem}
    u'(t) + Au(t) \ni f(t) & \text{on $t \in (0,T)$,} \\
    u(0) = u_0
    \end{cases} \tag{$\mathrm{CP}_{f,u_0}$}
\end{equation}
for some $f:(0,T) \mapsto V$ and $u_0 \in V$. 
Various concepts of solution have been designed for the problem \eqref{eq:main:notation:nonlocal:nonlocalProblem} and we refer to \cite{andreu2004parabolic} for a brief review of abstract Cauchy problems. We will rely on the following notion of solution as in \cite[Definition A.3]{andreu2004parabolic}. 
We note that the following type of solutions are called strong solutions in the nonlinear semigroup literature.

\begin{mydef}[Solution to \eqref{eq:background:orderRelation:abstractCauchyProblem}] \label{def:main:notation:nonlocal:strongSolutionCP}
A function $u$ is called a strong solution to \eqref{eq:background:orderRelation:abstractCauchyProblem} if $u(t,x) \in C([0,T];V) \cap \mathrm{W}_{\mathrm{loc}}^{1,1}((0,T);V)$, $u(0,\cdot) = u_0$ and $u'(t) + Au(t) \ni f(t)$ $\lambda_t$-a.e.. 
\end{mydef}

Our aim is to introduce a class of operators that will be of particular interest when solving the abstract Cauchy problem \eqref{eq:background:orderRelation:abstractCauchyProblem}.

\begin{mydef}[Accretive operator] \label{def:background:accretive}
We say that an operator A on $V$ is accretive if
\[
\Vert v - w \Vert \leq \Vert v - w + \lambda(\hat{v} - \hat{w} ) \Vert 
\]
for all $\lambda > 0$ and $(v,\hat{v}),(w,\hat{w}) \in \graph(A)$.
\end{mydef}

If we assume that $A$ is single valued (so we can write $A(u)$ without any ambiguity) then $A$ being accretive implies 
\[ \| v-w \| \leq \lambda \lda \l A+\frac{1}{\lambda} \Id\r(v) - \l A+\frac{1}{\lambda} \Id\r(v) \rda \]
which implies that $\l A+\frac{1}{\lambda}\Id\r^{-1}$ exists (and is Lipschitz with constant $\lambda$).

\begin{mydef}[$m$-accretive operators] \label{def:background:maccretive}
We say that an operator $A$ on $V$ is $m$-accretive if $A$ is accretive and $\range(\Id + \lambda A) = V$ for all $\lambda >0$.
\end{mydef}

The $m$-accretivity property of operators is a sufficient condition for the existence of a solution to \eqref{eq:background:orderRelation:abstractCauchyProblem} (see \cite[Chapter 2, Theorem 10.2]{pavel1987nonlinear}). 

\begin{theorem}[Existence of a solution to \eqref{eq:background:orderRelation:abstractCauchyProblem}]\label{thm:background:solution}
    Let $V$ be a reflexive Banach space, $A$ a $m$-accretive operator, $x_0 \in \domain(A)$ and $f \in \Wkp{1}{1}(0,T;V)$, then there exists a unique solution as in Definition \ref{def:main:notation:nonlocal:strongSolutionCP} to \eqref{eq:background:orderRelation:abstractCauchyProblem}. 
\end{theorem}

Accretivity of operators can be viewed as a generalization of monotony in Hilbert spaces (see \cite[Corollary A.13, Example A.14]{andreu2004parabolic}) -- the equivalence between $m$-accretivity and maximal monotony is called the Minty Theorem \cite[p.284]{andreu2004parabolic}. 
We recall that an operator $A$ on a Hilbert space $H$ is monotone if $\langle v-w, \hat{v}-\hat{w} \rangle_H \geq 0$ for all $(v,\hat{v}),(w,\hat{w}) \in \graph(A)$. 

\begin{proposition}[Equivalence of accretivity and monotony] \label{prop:background:monotony}
    Let $A$ be an operator in a Hilbert space $H$. Then, A is accretive in $H$ if and only if $A$ is monotone.
\end{proposition}

We specialize our discussion by considering Banach spaces that are spaces of functions defined on some $\Omega \subseteq \bbR^d$. Let $\Lp{0}(\Omega)$ be the set of measurable functions on $\Omega$ that map to $\bbR$ and define the two set of functions
\[
\cH = \{h \in \Ck{\infty}(\bbR) \spaceBar 0 \leq h' \leq 1, \text{ $\support(h')$ is compact, } h(0) = 0 \}
\]
and 
\[
\cJ = \{j:\bbR \mapsto [0,\infty] \spaceBar \text{$j$ is convex, lower semi-continuous and satisfies $j(0) = 0$}\}.
\]
Then, we can use the following notation for two functions $v,w \in \Lp{0}(\Omega)$: 
\[
v \llSgt w \quad \text{if and only if } \int_\Omega j(v(x)) \, \dd x \leq \int_\Omega j(w(x)) \, \dd x \text{ for all $j \in \cJ$}. 
\]

\begin{mydef}[Completely accretive operators] \label{def:background:completelyAccretive}
We say that an operator $A$ on $\Lp{0}(\Omega)$ is completely accretive if 
$v - w \llSgt v - w + \lambda (\hat{v} - \hat{w})$
for all $\lambda > 0$ and $(v,\hat{v}),(w,\hat{w}) \in \graph(A)$.
\end{mydef}

In Section \ref{subsec:wellPosedness}, we will consider operators $A$ such that $\graph(A) \subseteq \Lp{1}(\Omega) \times \Lp{1}(\Omega)$ with $\lambda_x(\Omega) < \infty$. In this case, another useful characterization of completely accretive operators on $\Lp{1}(\Omega)$ uses $\cH$ (see \cite[Corollary A.38]{andreu2004parabolic}).

\begin{proposition}[Characterization of completely accretive operators] \label{prop:background:completelyAccretiveCharacterization}
Let $A$ be an operator with $\graph(A) \subseteq \Lp{1}(\Omega) \times \Lp{1}(\Omega)$ and $\lambda_x(\Omega) < \infty$. Then, $A$ is completely accretive if and only if
\[
\int_\Omega h(v(x) - w(x))(\hat{v}(x) - \hat{w}(x) \, \dd x \geq 0.
\]
for any $h \in \cH$ and $(v,\hat{v}),(w,\hat{w}) \in \graph(A)$.
\end{proposition}

One can combine Definition \ref{def:background:maccretive} and Definition \ref{def:background:completelyAccretive} to define $m$-completely accretive operators. We now present a variant of \cite[Theorem A.20]{andreu2004parabolic} whose proof can be found in Section \ref{sec:supplementary:background}.

\begin{lemma}[Contraction property for completely accretive operators] \label{lem:background:orderRelation:contraction} 
Let $\Omega \subseteq \bbR^d$ be bounded, $p\geq 1$, $A$ be a completely accretive operator with $\graph(A) \subseteq \Lp{p}(\Omega) \times \Lp{p}(\Omega)$
and $u$ and $v$ be solutions as in Definition \ref{def:main:notation:nonlocal:strongSolutionCP} that respectively solve $(\mathrm{CP}_{f,u_0})$ and $(\mathrm{CP}_{g,v_0})$. Then, for any $1 \leq r \leq \infty$ and $0 \leq t \leq T$, we have:
\begin{equation} \label{eq:background:orderRelation:contraction}
\Vert u(t) - v(t) \Vert_{\Lp{r}} \leq \Vert u_0 - v_0 \Vert_{\Lp{r}} + \int_{0}^t \Vert f(s) - g(s) \Vert_{\Lp{r}} \, \dd s.
\end{equation}
In particular, we have that this type of solutions to $(\mathrm{CP}_{f,u_0})$ are unique. 
\end{lemma}

Our motivation for the use of nonlinear semigroup theory really stems from Lemma \ref{lem:background:orderRelation:contraction}. Indeed, we will consider Cauchy problems of the form \eqref{eq:background:orderRelation:abstractCauchyProblem} for our gradient flows. One could be tempted to use direct results for the existence of gradient flows such as described in \cite[Section 8]{santambrogio2015optimal} for example. However, we would not obtain the very strong contraction property \eqref{eq:background:orderRelation:contraction} which will be essential to establish our rates in Section \ref{sec:proofs}.

\subsection{Nonlinear problems}

In this subsection, we will introduce several results related to nonlinear problems which will be used in Section~\ref{subsec:wellPosedness}.

\begin{mydef}[Hemicontinuity]
    Let $A:V \mapsto V^*$ be an operator. We say that $A$ is hemicontinous if for all $v,w,z \in V$, the function $\lambda \mapsto \langle A(v + \lambda w), z \rangle_{V^*}$ is continuous from $\bbR$ to $\bbR$.
\end{mydef}

A weak version of \cite[Lemma 6.2]{fracu1990} is as follows and will be used for the well-posedness result of the solution to the problem \eqref{eq:main:notation:localProblem:localProblem} in the proof of Theorem  \ref{thm:proofs:wellPosedness:localProblem:existenceUniqueness}.

\begin{lemma}[Continuity implies hemicontinuity] \label{lem:background:hemicontinuity}
    Let $A:V\mapsto V^*$ be an operator. If $A$ is continuous, then $A$ is hemicontinuous.
\end{lemma}

The next result deals with the existence of a solution to a nonlinear stationary problem \cite[Chapter 2, Theorem 2.1]{lions1969quelques}. This will be an essential step in the proof of the range condition of our operators in Proposition~\ref{prop:proofs:wellPosedness:nonlocalProblem:completeAccretivityRangeCondition}.

\begin{theorem}[Solutions to nonlinear stationary problems] \label{thm:background:nonlinearStationary}
    Let $V$ be a separable reflexive Banach space and $A:V\mapsto V^*$ an operator which is bounded, hemicontinuous and 
    satisfies 
    \[
    \langle A(v) - A(w), v-w \rangle_{V^*} \geq 0 \qquad \text{for all $v,w \in V$ as well as}
    \]
    \[
    \frac{\langle A(v), v \rangle_{V^*}}{\Vert v \Vert_V} \to \infty \qquad \text{when $\Vert v \Vert_V \to \infty$.}
    \]
Then, for every $f \in V^*$, there exists $u \in V$ such that 
\[
A(u) = f.
\]
\end{theorem}

For the existence of the local problem, we will use the following result on nonlinear evolution problems which can be found in \cite[Chapter 2, Theorem 1.4 and Remark 1.13]{lions1969quelques}.

\begin{theorem}[Solutions to nonlinear evolution problems] \label{thm:background:nonlinearEvolution}
    Let $H$ be an Hilbert space and $V_i$ reflexive Banach spaces with $V_i \subseteq H$ and $V_i$ dense in $H$ for $1 \leq i \leq q$. Assume that $\cap_{i=1}^q V_i$ is seperable and dense in $H$. For $1 \leq i \leq q$, let $A_i:V_i \mapsto V_i^*$ be an operator which is bounded, hemicontinuous and, for some $1 < p_i < \infty$ and $c_i > 0$, satisfies: 
    \begin{enumerate}
        \item $\Vert A_i(v) \Vert_{V_i^*} \leq c_i \Vert v \Vert_{V_i}^{p_i-1}$;
        \item $\langle A_i(v) - A_i(w), v-w \rangle_{V_i^*} \geq 0$ for all $v,w \in V_i$;
        \item there exists a seminorm $[\cdot]_i$ on $V_i$ with constants $\alpha_i,\lambda_i,\beta_i > 0$ such that $[\cdot]_i + \lambda_i \Vert \cdot \Vert_{H} \geq \beta_i \Vert \cdot \Vert_{V_i}$ on $V_i$ and $\langle A_i(v) , v \rangle_{V_i^*} \geq \alpha_i [v]_i^{p_i}$ for all $v \in V_i$.
    \end{enumerate}
    Then, for $u_0 \in H$ and $f \in \sum_{i=1}^q \Lp{p_i^*}(0,T;V_i^*)$ where $1/p_i + 1/p_i^* = 1$, there exists a unique function $u$ such that 
    \[
    u \in \bigcap_{i=1}^q \Lp{p_i}(0,T;V_i), \quad u \in \Lp{\infty}(0,T;H)
    \]
    that satisfies
    \[
    \frac{\partial u}{\partial t} + \sum_{i=1}^q A_i(u) = f \quad \text{and} \quad u(0) = u_0.
    \]
\end{theorem}

\subsection{Piecewise constant approximations} \label{subsec:approximations}

We start by introducing a space of functions which will be relevant for our discrete-to-continuum approximations. We refer to \cite{triebel2006theory} for a detailed discussion. 

\begin{mydef}[Lipschitz spaces]
Let $\Omega$ be an open bounded subset of $\bbR^d$. For $g \in \Lp{q}(\Omega)$ with $q \in [1,+\infty)$, we define the (first-order) $\Lp{q}(\Omega)$ modulus of smoothness by
\[
\omega(g,h)_q =
 \sup_{z \in \bbR^d, |z| < h} \l \int_{x, x + z \in \Omega}|g(x + z)-g(x)|^q \, \dd x \r^{1/q}.
\]
For $0 <s \leq 1$, the Lipschitz spaces $\Lip(s,\Lp{q}(\Omega))$ consist of all functions $g \in \Lp{q}(\Omega)$ for which
\[
|g|_{\Lip(s,\Lp{q}(\Omega))} = \sup_{h > 0} h^{-s} \omega(g,h)_q < +\infty .
\]
\end{mydef}


Lipschitz spaces contain functions with, roughly speaking, $s$ "derivatives" in $\Lp{q}(\Omega)$. We also note that we restrict ourselves to $0<s \leq 1$ since for $s \geq 1$, the only functions in $\Lip(s,\Lp{q}(\Omega))$ are constants by \cite[Chapter 2, Proposition 7.1]{devore1993constructive}. Lipschitz spaces allow for a broad range of functions and namely, $\Lip(1,\Lp{1}(\Omega))$ contains functions of bounded variation, for example \cite[Chapter 2, Lemma 9.2]{devore1993constructive}.

We will be considering the error rate between a function and its piecewise constant approximation. The results presented below are part of the broader literature on approximation theory and in particular, spline approximations (we refer to \cite{devore1993constructive} for a review of such topics). We begin by defining operators that will allow us to connect discrete and continuum spaces.

Let $\Omega \subseteq \bbR^d$ be some bounded set and let $\Pi = \{\pi_i\}_{i=1}^{\vert \Pi \vert}$ be a disjoint partition of $\Omega$ with cardinality $\vert \Pi \vert$. We define the projection operator $\cP_{\Pi}:\Lp{1}(\Omega) \mapsto \bbR^{\vert \Pi \vert}$ and the injection operator $\cI_\Pi:\mathbb{R}^{\vert \Pi \vert} \mapsto \Lp{1}(\Omega)$ as
\[
(\cP_\Pi u)_i =  \frac{1}{\lambda_x(\pi_i)} \int_{\pi_i} u(x) \, \dd x \quad \text{and} \quad (\cI_\Pi \bar{u} )(x) = \sum_{i=1}^{\vert \Pi \vert} \bar{u}_i \chi_{\pi_i}(x) 
\]
respectively for $u \in \Lp{1}(\Omega)$ and $i =1, \dots, \vert \Pi \vert$, $\bar{u} \in \mathbb{R}^{\vert \Pi \vert}$ and $x \in \Omega$. When using the partition $\{\pi_i \times \pi_j \}_{i,j=1}^{\vert \Pi \vert}$ of $\Omega \times \Omega$, by an abuse of notation, we have 
\[
(\cP_\Pi v)_{i,j} = \frac{1}{\lambda_x(\pi_i)\lambda_x(\pi_j)}  \int_{\pi_i} \int_{\pi_j} v(x,y) \, \dd y \dd x \quad \text{and} \quad (\cI_\Pi \bar{v})(x,y) = \sum_{i,j=1}^{\vert \Pi \vert} \bar{v}_{i,j} \chi_{\pi_i}(x) \chi_{\pi_j}(y)
\]
for $v \in \Lp{1}(\Omega \times \Omega)$, $\bar{v} \in \bbR^{\vert \Pi \vert \times \vert \Pi \vert}$ and $(x,y) \in \Omega \times \Omega$.

For a function $u$, in order to obtain quantitative rates for $\Vert u - \cI_\Pi \cP_\Pi u \Vert_{\Lp{p}}$ the choice of $\Pi$ and the regularity of $u$ are essential. We now describe one construction of a partition $\Pi$ but refer to \cite{Davydov} and references therein for more examples. In order to simplify the discussion, we consider $\Omega = (0,1)^d$ although some of the results described below hold in more general cases.

For $n \in \bbN$, define $[n] = {1,\dots,n}$ and let ${\bi} = (i_1,i_2,\ldots,i_d) \in [n]^d$ be a multi-index. We partition $\Omega$ into $n^d$ hypercubes with sides of length $n^{-1}$ and denote this partition by $\Pi_{\uni,n} = \{\Omega_{n,\bi}\}_{\bi \in [n]^d}$. We have the following approximation lemma whose proof can be found in Section \ref{sec:supplementary:background}.

\begin{lemma}[Approximation on the uniform partition] \label{lem:background:piecewise:equiRates}
Let $\Omega = (0,1)^d$, $g\in\Lip(s,\Lp{q}(\Omega))$ and assume the partition $\Pi_{\uni,n}$ on $\Omega$. For $0<s\leq 1$ and $1 \leq q \leq \infty$, we have
\[
\|g-\cI_{\Pi_{\uni,n}}\cP_{\Pi_{\uni,n}}g\|_{\Lp{q}} \leq C |g|_{\Lip(s,\Lp{q}(\Omega))} n^{-s},
\]
for some constant $C > 0$ depending only on $d$. 

Furthermore, if $g$ is in
$\mathrm{C}^{0,\alpha}$ for $0 < \alpha \leq 1$ and $\eps > 0$, then:
        \[
        \|g(\cdot/\eps)-\cI_{\Pi_{\uni,n}}\cP_{\Pi_{\uni,n}}g(\cdot/\eps)\|_{\Lp{q}} \leq C \eps^{-\alpha} n^{-\alpha}
        \]
        for some $C > 0$ dependant on $d$ and $\Omega$.
\end{lemma}

\begin{remark}[Rates for functions on the product space]
We note that the conclusions of Lemma \ref{lem:background:piecewise:equiRates} 
also apply to functions $g$ defined on $\Omega \times \Omega$.
\end{remark}

\subsection{Random graph models} \label{subsec:randomGraphs}

Whenever we deal with random graph models, we will assume that we have the uniform partition $\Pi_{\uni,n} = \{\Omega_{n,i}\}_{i=1}^n$ of $(0,1)$ and shall therefore write $\cP_n = \cP_{\Pi_{\uni,n}}$ and $\cI_n = \cI_{\Pi_{\uni,n}}$.

In some applications, data is represented in the form of an undirected graph. One approach to understanding the underlying structure is to analyse the convergence properties of the graph as the number of vertices goes to infinity. We therefore consider graphs through their weight function defined on $[0,1]^2$: for a graph $G$ with vertices labelled by $[n]$ and weight matrix $\{\bar{K}_{n,ij}\}_{i,j=1}^n$ with $\bar{K}_{n,ii} = 0$ and $\bar{K}_{n,ji} = \bar{K}_{n,ij} \geq 0$ for $i\neq j$, for $(x,y) \in \Omega_{n,i} \times \Omega_{n,j}$, we define $\tilde{K}_n(x,y) = \bar{K}_{n,ij}$. The objective is now to analyze the limit of the step-functions $\tilde{K}_n$ as $n \to \infty$: this is well-known to be a graphon, i.e. a symmetric kernel function in $\Lp{1}([0,1]^2)$ (see \cite{borgs2014an} and references therein). Given a sequence of graphs from a certain graph model, i.e. a sequence of graphs for which there is a systematic way to  determine the graph weights $\{\bar{K}_{n,ij}\}_{i,j=1}^n$, we can establish the convergence of the sequence to the corresponding graphon in an appropriate metric. 
In the data-centric approach, one tries to fit graph models to data therewith estimating the underlying graphon
(see for example \cite{pmid19934050}, \cite{wolfe2013nonparametric}): the intuition here is that graph sequences that have related graphons as limiting points should share similarities.

In this paper, we consider a general sparse random graph model originally introduced and studied by \cite{borgs2014an}.

\begin{mydef}[Random graph models]\label{def:randomGraphModels}
For $n \in \bbN$, let $\rho_n > 0$ and $\bar{K} \in \bbR^{n \times n}$ be a symmetric matrix with non-negative entries.
Assume $\rho_n\bar{K}_{ij}\leq 1$ for all $i,j\in [n]$ and $\bar{K}_{ii}=0$ for all $i\in [n]$.
Let
\[ \l\bar{\Lambda}_{n}\r_{ij} = \lb \begin{array}{ll} \frac{1}{\rho_n} & \text{with probability } \rho_n\bar{K}_{ij} \\ 0 & \text{else.} \end{array} \rd \]
We define the random graph $G(n,\bar{\Lambda}_n)$ to be the random graph with vertices $[n]$ and weight matrix $\bar{\Lambda}_n$. 
\end{mydef}

As a simple example, consider the case where $\rho_n = 1$ and the entries of $\bar{K}$ are $\bar{K}_{ij} = p \in (0,1)$ for $1\leq i,j \leq n$. Then, vertices $i$ and $j$ are connected with probability $p$ and $G(n,\bar{\Lambda}_n)$ is the Erd\H{o}s-R\'enyi graph model $G(n,p)$.

Note that in previous papers (for example \cite[Section 3.2]{ElBouchairi}), the weight matrix only has 0-1 weights but, when the graphs are used in evolution problems, the quantities are normalized by $\rho_n^{-1}$ which is related to the average degree of the graph. Altering the definition as we did in Definition \ref{def:randomGraphModels} leads to more a straight-forward and clearer problem setting.

In our paper, we will be considering the limit to a local problem and are therefore less concerned with statements related to the convergence of our random graph models. We refer to \cite{borgs2014an} and \cite[Proposition 3.1]{ElBouchairi} for more details on the topic however.

\begin{remark}[Extension of graph models]
We mention at the beginning of Section \ref{subsec:randomGraphs} that we consider the uniform partition $\Pi_{\uni,n} = \{\Omega_{n,i}\}_{i=1}^n$ of $ \Omega = (0,1)$ in the setting of random graph models. The reason for choosing $d = 1$ follows from a standard data-centric approach (see \cite[Section 2.7]{borgs2014an}): to model the relationships between $n$ data points, given a graphon $\tilde{K}$, we usually sample $n$ points $x_i \iid \text{Unif}([0,1])$ and set graph weights $\bar{K}_{ij} = \tilde{K}(x_i,x_j)$ for $i \neq j$ and $\bar{K}_{ii} = 0$. We then create the appropriate graph models (for example as in Definition \ref{def:randomGraphModels}) with the goal of showing that the step-functions associated to these models converge to the initial graphon $\tilde{K}$. We note that the first part of the procedure, i.e. setting the weights $\bar{K}_{ij}$, could work equally well with a function $\tilde{K}:[0,1]^{2d} \mapsto [0,\infty)$ with $d \geq 1$; however, the second part of the procedure only makes sense for $d = 1$ since the convergence of step-functions associated to graphs 
supposes that the limiting function is also defined on $[0,1]^2$ (and is therefore a graphon).

As we are not concerned with showing the convergence of our graph models to graphons, we could therefore formulate everything for $d \geq 1$. This would however not be consistent with the standard data-centric approach where random graph models are generated from graphons. An interesting research direction would be to generate random hypergraph models (a hypergraph is a graph that also includes non-pairwise relationships, for example faces connecting three vertices instead of pairwise edges) from functions $\tilde{K}:[0,1]^{2d} \to [0,\infty)$, and show the convergence of step-functions associated to these hypergraphs
to $\tilde{K}$. Our results in Section \ref{subsec:application} should naturally extend to such models.
\end{remark}

\section{Main results} \label{sec:main}

\subsection{Assumptions}

In this section, we enumerate all the assumptions on the space, operators, kernels and length-scale that we will be using throughout the paper.

\begin{assumptions}[Assumptions on the space]\leavevmode

\begin{enumerate}[label=\textbf{S.\arabic*}]
    \setlength\itemsep{-1mm}
\item The space $\Omega$ is a bounded open subset of $\bbR^d$.\label{ass:main:assumptions:S1}
\item The space $\Omega$ has Lipschitz boundary. \label{ass:main:assumptions:S2}
\end{enumerate}
\end{assumptions}

\begin{assumptions}[Assumptions on the operator $\cA$]\leavevmode
\begin{enumerate}[label=\textbf{O.\arabic*}]
    \setlength\itemsep{-1mm}
\item The operator $\cA:\Lp{2}(\Omega)\mapsto\Lp{2}(\Omega)$ is bounded and linear. We write $C_{\mathrm{op}} = \Vert \cA \Vert_{\mathrm{op}}$, the operator norm of $\cA$. \label{ass:main:assumptions:O1}
\item The operator $\cA^* \cA$ is order-preserving: for all $u\in \Lp{2}(\Omega)$ and $x \in \Omega$, $\signFunc(\cA^*\cA u(x)) = \signFunc(u(x))$ where $\signFunc$ is the sign function. \label{ass:main:assumptions:O2}
\item The operator $\cA^*\cA$ satisfies the following: $\cA^*\cA (0) = 0$. \label{ass:main:assumptions:O3}
\item For $n \in \bbN$ and a partition $\Pi_n$, there exists a positive semi-definite linear operator $\bar{G}_n:\bbR^{\vert \Pi_n \vert} \mapsto \bbR^{\vert \Pi_n \vert}$ such that $\cI_n(\bar{G}_n(\bar{u})) = \cA^*\cA(\cI_n\bar{u}).$
\label{ass:main:assumptions:O4}
\end{enumerate}
\end{assumptions}

From \cite[Theorem VI.2 and Theorem VI.3]{reed1981functional}), if Assumption \ref{ass:main:assumptions:O1} holds, it is easy to see that
    $\Vert \cA^* \cA v \Vert_{\Lp{2}} \leq C_{\mathrm{op}}^2 \Vert v \Vert_{\Lp{2}}$
for all $v \in \Lp{2}(\Omega)$.

\begin{assumptions}[Assumptions on the kernel $K$]\leavevmode
\begin{enumerate}[label=\textbf{K.\arabic*}]
    \setlength\itemsep{-1mm}
\item The kernel $K:[0,\infty) \mapsto [0,\infty)$ is bounded in $\Lp{\infty}$. \label{ass:main:assumptions:K1}
\item The kernel $K:[0,\infty) \mapsto [0,\infty)$ 
has $\support(K) = [0,1]$. \label{ass:main:assumptions:K2}
\end{enumerate}
\end{assumptions}

\begin{assumptions}[Assumptions on regularity]\leavevmode
\begin{enumerate}[label=\textbf{R.\arabic*}]
    \setlength\itemsep{-1mm}
\item The function $u \in \Ck{0}([0,T];\Lp{p}(\Omega))$ solving~\eqref{eq:main:notation:nonlocal:nonlocalProblem} satisfies $u(t) \in \Wkp{1}{p}(\Omega)$ for $0 < t <T$. \label{ass:main:assumptions:R1}
\item For $h \geq 1$, $s>3+d/p$  and $r>2+d/p$, the function $u \in \Lp{p}(0,T;\Wkp{1}{p}(\Omega)) \cap \Lp{2}(0,T;\Lp{2}(\Omega))$ solving~\eqref{eq:main:notation:localProblem:localProblem} satisfies $u\in \Lp{h}(0,T;\Wkp{s}{p}(\Omega)) \cap \Lp{\infty}(0,T;\Wkp{r}{p}(\Omega))$. \label{ass:main:assumptions:R2}
\item 
For $h \geq 1$ and  $s > 3 + d/p$, there exists $\tilde{C}$ independent of $T$ such that the function \[
u \in \Lp{p}(0,T;\Wkp{1}{p}(\Omega)) \cap \Lp{2}(0,T;\Lp{2}(\Omega))\]
solving~\eqref{eq:main:notation:localProblem:localProblem} satisfies
 \[
        u\in \Lp{h}(0,T;\Wkp{s}{p}(\Omega)) \quad \text{and} \quad \max \left\{ \sup_{t \in (0,T)} \Vert \nabla u \Vert_{\Lp{\infty}}, \sup_{t \in (0,T)} \Vert \nabla^2 u \Vert_{\Lp{\infty}}\right\} \leq \tilde{C}
\]
for all $T > 0$.\label{ass:main:assumptions:R3}

\end{enumerate}
\end{assumptions}

\begin{assumptions}[Assumptions on the length-scale]\leavevmode
\begin{enumerate}[label=\textbf{L.\arabic*}]
    \vspace{-2mm}
    \setlength\itemsep{-1mm}
\item The length scale $\eps=\eps_n$ is positive and converges to 0, i.e. $0<\eps_n \to 0$ as $n \to \infty$.\label{ass:main:assumptions:L1}
\end{enumerate}
\end{assumptions}

\subsection{Setting}

\subsubsection{Discrete problem}

Given a partition $\Pi$ of $\Omega$ and a symmetric $\bar{K} \in \bbR^{\vert \Pi \vert \times \vert \Pi \vert}$, we define the discrete nonlocal $p$-Laplacian operator for $\bar{u} \in \bbR^{\vert \Pi \vert}$ and $1 \leq i \leq \vert \Pi \vert$ as follows:
\[
(\Delta_{p,\Pi}^{\bar{K}} \bar{u})_i = - \sum_{j=1}^{\vert \Pi \vert} \lambda_x(\pi_j) \bar{K}_{i,j} \vert (\bar{u})_j - (\bar{u})_i \vert ^{p-2}((\bar{u})_j - (\bar{u})_i).   
\]

Let $\bar{f} \in \bbR^{\vert \Pi \vert}$ and $\bar{u}_0 \in \bbR^{\vert \Pi \vert}$. We will consider the fully discrete nonlocal evolution problem
\begin{equation} \label{eq:main:notation:discrete:nonlocalProblemFully}
    \begin{cases}
    \frac{\bar{u}^k - \bar{u}^{k-1}}{ \tau^{k-1} } + \mu \Delta_{p,\Pi}^{\bar{K}} \bar{u}^k + \bar{G}(\bar{u}^k) =  \bar{f}, & \text{for} \quad 1 \leq k \leq N  \\
    \bar{u}(0) = \bar{u}_0
    \end{cases}
\end{equation}
for some $\mu > 0$, a positive sequence $\{\tau^k\}_{k=1}^{N-1}$ (with $\sum_{k=0}^{N-1}\tau^k = T$) 
and linear operator $\bar{G}:\bbR^{\vert \Pi \vert} \to \bbR^{\vert \Pi \vert}$. 
We also define $\tau = \max_{k=1,\dots,N} \tau^k$ and will write $\tau_n$ for $\tau$ when 
our sequence $\{\tau^k\}_{k=1}^{N-1}$ will be indexed by $n$.
We say that $\bar{u}$ solves \eqref{eq:main:notation:discrete:nonlocalProblemFully} with parameters $\bar{K}$, $\bar{f}$ and $\bar{u}_0$.

Later we choose $\tau^k = t^{k+1}-t^k$ for a discretisation $0=t^0<t^1<\dots<t^N=T$ of $[0,T]$ and we will need the following two quantities. First, we define the time interpolated version of the injected vectors $\{\bar{u}^{k}\}_{k=1}^N$:
\[
\uTimeInter(t,x) = \frac{t^k - t}{\tau^{k-1}} \l \cI_{\Pi_n} \bar{u}^{k-1}\r(x) + \frac{t - t^{k-1}}{\tau^{k-1}} \l \cI_{\Pi_n} \bar{u}^k\r (x) \text{ for } (t,x) \in (t^{k-1},t^k] \times \Omega
\]
and $\uTimeInter(0,x) = \cI_{\Pi_n} \bar{u}_0$.
Second, we define the time injected version of the injected vectors $\{\bar{u}^{k}\}_{k=1}^N$:
\[
\uTimeInject(t,x) = \sum_{i=1}^N \l \cI_{\Pi_n} \bar{u}^{k} \r (x)\chi_{(t^{k-1},t^k]}(t).
\]

Most of our results will holds for any partition $\Pi$ but, keeping in mind that we are ultimately interested in convergence results, meaning the case where $\max_{j \leq \vert \Pi \vert} \lambda_x(\pi_j) \to 0$, we will index our partition with a parameter $n$. We denote the latter by $\Pi_n = \{\pi_j^n\}_{j=1}^{\vert \Pi_n \vert}$. As an example, one can consider $\Pi_{\uni,n}$. For ease of notation we will write $\Delta^{\bar{K}}_{p,n} = \Delta_{p,\Pi_n}^{\bar{K}}$, $\cI_n = \cI_{\Pi_n}, \cP_n = \cP_{\Pi_n}$ and similarly for other quantities.

\subsubsection{Nonlocal problem} 
Given a function a kernel function $K:[0,\infty) \to [0,\infty)$, we define the nonlocal $p$-Laplacian operator $\Delta_p^K$ for a function $u\in \Lp{1}(\Omega)$ and $x \in \Omega$ as follows:
\[
\Delta_p^K u(x) = - \int_\Omega K(\vert x - y \vert) \vert u(y) - u(x) \vert^{p-2} (u(y) - u(x)) \, \dd y.
\]
For a kernel $K$, we define the function $\Tilde{K}:\Omega \times \Omega \mapsto [0,\infty)$ by $\Tilde{K}(x,y) = K(\vert x-y \vert)$. For a general function $v:\Omega \times \Omega \mapsto [0,\infty)$, we naturally have:
\[
\Delta_p^v u(x) = - \int_\Omega v(x,y) \vert u(y) - u(x) \vert^{p-2} (u(y) - u(x)) \, \dd y.
\]
Furthermore, for $f \in \Lp{2}(\Omega)$, $p \geq 2$ and $q$ such that $p^{-1} + q^{-1} = 1$, we define the following evolution operator $\cE^K_{\cA,f}:\Lp{p}(\Omega) \mapsto \Lp{q}(\Omega)$ for $\mu > 0$, a function $u \in \Lp{p}(\Omega)$ and some linear operator $\cA:\Lp{2}(\Omega) \mapsto \Lp{2}(\Omega)$:
\[ 
\cE^K_{\cA,f}(u) = \mu \Delta_p^Ku + \cA^* \cA u - f.
\]
The well-posedness of this evolution operator will be discussed in Theorem \ref{thm:proofs:wellPosedness:nonlocalProblem:existenceUniqueness}.

We will consider the following nonlocal evolution problem:
\begin{equation} \label{eq:main:notation:nonlocal:nonlocalProblem}
    \begin{cases}
    \frac{\partial}{\partial t }u(t,x) + \cE^K_{\cA,f}(u(t,x)) = 0 & \text{on} \quad (0,T) \times \Omega, \\
    u(0,x) = u_0(x).
    \end{cases}
\end{equation}
We say that $u$ solves \eqref{eq:main:notation:nonlocal:nonlocalProblem} with parameters $K$, $f$ and $u_0$. In particular, in order to link the above problem with \eqref{eq:intro:regularizationProblem}, we will be interested in \eqref{eq:main:notation:nonlocal:nonlocalProblem} when $\cE^K_{\cA,f}(u) = \cE^K_{\cA,\cA^*\ell}(u)$ for some $\ell \in \Lp{2}(\Omega)$.

We are interested in the following solution which is just of Definition \ref{def:main:notation:nonlocal:strongSolutionCP} using the evolution operator appearing in \eqref{eq:main:notation:nonlocal:nonlocalProblem}.

\begin{mydef}[Nonlocal problem solution] \label{def:main:notation:nonlocal:strongSolution}
Assume that Assumptions \ref{ass:main:assumptions:S1}, \ref{ass:main:assumptions:O1}, \ref{ass:main:assumptions:O2} and \ref{ass:main:assumptions:K1} hold.
For $p \geq 2$, $T > 0$, $\mu > 0$, $f \in \Lp{p}(\Omega)$, $u_0 \in \Lp{p}(\Omega)$, a function $u(t,x)$ is a solution to the nonlocal problem \eqref{eq:main:notation:nonlocal:nonlocalProblem} if $u(t,x) \in C([0,T];\Lp{p}(\Omega)) \cap W_{\mathrm{loc}}^{1,1}((0,T);\Lp{p}(\Omega))$, $u(0,x) = u_0(x)$ $\lambda_x$-a.e. on $\Omega$ and $\lambda_t$-a.e.:
\[
\frac{\partial}{\partial t }u(t,x) + \mu \Delta_p^K u(t,x) + \cA^*\cA u(t,x) = f(x) \quad \text{$\lambda_x$-a.e..}
\]
\end{mydef}

If we are further given a positive sequence $\{\eps_n\}_{n=1}^\infty$, we will write 
$K_{\eps_n} = \frac{2}{c(p,d)\eps_n^{d+p}}K(\cdot / \eps_n)$ 
where

\begin{equation} \label{eq:main:notation:nonlocal:cpd}
    c(p,d) = \int_{\bbR^d} K(\vert x \vert) \vert x_d \vert^p \, \dd x.
\end{equation}

\commentOut{
By \cite[Lemma 5]{https://doi.org/10.48550/arxiv.1711.10144}, with $\Gamma(t) = \int_{0}^\infty e^{-x}x^{t-1} \, \dd x$, we have 
\[
\sigma(p,d) = \frac{\Gamma\l \frac{1}{2} \r^{d-2} \Gamma \l \frac{p-1}{2} \r \Gamma \l \frac{3}{2} \r}{\Gamma \l \frac{d+p}{2} \r} \int_0^\infty K(\sqrt{t}) t^{\frac{d+p}{2}-1} \, \dd t = \frac{\Gamma\l \frac{1}{2} \r^{d-2} \Gamma \l \frac{p-1}{2} \r \Gamma \l \frac{3}{2} \r}{\Gamma \l \frac{d+p}{2} \r} \int_0^\infty K(r) r^{p + d - 1} \, \dd r
\]
and 
\[
\c(p,d) = \frac{\Gamma\l \frac{1}{2} \r^{d-1} \Gamma \l \frac{p+1}{2} \r }{\Gamma \l \frac{d+p}{2} \r} \int_0^\infty K(\sqrt{t}) t^{\frac{d+p}{2}-1} \, \dd t = \frac{\Gamma\l \frac{1}{2} \r^{d-1} \Gamma \l \frac{p+1}{2} \r }{\Gamma \l \frac{d+p}{2} \r} \int_0^\infty K(r) r^{p + d - 1} \, \dd r
\]
which implies that if $K$ has finite $p$-moment, both constants are finite.
}

\subsubsection{Local problem}

The local $p$-Laplacian operator is defined as 
\[
\Delta_p u = - \diver(\vert \nabla u \vert^{p-2} \nabla u ).
\]

For $p \geq 2$, $\mu >0$, $\ell \in \Lp{2}(\Omega)$ and $u_0 \in \Lp{p}(\Omega)$, we will consider the following local evolution problem: 
\begin{equation} \label{eq:main:notation:localProblem:localProblem}
    \begin{cases}
    \frac{\partial}{\partial t }u(t,x) + \mu \Delta_p u(t,x) + \cA^*\cA u(t,x) =  \cA^*\ell(x), & \text{on} \quad (0,T) \times \Omega \\
    \vert \nabla u(t,x) \vert^{p-2} \nabla u(t,x) \cdot \overrightarrow{n} = 0, &\text{on} \quad (0,T)\times \partial\Omega \\
    u(0,x) = u_0(x)
    \end{cases}
\end{equation}
for some linear operator $\cA:\Lp{2}(\Omega) \mapsto \Lp{2}(\Omega)$.

The following definition is inspired by the results of \cite[Chapter 2, Theorem 1.4]{lions1969quelques}.

\begin{mydef}[Local weak solution] \label{def:main:notation:localProblem:weakSolution}
Assume that Assumptions \ref{ass:main:assumptions:S1} and \ref{ass:main:assumptions:O1} hold. For $p \geq 2$, $T > 0$, $\ell \in \Lp{2}(\Omega)$, $u_0 \in \Lp{p}(\Omega)$, a function $u(t,x)$ is a weak solution to \eqref{eq:main:notation:localProblem:localProblem} if $u(t,x) \in \Lp{p}(0,T;\Wkp{1}{p}(\Omega)) \cap \Lp{2}(0,T;\Lp{2}(\Omega))$ and $u \in \Lp{\infty}(0,T;\Lp{2}(\Omega))$ with $u(0,\cdot) = u_0$ $\lambda_x$-a.e. on $\Omega$ and if $\lambda_t$-a.e.: 
\[
    \int_\Omega \frac{\partial}{\partial t} u(t,x) \zeta(x) \, \dd x + \mu \int_\Omega \vert \nabla u(t,x) \vert^{p-2} \nabla u(t,x) \cdot \nabla \zeta(x) \, \dd x + \int_\Omega \cA^* \cA u(t,x) \zeta(x) \, \dd x = \int_\Omega \cA^* \ell(x) \zeta(x) \, \dd x 
\]
for all $\zeta \in \Wkp{1}{p}(\Omega)$.
\end{mydef}

\subsubsection{Random graphs problem}

Given a graph $G(n,\bar{\Lambda}_n)$, 
$\bar{u}_0 \in \bbR^n$ and $\bar{f} \in \bbR^n$, we can also consider the following evolution problem:
\begin{equation} \label{eq:main:notation:random:evolutionProblem}
    \begin{cases}
    \frac{\bar{u}^k - \bar{u}^{k-1}}{ \tau^{k-1} } + \mu \Delta_{p,n}^{\bar{\Lambda}_n} \bar{u}^k + \bar{G}(\bar{u}^k) =  \bar{f}, & \text{for} \quad 1 \leq k \leq N  \\
    \bar{u}(0) = \bar{u}_0
    \end{cases}
\end{equation}
for some $\mu > 0$, linear operator $\bar{G}:\bbR^{n} \to \bbR^{n}$, partition $0 = t^0 < t^1 < \dots < t^N = T$ and where $\tau^{k-1} = t^k - t^{k-1}$. We say that $\bar{u}$ solves \eqref{eq:main:notation:discrete:nonlocalProblemFully} with parameters $\bar{\Lambda}_n$, $\bar{f}$ and $\bar{u}_0$.

\commentOut{
\begin{assumptions}
Assumptions on the solution of the nonlocal evolution problem \eqref{eq:main:notation:nonlocal:nonlocalProblem} $u_K$.
\begin{enumerate}[label=\textbf{U.\arabic*}]
\item The solution $u_K$ to \eqref{eq:main:notation:nonlocal:nonlocalProblem} is in $C([0,T];\Wkp{1}{p}(\Omega)) \cap W_{\text{loc}}^{1,1}((0,T);\Wkp{1}{p}(\Omega))$. \label{ass:main:assumptions:U1}
\end{enumerate}
\end{assumptions}
}

\subsection{Main results} \label{sec:main:main}

\subsubsection{Well-posedness of the nonlocal continuum gradient flow}

In order to prove well-posedness, we loosely follow the strategy in \cite{ANDREU2008201}. Our main contribution lies in the verification of the range condition in Proposition \ref{prop:proofs:wellPosedness:nonlocalProblem:completeAccretivityRangeCondition}: we propose an alternative proof for a generalization of the commonly used \cite[Theorem 2.4]{ANDREU2008201} based on principles related to $\Gamma$-convergence (see \cite{gammaConvergence} or \cite{maso2012introduction}). In particular, showing the range condition boils down to solving a PDE which is unsolvable by direct methods. We therefore modify the latter by adding a term that will make the operators involved coercive. We then use a compactness argument to show that the solutions to the PDE approximations converge to a limiting function which solves the initial PDE by the $\liminf$-inequality of $\Gamma$-convergence.

\begin{proposition}[Complete accretivity and range condition] \label{prop:proofs:wellPosedness:nonlocalProblem:completeAccretivityRangeCondition}

Assume that Assumptions \ref{ass:main:assumptions:S1}, \ref{ass:main:assumptions:O1}, \ref{ass:main:assumptions:O2} and \ref{ass:main:assumptions:K1} hold. Let $p \geq 2$ and assume that $f \in \Lp{p}(\Omega)$. Then, the evolution operator $\cE^K_{\cA,f}$ is completely accretive and satisfies the range condition $\Lp{p}(\Omega) \subseteq \range(\Id + \lambda \cE^K_{\cA,f})$ for $\lambda > 0$.
\end{proposition}

The proof of the proposition is given in Section~\ref{subsubsec:wellPosedness:nonLocalProblem}.
From the proposition one easily deduces the existence of solutions to the nonlocal problem.
The proof of the theorem is also given in Section~\ref{subsubsec:wellPosedness:nonLocalProblem}.

\begin{theorem}[Existence and uniqueness of a solution for the nonlocal problem] \label{thm:proofs:wellPosedness:nonlocalProblem:existenceUniqueness}
 
Assume that Assumptions \ref{ass:main:assumptions:S1}, \ref{ass:main:assumptions:O1}, \ref{ass:main:assumptions:O2} and \ref{ass:main:assumptions:K1} hold. Let $p \geq 2$, $T>0$, $u_0 \in \Lp{p}(\Omega)$ and assume that $f \in \Lp{p}(\Omega)$. Then, there exists a unique solution $u$ as in Definition \ref{def:main:notation:nonlocal:strongSolution} to the evolution problem \eqref{eq:main:notation:nonlocal:nonlocalProblem} with the operator $\cE^K_{\cA,f}$ and initial value $u_0$.

Furthermore, if $v$ is a solution as above solving \eqref{eq:main:notation:nonlocal:nonlocalProblem} with the operator $\cE^K_{\cA,g}$ and initial value $v_0$, we have
\begin{equation} \label{eq:proofs:wellPosedness:nonlocalProblem:existenceUniqueness:contraction}
    \Vert u(t) - v(t) \Vert_{\Lp{r}} \leq \Vert u_0 - v_0 \Vert_{\Lp{r}} + t \Vert f - g \Vert_{\Lp{r}}
\end{equation}
for $1 \leq r \leq \infty $ and $0 \leq t \leq T$.

In addition, if Assumption \ref{ass:main:assumptions:O3} holds, then 
\begin{equation} \label{eq:proofs:wellPosedness:nonlocalProblem:existenceUniqueness:bound}
    \Vert u(t) \Vert_{\Lp{r}} \leq \Vert u_0 \Vert_{\Lp{r}} + t \Vert f \Vert_{\Lp{r}}
\end{equation}
for $1 \leq r \leq \infty $ and $0 \leq t \leq T$.

\end{theorem}

\subsubsection{Rates of convergence}

The rates are established by combining two intermediate results. On one hand, we obtain nonlocal-to-local continuum rates by switching from finite-differences to derivatives which allows one to prove the convergence of the nonlocal $p$-Laplacian operator to the local one. On the other hand, the central point of the proofs is to leverage the fact that the injected discrete gradient flow solution solves a nonlocal continuum gradient flow with particular parameters. We then rely on the contraction properties \eqref{eq:proofs:wellPosedness:nonlocalProblem:existenceUniqueness:contraction} and \eqref{eq:proofs:wellPosedness:nonlocalProblem:existenceUniqueness:bound} to express the rates of convergence of discrete-to-continuum solutions in the nonlocal setting in terms of the discretization error of the initial condition, data and kernel functions.

The next result precisely describes the interplay between our space-localization parameter $\eps_n$ and our time-discretization parameter $\tau_n$ in order to ensure convergence of the discrete solution to the solution of \eqref{eq:main:notation:localProblem:localProblem}.

\begin{corollary}[Discrete-to-continuum local rates] \label{cor:proofs:rates:discreteNonlocalContinuumLocal:simplified}

Assume that assumptions \ref{ass:main:assumptions:S1}, \ref{ass:main:assumptions:O1}, \ref{ass:main:assumptions:O2}, \ref{ass:main:assumptions:O3}, \ref{ass:main:assumptions:O4}, \ref{ass:main:assumptions:K1} and \ref{ass:main:assumptions:L1} hold. Let $p \geq 2$, $\mu > 0$, $T>0$, $u_0 \in \Lp{p}(\Omega)$, $\ell \in \Lp{2}(\Omega)$, $\Omega'$ be compactly contained in $\Omega$ and assume that $\cA^* \ell \in \Lp{p}(\Omega)$.
Furthermore, let $n \in \bbN$ and define $\bar{K}_{\eps_n} = \cP_n \tilde{K}_{\eps_n}$, $\bar{f} = \cP_n \cA^*\ell $, $\bar{u}_0 = \cP_n u_0$.

Then, for any partition $0 = t^0 < t^1 < \dots < t^N = T$, there exists a sequence $\{\bar{u}_n^k\}_{k=0}^N$ satisfying \eqref{eq:main:notation:discrete:nonlocalProblemFully} with $\bar{K}_{\eps_n}$, $\bar{f}$, $\bar{u}_0$ and $\bar{G}_n$ chosen as above, a solution $u_{\eps_n}$ to \eqref{eq:main:notation:nonlocal:nonlocalProblem} with kernel $K_{\eps_n}$ and a solution $u$ to \eqref{eq:main:notation:localProblem:localProblem}.

In addition, assume that Assumptions \ref{ass:main:assumptions:S2} and \ref{ass:main:assumptions:K2} hold, $p \geq 3$, that we are using the partition $\Pi_{\uni,n}$, that $u_{\eps_n}$ satisfies Assumption \ref{ass:main:assumptions:R1} for all $T >0$ and that $u$ satisfies Assumption \ref{ass:main:assumptions:R3}.
For some $1 \leq q_1 < \infty$ and $ 1 \leq q_2 < \infty$ and $0 < \alpha_i \leq 1$ for $1 \leq i \leq 3$, assume furthermore that $u_0 \in \Lip(\alpha_1,\Lp{q_1}(\Omega)) \cap \Lp{\infty}(\Omega)$, $\cA^*\ell \in \Lip(\alpha_2,\Lp{q_2}(\Omega)) \cap \Lp{\infty}(\Omega)$ and $K \in \text{C}^{0,\alpha_3}(\Omega)$. Then, if for some $\kappa > 0$ we set $T(n) = \l \frac{1}{1 +  C^4_{\mathrm{op}} } \r \log(\eps_n^{-\kappa})$ and assume that 
\begin{equation} \label{eq:main:taun}
\tau_n \ll \frac{\eps_n^{2(d+p) + \kappa }}{\log(\eps_n^{-\kappa})^{(2p-3)}}
\end{equation}
as well as
\begin{align*}
\eps_n &\gg \max \Biggl\{ n^{-\alpha_1/\kappa} , n^{-\alpha_2/\kappa}, \ls \expW \l n^{\frac{\alpha_3}{\max\l 1 + (d+p+\alpha_3)/\kappa , p - 1 \r}}\r \rs^{-1/\kappa}\Biggr\},
\end{align*}
for $n$ large enough, we have:
\begin{align}
  \sup_{1 \leq k \leq N } \sup_{t \in (t^{k-1},t^k]} &\Vert \cI_n \bar{u}_{n}^k - u(t,\cdot) \Vert_{\Lp{2}(\Omega')} \leq C \Bigg( \eps_n\log(\eps_n^{-\kappa}) \notag \\
  &+ \eps_n^{-\kappa}\ls n^{-\alpha_1} + n^{-\alpha_2} + \frac{\log(\eps_n^{-\kappa})^{(p-1)}}{\eps_n^{d+p + \alpha_3} n^{\alpha_3}} + \tau_n \frac{\log(\eps_n^{-\kappa})^{2p-3}}{\eps_n^{2(d+p)}} \rs  \Biggr)  \label{eq:proofs:rates:simplified:final}
\end{align}
for some $C > 0$ that might be dependent on $\Omega$ (and $d$), $u_0$ and $\cA^* \ell$ and the latter right-hand side tends to 0 as $n \to \infty$.

\end{corollary}

The proof of the corollary can be found in Section \ref{subsec:ratesDiscreteLocal}.

\begin{remark}[Asymptotics of $N$]
    It is a natural requirement that $ N \to \infty$ as $T(n) \to \infty$ as in Corollary \ref{cor:proofs:rates:discreteNonlocalContinuumLocal:simplified}. This is indeed the case (and also will be for Corollary \ref{cor:application:simplified} by the same argument) since by definition $N \geq T(n)/\tau_n$ and~\eqref{eq:main:taun}
    ensures that $\tau_n \to 0$ with $\eps_n \to 0$. 
\end{remark}

\begin{remark}[Parameters in Corollary \ref{cor:proofs:rates:discreteNonlocalContinuumLocal:simplified}]
    Despite their appearances, the conditions on $\eps_n$ and $\tau_n$ in Corollary \ref{cor:proofs:rates:discreteNonlocalContinuumLocal:simplified} are not constraining. Indeed, all parameters involved are chosen by the practitioner prior to the implementation of the numerical procedure. 
\end{remark}

The regularity requirements of both the nonlocal solution and the local solution are discussed in greater detail in Remarks \ref{rem:proofs;rates:continuumRates:increasedRegularity} and \ref{rem:proofs:rates:continuumRates:embeddings}. The nonlocal regularity assumption is linked to our approximation through finite differences. Indeed, the expression \eqref{eq:main:notation:nonlocal:nonlocalProblem} does not contain any differential term in the space component leading our solution to be in $\Lp{p}(\Omega)$ as opposed to $\Wkp{1}{p}(\Omega)$: when using finite elements, the approximations are in the same space as the solutions to original problem, which in our case is $\Wkp{1}{p}$ by Theorem \ref{thm:proofs:wellPosedness:localProblem:existenceUniqueness}. The local regularity assumption is a typical one and actually induces an interesting relationship between the regularity requirement and the regularization parameter $p$. 

The rates are formulated for $n$ large enough. This is discussed in Remark \ref{rem:proofs:rates:continuumRates:asymptotic} and is essentially dependent on the choice of our compactly contained set $\Omega' \subseteq \Omega$.

A simple additional step allows one to deduce the convergence of the discrete gradient flow to $u_\infty$. The proof is given in Section \ref{subsec:ratesDiscreteLocal}.
\begin{corollary}[Rates for $u_\infty$]
\label{cor:proofs:rates:discreteNonlocalContinuum:final}

Assume the same setting as in Corollary \ref{cor:proofs:rates:discreteNonlocalContinuumLocal:simplified} and furthermore that $\cA = \Id$. Then, for $n$ large enough, we have
\begin{align}
\Vert \cI_n \bar{u}_n^N - u_\infty \Vert_{\Lp{2}(\Omega')}& \leq C \Bigg( \eps_n\log(\eps_n^{-\kappa}) + \eps_n^{\kappa/4}(\cF(u_0)-\cF(u_\infty))^{1/2} \notag \\
  &+ \eps_n^{-\kappa}\ls n^{-\alpha_1} + n^{-\alpha_2} + \frac{\log(\eps_n^{-\kappa})^{(p-1)}}{\eps_n^{d+p + \alpha_3} n^{\alpha_3}} + \tau_n \frac{\log(\eps_n^{-\kappa})^{2p-3}}{\eps_n^{2(d+p)}} \rs  \Biggr)\label{eq:proofs:rates:simplified:fullRates}
\end{align}
for some $C > 0$ that might be dependent on $\Omega$ (and $d$), $u_0$ and $\cA^* \ell$ and the latter right-hand side tends to 0 as $n \to \infty$.
\end{corollary}

\begin{remark}[Curse of dimensionality]
    We note the presence of terms of the form $\eps_n^{-\gamma - \delta d}$ with $\gamma, \,\delta > 0$ in the right-hand sides of both \eqref{eq:proofs:rates:simplified:final} and \eqref{eq:proofs:rates:simplified:fullRates}. For fixed $n$ and $\eps_n < 1$, as $d \to \infty$, $\eps_n^{-\gamma - \delta d}$ will tend to infinity.
    While we can control this behaviours when $d$ is fixed, for high dimensions the rates deteriorate. 
\end{remark}

\subsubsection{Application to random graphs}

Everything discussed until now was determinisitic. In particular, the discretization procedure was based on partitioning the space into cells in a pre-defined way that would allow us to control the discretization error as described in Section \ref{subsec:approximations}. An alternative setting is the one of random graph models present in various applications. 

Obtaining results in the random graph setting is split in two steps: (1) prove rates of convergence between the discrete random and deterministic gradient flows and then (2) use the deterministic rates of Proposition \ref{prop:proofs:rates:NonFullyDiscreteLocal:L2rates} and Theorem \ref{thm:proofs:continuumRates:continuumNonlocalLocal}. 
The first part is conceptually similar to how the results in the nonlocal case are derived while adding the necessary probabilistic estimates. The main results are to be found
in Corollaries \ref{cor:application:simplified} and \ref{cor:application:final}. The proofs of the latter are given in Section \ref{subsec:application}.

\begin{corollary}[Discrete random-to-continuum local rates]
\label{cor:application:simplified}

Assume that Assumptions \ref{ass:main:assumptions:O1}, \ref{ass:main:assumptions:O2}, \ref{ass:main:assumptions:O3}, \ref{ass:main:assumptions:O4} and \ref{ass:main:assumptions:L1} hold. Let $p \geq 2$, $\mu > 0$, $T>0$, $u_0 \in \Lp{p}(\Omega)$, $\ell \in \Lp{2}(\Omega)$, $\Omega'$ be compactly contained in $\Omega$ and assume that $\cA^* \ell \in \Lp{p}(\Omega)$. Furthermore, let $n \in \bbN$ and define $\bar{K}_{\eps_n} = \cP_n \tilde{K}_{\eps_n}$, $\bar{f} = \cP_n \cA^*\ell $, $\bar{u}_0 = \cP_n u_0$.  We also suppose that $\rho_n$ is a positive sequence with $\rho_n \to 0$ and $\rho_n \ll \eps_n^{1+p}$. Let $\Lambda_n\in\bbR^{n\times n}$ be the weight matrix defined as in Definition~\ref{def:randomGraphModels} with $\bar{K}=\bar{K}_{\eps_n}$.

Then, for any partition $0 = t^0 < t^1 < \dots < t^N = T$, there exists a sequence $\{\bar{u}_n^k\}_{k=0}^N$ solving \eqref{eq:main:notation:random:evolutionProblem} with parameters $\bar{\Lambda}_n$, $\bar{f}$ and $\bar{u}_0$, a solution $u_{\eps_n}$ to \eqref{eq:main:notation:nonlocal:nonlocalProblem} with kernel $K_{\eps_n}$ and a solution $u$ to \eqref{eq:main:notation:localProblem:localProblem}.

In addition, assume that Assumptions \ref{ass:main:assumptions:S2} and \ref{ass:main:assumptions:K2} hold, $p \geq 3$, that we are using the partition $\Pi_{\uni,n}$, that $u_{\eps_n}$ satisfies Assumption \ref{ass:main:assumptions:R1} for all $T >0$ and that $u$ satisfies Assumption \ref{ass:main:assumptions:R3}.
For some $1 \leq q_1 < \infty$ and $ 1 \leq q_2 < \infty$ and $0 < \alpha_i \leq 1$ for $1 \leq i \leq 3$, assume furthermore that $u_0 \in \Lip(\alpha_1,\Lp{q_1}(\Omega)) \cap \Lp{\infty}(\Omega)$, $\cA^*\ell \in \Lip(\alpha_2,\Lp{q_2}(\Omega)) \cap \Lp{\infty}(\Omega)$ and $K \in \text{C}^{0,\alpha_3}(\Omega)$. For some $\kappa > 0$, let $T(n) = \l \frac{2}{2 + 3C_{\mathrm{op}}^4} \r \log(\eps_n^{-\kappa})$ and assume that
\[
\tau_n \ll \frac{1}{\log(\eps_n^{-\kappa})^{(2p-3)}} \eps_n^{2 + 2p + \kappa}, 
\]
\begin{align*}
\eps_n &\gg \max \Biggl\{ n^{-\alpha_1/\kappa}, n^{-\alpha_2/\kappa}, \ls \expW \l n^{\frac{\alpha_3}{\max\l 1 + (1+p + \alpha_3)/\kappa, p - 1 \r}}\r \rs^{-1/\kappa}, \\
& \ls \expW \l   \l \log(n)  \r^{1/\max(2(p-1),2+(1+3p)/\kappa)}   \r \rs^{-1/\kappa} \Biggr\}
\end{align*}
and
\[
\frac{\log(n) \eps_n^{2p}}{n} \ll \rho_n \ll \eps_n^{p+1}
\quad \text{as well as} \quad \frac{\log(\eps_n^{-\kappa})^{2(p-1)}}{ \eps_n^{1+3p} \log(n)} \ll \theta_n^2 \ll \eps_n^{2\kappa}.
\]
Then, for $n$ large enough, we have:
\begin{align}
  &\sup_{1 \leq k \leq N } \sup_{t \in (t^{k-1},t^k]} \Vert \cI_n \bar{u}_{n}^k - u(t,\cdot) \Vert_{\Lp{2}(\Omega')} \leq C \Bigg( \eps_n\log(\eps_n^{-\kappa}) \notag \\
  &+ \eps_n^{-\kappa}\ls n^{-\alpha_1} + n^{-\alpha_2} + \frac{\log(\eps_n^{-\kappa})^{(p-1)}}{\eps_n^{1+p + \alpha_3} n^{\alpha_3}} + \tau_n \frac{\log(\eps_n^{-\kappa})^{2p-3}}{\eps_n^{2(1+p)}}  + \theta_n \rs  \Biggr)  \label{eq:application:simplified:final}
\end{align}
for some $C > 0$ that might be dependent on $\Omega$, $u_0$ and $\cA^* \ell$ and with probability larger than
$$1 - \frac{C\log(\eps_n^{-\kappa})^{2(p-1)} \l 1+\frac{3 C_{\mathrm{op}}^4}{2}  \r^{2(1-p)}}{\theta_n^2\eps_n^{1+p} n\rho_n}.$$
Furthermore, the right-hand side of \eqref{eq:application:simplified:final} tends to 0 as $n \to \infty$ and the probability tends to $1$.

\end{corollary}

\begin{corollary}[Rates for $u_\infty$ using the random graph model]
\label{cor:application:final}

Assume the same setting as in Corollary \ref{cor:application:simplified} and furthermore that $\cA = \Id$. Then, for $n$ large enough, we have:
\begin{align}
  \Vert \cI_n \bar{u}_{n}^N - u_{\infty} \Vert_{\Lp{2}(\Omega')} &\leq C \Bigg( \eps_n\log(\eps_n^{-\kappa}) + \eps_n^{\kappa/5}(\cF(u_0)-\cF(u_\infty))^{1/2} \notag \\
  &+ \eps_n^{-\kappa}\ls n^{-\alpha_1} + n^{-\alpha_2} + \frac{\log(\eps_n^{-\kappa})^{(p-1)}}{\eps_n^{1+p + \alpha_3} n^{\alpha_3}} + \tau_n \frac{\log(\eps_n^{-\kappa})^{2p-3}}{\eps_n^{2(1+p)}}  + \theta_n \rs  \Biggr) \notag
\end{align}
for some $C > 0$ that might be dependent on $\Omega$ (and $d$), $u_0$ and $\cA^* \ell$ and with probability larger than 
$$1 - \frac{C\log(\eps_n^{-\kappa})^{2(p-1)}}{\theta_n^2\eps_n^{1+p} n\rho_n}.$$
Furthermore, the latter right-hand side tends to 0 as $n \to \infty$ and the probability tends to $1$.
\end{corollary}

\section{Proofs} \label{sec:proofs}

\subsection{Well-posedness} \label{subsec:wellPosedness}

Well-posedness of all our gradient flows is a central question. For the continuum nonlocal case, we will make use of nonlinear semigroup theory. The discrete case will follow from the continuum nonlocal case by using the interplay between the $p$-Laplacian and injection operators.
Lastly, the continuum local case will follow from classical results.

\subsubsection{Nonlocal problem} \label{subsubsec:wellPosedness:nonLocalProblem}

The following proposition, an extension of \cite[Theorem 3.9]{HINDS20121411}, will allow us to characterize certain functions both by an equation they satisfy as well as a variational problem they minimize. The proof follows the above-mentioned reference but has been included in Section \ref{sec:supplementary:wellPosedness} for completeness.

\begin{proposition}[Dirichlet principles] \label{prop:proofs:wellPosedness:nonLocalProblem:dirichlet}

Assume that Assumptions \ref{ass:main:assumptions:S1}, \ref{ass:main:assumptions:K1} and \ref{ass:main:assumptions:O1} hold. Let $p \geq 2$, $\mu > 0$ and $f \in \Lp{2}(\Omega)$. 
Given $n \in \bbN$, $\lambda > 0$ and functions $u,\phi \in \Lp{p}(\Omega)$, consider the equations
\begin{equation} \label{eq:proofs:wellPosedness:nonLocalProblem:dirichlet:equationN}
    \frac{\vert u \vert^{p-2}u}{n} + \lambda \l \mu \Delta_p^K u + \cA^* \cA u - f \r + u - \phi = 0
\end{equation}
and
\begin{equation} \label{eq:proofs:wellPosedness:nonLocalProblem:dirichlet:equation}
     \lambda(\mu \Delta_p^K u + \cA^* \cA u - f) + u - \phi = 0
\end{equation}
as well as their variational counterparts
\begin{align} 
   E_{n,\lambda,\cA,f}(u) &= \frac{1}{np} \int_\Omega \vert u \vert^p \, \dd x + \frac{\lambda \mu}{2p} \int_{\Omega\times\Omega} K(\vert x-y \vert) \vert u(x) - u(y) \vert^p  \, \dd x \, \dd y + \frac{\lambda}{2}\int_\Omega \l\cA u \r^2 \, \dd x \label{eq:proofs:wellPosedness:nonLocalProblem:dirichlet:En} \\
   & \qquad \qquad - \lambda \int_{\Omega} f u \, \dd x + \frac{1}{2}\int_\Omega (u-\phi)^2 \, \dd x \notag
\end{align}
and 
\begin{align} \label{eq:proofs:wellPosedness:nonLocalProblem:dirichlet:E}
    E_{\lambda,\cA,f}(u) & = \frac{\lambda \mu}{2p} \int_\Omega \int_\Omega K(\vert x-y \vert) \vert u(x) - u(y) \vert^p \, \dd y \dd x + \frac{\lambda}{2}\int_\Omega \l\cA u \r^2 \, \dd x \\
    & \qquad \qquad - \lambda \int_{\Omega} f u \, \dd x + \frac{1}{2}\int_\Omega (u-\phi)^2 \, \dd x.   \notag 
\end{align}

\begin{enumerate}
    \vspace{-2mm}
    \setlength\itemsep{-1mm}
    \item If $u \in \Lp{p}(\Omega)$ satisfies \eqref{eq:proofs:wellPosedness:nonLocalProblem:dirichlet:equationN} $\lambda_x$-a.e., then we have $E_{n,\lambda,\cA,f}(u) \leq E_{n,\lambda,\cA,f}(v)$ for all $v \in \Lp{p}(\Omega)$.
    \item If for $u \in \Lp{p}(\Omega)$ we have $E_{\lambda,\cA,f}(u) \leq E_{\lambda,\cA,f}(v)$ for all $v \in \Lp{p}(\Omega)$, then $u$ satisfies \eqref{eq:proofs:wellPosedness:nonLocalProblem:dirichlet:equation} $\lambda_x$-a.e..
\end{enumerate}
\end{proposition}

\begin{remark}[Dirichlet principles]
It is clear from the proof of Proposition \ref{prop:proofs:wellPosedness:nonLocalProblem:dirichlet} that the converse of the two statements of Proposition \ref{prop:proofs:wellPosedness:nonLocalProblem:dirichlet} can be shown analogously.
\end{remark}

In the next step, our aim will be to show that if $u_n \rightharpoonup u$, we have $\liminf_{n\to\infty} E_{n,\lambda,\cA,f}(u_n) \geq E_{\lambda,\cA,f}(u)$, where $E_{n,\lambda,\cA,f}$ and $E_{\lambda,\cA,f}$ are respectively defined in \eqref{eq:proofs:wellPosedness:nonLocalProblem:dirichlet:En} and \eqref{eq:proofs:wellPosedness:nonLocalProblem:dirichlet:E}.

\begin{proposition}[$\liminf$-inequality for $E_{n,\lambda,\cA,f}$]
\label{prop:proofs:wellPosedness:nonlocalProblem:weakLsc}
Assume that Assumptions \ref{ass:main:assumptions:S1}, \ref{ass:main:assumptions:K1} and \ref{ass:main:assumptions:O1} hold. Let $p \geq 2$, $\mu > 0$, $f \in \Lp{2}(\Omega)$, $\lambda > 0$ and $\phi \in \Lp{p}(\Omega)$. Let $\{u_n\}_{n=1}^\infty\subset \Lp{p}(\Omega)$ and $u\in\Lp{p}(\Omega)$ be functions so that $u_n \rightharpoonup u$. Then, 
\[
    E_{\lambda,\cA,f}(u) \leq \liminf_{n \to \infty} E_{n,\lambda,\cA,f}(u_n)
\]
where $E_{n,\lambda,\cA,f}$ and $E_{\lambda,\cA,f}$ are respectively defined in \eqref{eq:proofs:wellPosedness:nonLocalProblem:dirichlet:En} and \eqref{eq:proofs:wellPosedness:nonLocalProblem:dirichlet:E}.
\end{proposition}

\begin{proof}
In the proof $C>0$ will denote a constant that can be arbitrarily large, (which might be) dependent on the kernel $K$, $p$, $\lambda$, $\mu$, $\cA$ and/or $\Omega$, that may change from line to line.

We start by recalling \eqref{eq:proofs:wellPosedness:nonLocalProblem:dirichlet:En} for $v \in \Lp{p}(\Omega)$:
\begin{align*}
    E_{n,\lambda,\cA,f}(v) &= \frac{1}{np} \int_\Omega \vert v \vert^p \, \dd x + \frac{\lambda \mu}{2p} \int_{\Omega\times\Omega} K(\vert x-y \vert) \vert v(x) - v(y) \vert^p  \, \dd x \, \dd y + \frac{\lambda}{2}\int_\Omega \l\cA v \r^2 - 2fv \, \dd x\\
    & \qquad \qquad + \frac{1}{2}\int_\Omega (v-\phi)^2 \, \dd x  \\
    &=: T_1(v) + T_2(v)  + T_3(v) + T_4(v). 
\end{align*}

For the $T_1$ term, by weak lower semi-continuity of norms, we have
\[
\liminf_{n \to \infty} T_1(u_n) = \liminf_{n\to\infty} \frac{1}{np} \|u_n\|_{\Lp{p}}^p = 0. 
\]

For the $T_2$ term, we first note that $T_2$ is proper and convex. 
Next, let $0 < r < R$. We claim that $T_2$ is bounded on the ball $B_{\Lp{p}}(u,R)$. Indeed, for $v \in B_{\Lp{p}}(u,R)$:
\[
\vert T_2(v) \vert \leq C  \int_\Omega \int_\Omega 2^{p-1}(\vert v(y) \vert^p + \vert v(x) \vert^p ) \, \dd y \, \dd x  \leq C (R + \Vert u \Vert_{\Lp{p}}) \leq C. 
\]
Combining the above, by \cite[Proposition 5.11]{maso2012introduction}, we have that $T_2$ is Lipschitz continuous on $B_{\Lp{p}}(u,r)$ which in turn implies that $T_2$ is continuous and convex in $\Lp{p}(\Omega)$. By \cite[Corollary 2.2]{ekeland1999convex}, we deduce that $T_2$ is weakly lower-semicontinuous. 
Analogously, $T_3$ and $T_4$ are convex and continuous and therefore weakly lower semi-continuous by~\cite[Corollary 2.2]{ekeland1999convex}.

Collecting all the latter results, we obtain:
\[
\liminf_{n \to \infty} E_{n,\lambda,\cA,f}(u_n)  \geq T_2(u) + T_3(u) + T_4(u) = E_{\lambda,\cA,f}(u). \qedhere
\]
\end{proof}

The next properties are easily checked: in particular, monotony and coercivity follow from \cite[Lemma 2.3]{ANDREU2008201} and \cite[ Lemma 3.4 and Lemma 3.6]{Byrstrom}. For completeness, the proof may be found in Section \ref{sec:supplementary:wellPosedness}.

\begin{lemma}[Properties of $\cE_{n,\lambda,\cA,f}$] \label{lem:proofs:wellPosedness:nonlocalProblem:properties}
Assume Assumptions \ref{ass:main:assumptions:S1}, \ref{ass:main:assumptions:K1} and \ref{ass:main:assumptions:O1} hold. Let $p \geq 2$, $\mu > 0$ and $f \in \Lp{2}(\Omega)$. For $n \in \bbN$ and $\lambda > 0$, we define the operator:
\begin{equation} \label{eq:proofs:wellPosedness:nonlocalProblem:properties:En}
    \cE_{n,\lambda,\cA,f}(u) = \frac{\vert u \vert^{p-2}u}{n} + u + \lambda ( \mu \Delta_p^Ku + \cA^*\cA u - f).
\end{equation}
The following properties are satisfied:
\begin{enumerate}
    \vspace{-2mm}
    \setlength\itemsep{-1mm}
    \item For $q$ such that $p^{-1} + q^{-1} = 1$, $\cE_{n,\lambda,\cA,f}:\Lp{p}(\Omega) \mapsto \Lp{q}(\Omega)$ and
\[ \|\cE_{n,\lambda,\cA,f}(u)\|_{\Lp{q}} \leq C \l \l 1+\frac{1}{n}\r \|u\|_{\Lp{p}}^{\frac{p}{q}} + \|u\|_{\Lp{p}} + \|f\|_{\Lp{2}}\r; \] 
    \item $\cE_{n,\lambda,\cA,f}$ is hemicontinuous, monotone and coercive.
\end{enumerate}
\end{lemma}


We now proceed to show a range condition on our evolution operator which will allow us to deduce the existence of a solution to the nonlocal problem. 

\begin{proof}[Proof of Proposition \ref{prop:proofs:wellPosedness:nonlocalProblem:completeAccretivityRangeCondition}]

In the proof $C>0$ will denote a constant that can be arbitrarily large, (which might be) dependent on the kernel $K$ and/or $\Omega$, that may change from line to line.

We begin by showing complete accretivity of $\cE^K_{\cA,f}$. Let $u,v \in \Lp{p}(\Omega)$, $h \in \cH$ and consider: 
\begin{align*}
    \int_\Omega \l \cE^K_{\cA,f}(u) - \cE^K_{\cA,f}(v) \r h(u - v) \, \dd x &= \mu\int_\Omega \l \Delta_p^K u - \Delta_p^K v \r h(u - v) \, \dd x \\
    &+ \int_\Omega \l \cA^* \cA u - \cA^* \cA v \r h(u - v) \, \dd x\\
    &=: T_1 + T_2.
\end{align*}
For the $T_1$ term, 
we obtain:
\begin{align*}
T_1 & = \mu \int_{\Omega\times\Omega} K(|x-y|) \Bigg( |u(x) - u(y)|^{p-2}(u(x)-u(y)) - |v(x)-v(y)|^{p-2}(v(x)-v(y)) \Bigg) \\
 & \qquad \qquad h(u(x)-v(x)) \, \dd x \, \dd y \\
 & = \frac{\mu}{2}\int_{\Omega\times\Omega} K(\vert x - y \vert) \Bigg( |u(x)-u(y)|^{p-2} (u(x)-u(y)) - |v(x)-v(y)|^{p-2} (v(x) - v(y)) \Bigg) \\
 & \qquad \qquad \ls h(u(x)-v(x)) - h(u(y)-v(y))\rs \, \dd x \, \dd y. 
\end{align*}
Since $h$ and $t\mapsto |t|^{p-2}t$ are both increasing then by splitting the latter equation in cases where $u(y)-v(y) \geq u(x) - v(x)$ (and conversely) we see that $T_1 \geq 0$. By Assumptions \ref{ass:main:assumptions:O1} and \ref{ass:main:assumptions:O2}, we know that $T_2 \geq 0$ and therefore, by Proposition \ref{prop:background:completelyAccretiveCharacterization}
and Assumption \ref{ass:main:assumptions:S1}, $\cE^K_{\cA,f}$ is completely accretive. 

For the range condition, let $\phi \in \Lp{p}(\Omega) \subseteq \Lp{q}(\Omega)$ where $q = p /(p-1)$ since $p \geq 2$ . We first show that there exists a solution to the equation
\[
\frac{\vert u \vert^{p-2}u}{n} + u + \lambda \l \mu\Delta_p^Ku + \cA^* \cA u - f \r = \cE_{n,\lambda,\cA,f}(u) = \phi.
\]
where $\cE_{n,\lambda,\cA,f}$ is defined in \eqref{eq:proofs:wellPosedness:nonlocalProblem:properties:En}.
By Lemma \ref{lem:proofs:wellPosedness:nonlocalProblem:properties}, the operator $\cE_{n,\lambda,\cA,f}$ satisfies all the conditions required to apply Theorem \ref{thm:background:nonlinearStationary}
and we deduce that for all $n \in \bbN$, there exists a function $u_n \in \Lp{p}(\Omega)$ such that $\cE_{n,\lambda,\cA,f}(u_n) = \phi$, or formulated otherwise: $u_n$ satisfies \eqref{eq:proofs:wellPosedness:nonLocalProblem:dirichlet:equationN}. 

Next, we claim that $u_n \llSgt \phi + \lambda f$. Indeed, let $h \in \cH$ and, using \eqref{eq:proofs:wellPosedness:nonLocalProblem:dirichlet:equationN}, we compute as follows:
\begin{align}
\int_\Omega \l \phi + \lambda f \r h(u_{n}) \, \dd x &= \int_\Omega u_{n} h(u_{n}) \, \dd x + \lambda \mu \int_\Omega \Delta_p^K u_{n} h(u_{n}) \, \dd x  \notag \\
& \qquad \qquad + \frac{1}{n}\int_\Omega \vert u_{n} \vert^{p-2} u_{n}h(u_{n}) \, \dd x + \lambda \int_\Omega \cA^* \cA u_{n} h(u_{n}) \, \dd x \notag \\
&\geq \int_\Omega u_{n} h(u_{n}) \, \dd x \label{eq:proofs:wellPosedness:nonlocalProblem:completeAccretivityRangeCondition:orderRelation}
\end{align}
where we used 
the same argument as to show $T_1\geq 0$ (using $v=0$) to infer $\int_\Omega \Delta_p^K u_n h(u_n) \, \dd x\geq 0$,
the fact that $\signFunc(u_n) = \signFunc(h(u_n))$ and the Assumption~\ref{ass:main:assumptions:O2} for~\eqref{eq:proofs:wellPosedness:nonlocalProblem:completeAccretivityRangeCondition:orderRelation}. By Lemma \ref{lem:background:orderRelation:properties} and since $f \in \Lp{p}(\Omega)$, we therefore obtain $u_{n} \llSgt \phi + \lambda f \in \Lp{p}(\Omega)$ and consequently $\Vert u_{n} \Vert_{\Lp{p}} \leq \Vert \phi + \lambda f \Vert_{\Lp{p}} \leq C$. From this, we deduce that the set $\{u_n\}_{n=1}^\infty$ is uniformly bounded in $\Lp{p}(\Omega)$ and hence, there exists $u^* \in \Lp{p}(\Omega)$ and a subsequence $\{u_{n_m}\}_{m=1}^\infty$ such that $u_{n_m} \rightharpoonup u^*$ in $\Lp{p}(\Omega)$.

By Proposition \ref{prop:proofs:wellPosedness:nonLocalProblem:dirichlet}, for each $n$, $u_n$ minimizes $E_{n,\lambda,\cA,f}$ in $\Lp{p}(\Omega)$. We will now show that $u^*$ minimizes $E_{\lambda,\cA,f}$ in $\Lp{p}(\Omega)$. Let $v \in \Lp{p}(\Omega)$. In fact, by Proposition \ref{prop:proofs:wellPosedness:nonlocalProblem:weakLsc} we have 
\[
     E_{\lambda,\cA,f}(u^*) \leq \liminf_{m \to \infty} E_{n_m,\lambda,\cA,f}(u_{n_m}) \leq \liminf_{m \to \infty} E_{n_m,\lambda,\cA,f}(v) = E_{\lambda,\cA,f}(v). 
\]
Therefore, by Proposition \ref{prop:proofs:wellPosedness:nonLocalProblem:dirichlet}, $u^*$ satisfies \eqref{eq:proofs:wellPosedness:nonLocalProblem:dirichlet:equation} which concludes the proof of the range condition.
\end{proof}

The straight-forward proof of the next result may be found in Section \ref{sec:supplementary:wellPosedness}.

\begin{corollary}[Special cases of $\cE^K_{\cA,f}$] \label{cor:proofs:wellPosedness:nonlocalProblem:specialCases}

Assume that Assumptions \ref{ass:main:assumptions:S1}, \ref{ass:main:assumptions:O1}, \ref{ass:main:assumptions:O2} and \ref{ass:main:assumptions:K1} hold. Let $p \geq 2$, $\mu > 0$ and assume that $f \in \Lp{p}(\Omega)$.
Then, the operators $\cE^K_{\cA,f}$, $\cE^K_{\cA,0}$ and $\cE^K_{0,0}$ are $m$-completely accretive.
\end{corollary}

Finally, we can deduce the following existence and uniqueness result.

\begin{proof}[Proof of Theorem \ref{thm:proofs:wellPosedness:nonlocalProblem:existenceUniqueness}]
First, we note that if we find a solution to
\begin{equation} \label{eq:proofs:wellPosedness:nonlocalProblem:existenceUniqueness:formulation2}
\begin{cases}
\frac{\partial}{\partial t}u + \cE^K_{\cA,0}(u) = f \\
u(0,\cdot) = u_0,
\end{cases}
\end{equation} 
then, $u$ solves \eqref{eq:main:notation:nonlocal:nonlocalProblem} with the operator $\cE^K_{\cA,f}$. From Corollary \ref{cor:proofs:wellPosedness:nonlocalProblem:specialCases}, we know that $\cE^K_{\cA,0}$ is $m$-completely accretive. We can therefore apply Theorem \ref{thm:background:solution} to deduce the existence of a unique solution $u$ as in Definition \ref{def:main:notation:nonlocal:strongSolutionCP}.
We obtain \eqref{eq:proofs:wellPosedness:nonlocalProblem:existenceUniqueness:contraction} by applying Lemma \ref{lem:background:orderRelation:contraction}. 

For the last part, we note that if Assumption \ref{ass:main:assumptions:O3} is satisfied, then $v = 0$ solves \eqref{eq:proofs:wellPosedness:nonlocalProblem:existenceUniqueness:formulation2} with $f = 0$ and $u_0 = 0$, so that by inserting $v = 0$ in \eqref{eq:proofs:wellPosedness:nonlocalProblem:existenceUniqueness:contraction}, we obtain \eqref{eq:proofs:wellPosedness:nonlocalProblem:existenceUniqueness:bound}.
\end{proof}

The next result is a stability result for solutions to \eqref{eq:main:notation:nonlocal:nonlocalProblem} which can be considered an extension of \cite[Theorem 5.1]{ElBouchairi} whose proof can be found in Section \ref{sec:supplementary:wellPosedness}.

\begin{proposition}[Stability of solutions to \eqref{eq:main:notation:nonlocal:nonlocalProblem}]
\label{prop:proofs:nonlocalDiscreteContinuum:stability}
 Assume that Assumptions \ref{ass:main:assumptions:S1}, \ref{ass:main:assumptions:O1}, \ref{ass:main:assumptions:O2} and \ref{ass:main:assumptions:O3} hold. Let $p \geq 2$, $\mu > 0$ and $T\geq 1$. Furthermore, for $i=1,2$, let $K_i$ satisfy Assumption \ref{ass:main:assumptions:K1}, $u_{0,i} \in \Lp{p}(\Omega)$ and $f_i \in \Lp{p}(\Omega)$. 
Then, for $i=1,2$, there exists a unique solution $u_i$ to the nonlocal problem \eqref{eq:main:notation:nonlocal:nonlocalProblem} with evolution operator $\cE^{K_i}_{\cA,f_i}$ and initial condition $u_{0,i}$ and we have the following stability estimates for some $C > 0$ dependent on $\Omega$, $u_0$ and $f_i$:
\begin{enumerate}
    \vspace{-2mm}
    \setlength\itemsep{-1mm}
    \item if, for either $i=1$ or $i=2$, we have $u_{0,i} \in \Lp{2(p-1)}(\Omega)$ and $f_i \in \Lp{2(p-1)}(\Omega)$ then
\end{enumerate}
\begin{equation*}
    \Vert u_2(t,\cdot) - u_1(t,\cdot) \Vert_{\Lp{2}} \leq Ct^p \l  \sup_{x \in \Omega} \Vert K_2(x-\cdot) - K_1(x-\cdot) \Vert_{\Lp{2}} + \Vert f_1 - f_2 \Vert_{\Lp{2}} \r + \Vert u_{0,2} - u_{0,1} \Vert_{\Lp{2}};
\end{equation*}
\begin{enumerate}
\setcounter{enumi}{1}
\item if, for either $i-1$ or $i=2$, we have $u_{0,i} \in \Lp{\infty}(\Omega)$ and $f_i \in \Lp{\infty}(\Omega)$ then
\end{enumerate}
\begin{equation*} 
    \Vert u_2(t,\cdot) - u_1(t,\cdot) \Vert_{\Lp{2}} \leq C t^p \l  \Vert K_2 - K_1 \Vert_{\Lp{2}} + \Vert f_1 - f_2 \Vert_{\Lp{2}} \r + \Vert u_{0,2} - u_{0,1} \Vert_{\Lp{2}}.
\end{equation*}
\end{proposition}

\begin{remark}[Generality of $f$ in \eqref{eq:main:notation:nonlocal:nonlocalProblem}]

In Section \ref{subsec:wellPosedness}, we have considered a general function $f$ in \eqref{eq:main:notation:nonlocal:nonlocalProblem}.
We could continue to do so below, but for ease of exposition, from now on, we will only consider $f = \cA^*\ell$ or $f = \cP_n \cA^*\ell$.
\end{remark}

\subsubsection{Discrete nonlocal problem}

As a corollary of Proposition \ref{prop:proofs:wellPosedness:nonlocalProblem:completeAccretivityRangeCondition}, we can prove the well-posedness of the discrete problem \eqref{eq:main:notation:discrete:nonlocalProblemFully} in the case that is of interest to us. The proof of the next corollary also displays the importance of accretivity of our operator.

\begin{corollary}[Well-posedness of \eqref{eq:main:notation:discrete:nonlocalProblemFully}]
\label{cor:proofs:wellPosedness:discreteProblem:existence}

Assume that Assumptions \ref{ass:main:assumptions:S1}, \ref{ass:main:assumptions:O1}, \ref{ass:main:assumptions:O2}, \ref{ass:main:assumptions:O3}, \ref{ass:main:assumptions:O4} and \ref{ass:main:assumptions:K1} hold. Let $p \geq 2$, $\mu > 0$, $T>0$, $u_0 \in \Lp{p}(\Omega)$, $\ell \in \Lp{2}(\Omega)$ and assume that $\cA^* \ell \in \Lp{p}(\Omega)$. Furthermore, let $n \in \bbN$ and define $\bar{K} = \cP_n \tilde{K}$, $\bar{f} = \cP_n \cA^*\ell $, $\bar{u}_0 = \cP_n u_0$.
Then, for any partition $0 = t^0 < t^1 < \dots < t^N = T$, there exists a sequence $\{\bar{u}_n^k\}_{k=0}^N$ satisfying \eqref{eq:main:notation:discrete:nonlocalProblemFully} with the above parameters that is well-defined and unique. We also have
\begin{equation} \label{eq:proofs:wellPosedness:discreteProblem:uniformBound}
\Vert \cI_n \bar{u}_n^k \Vert_{\Lp{r}} \leq \Vert u_0 \Vert_{\Lp{r}} + T \Vert \cA^*\ell \Vert_{\Lp{r}}
\end{equation}
for $1 \leq r \leq \infty$.
\end{corollary}

\begin{proof}
We first start by considering the well-posedness of the sequence $\{u_n^k\}_{k=0}^N$ defined iteratively by $u_n^0 = \cI_n \bar{u}_0 = \cI_n \cP_n u_0$ and 
\begin{equation} \label{eq:proofs:wellPosedness:discreteProblem:iterative}
 (\Id + \tau^{k-1}\cE^{\cI_n \bar{K}}_{\cA,\cI_n \bar{f}})(u_n^k) = u_n^k + \tau^{k-1}(\mu \Delta_p^{\cI_n\bar{K}} u_n^k + \cA^*\cA u_n^k -\cI_n \bar{f}) = u_n^{k-1}.
\end{equation}
This can be reformulated as $u_n^k = (\Id + \tau^{k-1}\cE^{\cI_n \bar{K}}_{\cA,0})^{-1}(u_n^{k-1} + \tau^{k-1} \cI_n \bar{f})$

By \cite[Lemma 2.1]{ElBouchairi}, we have that $\Vert \cI_n \bar{K} \Vert_{\Lp{\infty}} \leq \Vert K \Vert_{\Lp{\infty}} < \infty$, so we can apply Corollary \ref{cor:proofs:wellPosedness:nonlocalProblem:specialCases} to deduce that $\cE^{\cI_n \bar{K}}_{\cA,0}$ is $m$-completely accretive. In particular, by Proposition~\ref{prop:proofs:wellPosedness:nonlocalProblem:completeAccretivityRangeCondition} and~\cite[Section 2]{BenilanCrandall}, for $\lambda > 0$, we have that $(\Id + \lambda \cE^{\cI_n \bar{K}}_{\cA,0})^{-1}$ is single-valued on $\domain((\Id + \lambda \cE^{\cI_n \bar{K}}_{\cA,0})^{-1}) = \range(\Id + \lambda \cE^{\cI_n \bar{K}}_{\cA,0}) = \Lp{p}(\Omega)$ and, for $1 \leq r \leq \infty$:
\[
\Vert (\Id + \lambda \cE^{\cI_n \bar{K}}_{\cA,0})^{-1}(g_1) - (\Id + \lambda \cE_{\cA,0})^{-1}(g_2) \Vert_{\Lp{r}} \leq \Vert g_1 - g_2 \Vert_{\Lp{r}} 
\]
for any $g_i \in \Lp{p}(\Omega)$. By Assumption \ref{ass:main:assumptions:O3}, we have that $(\Id + \lambda \cE^{\cI_n \bar{K}}_{\cA,0})^{-1}(0) = 0$ so that, combining with the above,
\begin{equation*} 
    \Vert (\Id + \lambda \cE^{\cI_n \bar{K}}_{\cA,0})^{-1}(g) \Vert_{\Lp{r}} \leq \Vert g  \Vert_{\Lp{r}} 
\end{equation*}
for any $g \in \Lp{p}(\Omega)$.

Since by \cite[Lemma 2.1]{ElBouchairi}, we have that $\Vert \cI_n \bar{f} \Vert_{\Lp{r}} \leq \Vert \cA^* \ell \Vert_{\Lp{r}}$ and $\Vert \cI_n \bar{u}_0 \Vert_{\Lp{p}} \leq \Vert u_0 \Vert_{\Lp{p}}$,  we can now proceed to show that $u_n^k$ is well-posed by induction. For $k = 0$, we have that $u_n^k = \cI_n \bar{u}_0 \in \Lp{p}(\Omega)$. Now, assume that for $1 \leq m \leq k-1$, $u_n^m \in \Lp{p}(\Omega)$ is well-defined. Then, $u_n^{k-1} + \tau^{k-1}\cI_n \bar{f} \in \Lp{p}(\Omega)$ and, since $(\Id + \tau^{k-1} \cE^{\cI_n \bar{K}}_{\cA,0})^{-1}$ is single-valued on $\Lp{p}(\Omega)$, $u_n^{k}$ is well-defined with 
\begin{align*}
\Vert u_n^k \Vert_{\Lp{r}} & = \lda \l \Id + \tau^{k-1} \cE_{\cA,0}^{\cI_n\bar{K}}\r^{-1} (u_n^{k-1} + \tau^{k-1}\cI_n\bar{f}) \rda_{\Lp{r}} \\
 & \leq \lda u_n^{k-1} + \tau^{k-1}\cI_n\bar{f} \rda_{\Lp{r}} \\
 & \leq \lda u_n^{k-1} \rda_{\Lp{r}} + \tau^{k-1} \lda \cI_n\bar{f} \rda_{\Lp{r}}.
\end{align*}
By induction
\[ \Vert u_n^k \Vert_{\Lp{r}} \leq \|u_n^0\|_{\Lp{r}} + \sum_{m=0}^{k-1} \tau^m \| \cI_n \bar{f} \|_{\Lp{r}} = \|u_n^0\|_{\Lp{r}} +t^k \| \cI_n \bar{f} \|_{\Lp{r}} \leq \|u_n^0\|_{\Lp{r}} + T \| \cI_n \bar{f} \|_{\Lp{r}}. \]
The well-posedness implies that $\{u_n^k\}_{k=0}^N$ is the unique sequence that is defined iteratively by \eqref{eq:proofs:wellPosedness:discreteProblem:iterative} and such that $u_n^0 = \cI_n \bar{u}_0$.

Now, assume that there exists $\{\bar{u}_n^k\}_{k=0}^N$ that solves \eqref{eq:main:notation:discrete:nonlocalProblemFully} with $\bar{f}$, $\bar{u}_0$, $\bar{G}_n$ and $\bar{K}$ as defined above. Then, we have $\cI_n \bar{u}_0 = \cI_n \cP_n u_0 = u_n^0$ and
\begin{align}
    (\Id + \tau^{k-1} \cE^{\cI_n \bar{K}}_{\cA,0})\cI_n \bar{u}_n^k 
    &= \cI_n \bar{u}_{n}^k + \tau^{k - 1} \l  \cI_n (\mu \Delta_{p,n}^{\bar{K}} \bar{u}_n^k ) + \cI_n( \bar{G}_n(\bar{u}_{n}^k) ) \r \label{eq:proofs:wellPosedness:discreteProblem:discreteToContinuum} \\
    &= \sum_{i=1}^{\vert \Pi_n \vert} \chi_{\pi_i^n} \ls (\bar{u}_n^{k-1})_i + \tau^{k-1} (\bar{f})_i \rs \label{eq:proofs:wellPosedness:discreteProblem:problem} \\
    &= \cI_n \bar{u}_n^{k-1} + \tau^{k-1}\cI_n \bar{f} \notag
\end{align}
where we used Assumption \ref{ass:main:assumptions:O4} and \cite[Lemma 6.1]{ElBouchairi} for \eqref{eq:proofs:wellPosedness:discreteProblem:discreteToContinuum} and \eqref{eq:main:notation:discrete:nonlocalProblemFully} for \eqref{eq:proofs:wellPosedness:discreteProblem:problem}. By the uniqueness of the sequence $\{u_n^k\}_{k=0}^N$, we have that $u_n^k = \cI_n \bar{u}_n^k$.

To conclude the proof, we show the existence of $\{\bar{u}_n^k\}_{k=0}^N$. First, recall from \eqref{eq:main:notation:discrete:nonlocalProblemFully} that $\bar{u}_n^0 = \bar{u}_0$ and 
\[
\bar{u}_n^k = (\Id + \tau^{k-1}(\mu \Delta_{p,n}^{\bar{K}} + \bar{G}_n))^{-1}(\bar{u}^{k-1}_{n} + \tau^{k-1}\bar{f}).
\]
By the same argument as above, we need to show that the operator $\mu \Delta_{p,n}^{\bar{K}} + \bar{G}_n$ is accretive on $\bbR^{\vert \Pi_n \vert}$ for $(\Id + \tau^{k-1}(\mu \Delta_{p,n}^{\bar{K}} + \bar{G}_n))^{-1}$ to be well-defined and unique.
It is clear that $\range(\Delta_{p,n}^{\bar{K}} + \bar{G}_n) \subseteq \bbR^{\vert \Pi_n \vert}$. By 
Proposition \ref{prop:background:monotony}, we know that accretivity is equivalent to monotony in $\bbR^{\vert \Pi_n \vert}$ and it therefore only remains to verify that for $\bar{v},\bar{w} \in \bbR^{\vert \Pi_n \vert}$:
\begin{align}
\langle (\mu \Delta_{p,n}^{\bar{K}} + \bar{G}_n)(\bar{v})  - (\mu \Delta_{p,n}^{\bar{K}} + \bar{G}_n)(\bar{w}), \bar{v} - \bar{w} \rangle &= \langle \mu \Delta_{p,n}^{\bar{K}}(\bar{v}) - \mu \Delta_{p,n}^{\bar{K}}(\bar{w}), \bar{v} - \bar{w} \rangle \notag \\
&+ \langle \bar{G}_n(\bar{v}) - \bar{G}_n(\bar{w}), \bar{v} - \bar{w} \rangle\notag \\
&=: T_1 + T_2  \geq 0. \notag
\end{align}
The proof of $T_1 \geq 0$ is analogous to what we have shown for the continuum nonlocal Laplacian in the proof of Proposition \ref{prop:proofs:wellPosedness:nonlocalProblem:completeAccretivityRangeCondition}; $T_2 \geq 0$ is due to Assumption \ref{ass:main:assumptions:O4}.
\end{proof}

\subsubsection{Local problem}

\begin{theorem}[Well-posededness of \eqref{eq:main:notation:localProblem:localProblem}]
\label{thm:proofs:wellPosedness:localProblem:existenceUniqueness}
Assume that Assumptions \ref{ass:main:assumptions:S1} and \ref{ass:main:assumptions:O1}
hold.
Let $p \geq 2$, $\mu > 0$, $T>0$, $\ell \in \Lp{2}(\Omega)$ and $u_0 \in \Lp{p}(\Omega)$. Then, there exists a unique weak solution $u(t,x)$ to the evolution problem \eqref{eq:main:notation:localProblem:localProblem}. 
\end{theorem}

\begin{proof}
We begin the proof by verifying some properties of the operator $\cA^*\cA$. By Assumption \ref{ass:main:assumptions:O1}, $\cA^* \cA$ is linear and we know that $\cA^* \cA$ is 
bounded, hence continuous and therefore hemicontinuous by Lemma \ref{lem:background:hemicontinuity}. Since $\cA$ is linear, we can define the seminorm $ \cS(v) := \Vert \cA v \Vert_{\Lp{2}}$ on $\Lp{2}(\Omega)$. By the boundedness of $\cA$, we then note that $\cS(\cdot) + \Vert \cdot \Vert_{\Lp{2}}$ and $\Vert \cdot \Vert_{\Lp{2}}$ are equivalent. 
Finally, we have that $\cA^* \cA$ is monotone since for $u,v \in \Lp{2}(\Omega)$:
\[
\langle \cA^* \cA u - \cA^* \cA v, u-v \rangle_{\Lp{2}} = \langle \cA^* \cA (u-v), u-v \rangle_{\Lp{2}} = \Vert \cA(u-v) \Vert_{\Lp{2}}^2 \geq 0. 
\]
We deduce that
$\cA^* \cA$ satisfies the assumptions of Theorem \ref{thm:background:nonlinearEvolution} with $V_1 = H = V_1^* = \Lp{2}(\Omega)$ and $p_1 = 2$.

We now define the operator $\cD$ by 
\[
\cD(u)(v) = \int_\Omega \vert \nabla u \vert^{p-2} \nabla u \cdot \nabla v \, \dd x 
    \]
for $u,v \in \Wkp{1}{p}$. It is straight-forward (similarly to Lemma \ref{lem:proofs:wellPosedness:nonlocalProblem:properties}) to check that $\cD:\Wkp{1}{p} \mapsto (\Wkp{1}{p})^*$ is bounded, hemicontinuous and satisfies $\Vert \cD(u)\Vert_{(\Wkp{1}{p})^*} \leq C \Vert u \Vert_{\Wkp{1}{p}}^{p-1}$ as well as $(\cD(u)-\cD(v))(u-v) \geq 0$ for all $u,v \in \Wkp{1}{p}$ (the last claim follows from \cite[Lemma 3.6]{Byrstrom}). Furthermore, we define the Sobolev seminorm $[v] = \Vert \nabla v \Vert_{\Lp{p}}$ and one can show that there exists $\lambda >0$ such that $[v] + \lambda \Vert v \Vert_{\Lp{2}} \geq \Vert v \Vert_{\Wkp{1}{p}}$ (see for example \cite[Chapter 2, Section 1.5.1]{lions1969quelques}) for all $v \in \Wkp{1}{p}$. It is also clear that $\cD(v)(v) \geq [v]^p$. Hence $\cD$ satisfies the assumptions of Theorem \ref{thm:background:nonlinearEvolution} with $V_2 = \Wkp{1}{p}(\Omega) \subseteq \Lp{2}(\Omega) \subseteq \Wkp{1}{p}(\Omega)^*$ (where the inclusion is made possible by the fact that $p \geq 2$) and $p_2 = p$.

After an application of Theorem \ref{thm:background:nonlinearEvolution}, we therefore obtain the existence of a unique function 
\[
u \in \Lp{p}(0,T;\Wkp{1}{p}(\Omega)) \cap \Lp{2}(0,T;\Lp{2}(\Omega)) \text{ and $u \in \Lp{\infty}(0,T;\Lp{2}(\Omega))$}
\]
satisfying Definition \ref{def:main:notation:localProblem:weakSolution} (see \cite[Chapter 2, Section 1.5.1 or Example 1.7.2]{lions1969quelques} for the treatment of the boundary term). 
\end{proof}

\subsection{Rates}

We now turn our attention to establishing rates between our gradient flows defined on $[0,T]$. However, all of our rates will be expressed explicitly as a function of $T$ as our aim will be to take $T\to \infty$ as to approximate the solution of \eqref{eq:intro:regularizationProblem}. We begin with the continuum-to-continuum case.

\subsubsection{Continuum nonlocal to local rates}

\begin{remark}[Higher regularity] \label{rem:proofs;rates:continuumRates:increasedRegularity}
In Theorem \ref{thm:proofs:continuumRates:continuumNonlocalLocal}, we will assume higher regularity on the solutions of \eqref{eq:main:notation:nonlocal:nonlocalProblem} and \eqref{eq:main:notation:localProblem:localProblem}. Previous attempts at deriving rates in related problems \cite{10.2307/2158008},\cite{M2AN_1975__9_2_41_0},\cite{10.2307/2158136},\cite{10.2307/2153239} have led to similar assumptions. 

In the nonlocal case, we require the solutions to be in the same space as the local solution, namely $\Wkp{1}{p}(\Omega)$. This is a natural requirement for establishing rates which is fulfilled when approximating local problems by finite elements as in \cite{M2AN_1975__9_2_41_0}. This is in contrast with 
our finite-differences approach for which conservation of regularity properties of the local problem is not inherent. For a similar problem, regularity of a nonlocal solution is studied in \cite{aubert2009nonlocal}.

In the local case, we extend the regularity from $\Wkp{1}{p}(\Omega)$ to some fractional Sobolev space $\Wkp{s}{p}(\Omega)$. This will allow us to consider a function $u$ with continuous first, second and third derivatives. In order to align for example with \cite[Section 6]{M2AN_1975__9_2_41_0} where functions are taken in $\Wkp{2}{p}(\Omega)$, and due to the regularity of Sobolev spaces 
we would like to have $s = k + \delta$ for $k = 2,3$. As it will turn out, this uncovers a subtle interplay between regularity and dimension of the underlying space.
\end{remark}

\begin{remark}[Embeddings of $\Wkp{k+\delta}{p}(\Omega)$ into $\Ck{k}(\Omega)$] 
\label{rem:proofs:rates:continuumRates:embeddings}
In view of Remark \ref{rem:proofs;rates:continuumRates:increasedRegularity}, we are interested in the embedding of $\Wkp{k + \delta}{p}(\Omega)$ into $\Ck{k}(\Omega)$.
However, $\Wkp{k+\delta}{p}(\Omega)$ embeds into $\Ck{k}(\Omega)$ only when $\delta > d/p$: the weaker we want our regularity assumption on the local solution to be, the higher the regularization parameter $p$ needs to be. 
\end{remark}

The proof of the next result follows the same structure as the proof of Proposition \ref{prop:proofs:rates:NonFullyDiscreteLocal:L2rates} in order to apply Gronwall's lemma. The terms involving the application of $\Delta_p^{K^{\eps_n}} - \Delta_p$ on a regular function give rise to rates through Taylor expansions as in \cite[Theorem A.1]{Calder_2018}. For completeness, the full proof is included in Section \ref{sec:supplementary:rates}.

\begin{theorem}[Continuum nonlocal-to-local rates]
\label{thm:proofs:continuumRates:continuumNonlocalLocal}

Assume that Assumptions \ref{ass:main:assumptions:S1}, \ref{ass:main:assumptions:O1}, \ref{ass:main:assumptions:O2}, \ref{ass:main:assumptions:K1} and \ref{ass:main:assumptions:L1} hold. Let $p \geq 2$, $\mu > 0$, $T>0$, $\ell \in \Lp{2}(\Omega)$, $u_0 \in \Lp{p}(\Omega)$, $\Omega'$ be compactly contained in $\Omega$ and assume that $\cA^* \ell \in \Lp{p}(\Omega)$.
Then, for all $n$, there exists a solution $u_{\eps_n}$ to \eqref{eq:main:notation:nonlocal:nonlocalProblem} with kernel $K_{\eps_n}$ and $f=\cA^*\ell$ and a solution $u$ to \eqref{eq:main:notation:localProblem:localProblem}.

In addition, assume that Assumptions \ref{ass:main:assumptions:S2} and \ref{ass:main:assumptions:K2} hold, $p \geq 3$, that $u_{\eps_n}$ satisfies Assumption \ref{ass:main:assumptions:R1} and $u$ satisfies Assumption \ref{ass:main:assumptions:R2}.
Then, for $n$ large enough, we have:\begin{equation}\label{eq:proofs:continuumRates:l2convergence:rates}
        \Vert u_{\eps_n}(t,\cdot) - u(t,\cdot) \Vert_{\Lp{2}(\Omega')} \leq   \mathcal{O} \l \eps_n t C_1^{p-3} \ls C_1 + C_2^2 \rs \r
    \end{equation}
    where $C_1 = \sup_{t \in (0,T)} \Vert \nabla u(t,\cdot) \Vert_{\Lp{\infty}}$ and $C_2 = \sup_{t \in (0,T)} \Vert \nabla^2 u(t,\cdot) \Vert_{\Lp{\infty}}$.
\end{theorem}

\begin{remark}[Asymptotic rates in Theorem \ref{thm:proofs:continuumRates:continuumNonlocalLocal}] \label{rem:proofs:rates:continuumRates:asymptotic}

In Theorem \ref{thm:proofs:continuumRates:continuumNonlocalLocal}, the rates hold for $n$ large enough, but this is not constraining in practice. Indeed, the latter condition is derived from Lemma \ref{lem:proofs:rates:continuumRates:domainIntegration}. This implies that the rates are relevant as soon as $\eps_n < m/c$ where $m$ is the minimum distance between the closure of the compactly contained set and the boundary of $\Omega$ and $c$ is chosen so that $\closure(\Omega^\prime)\subset B(0,c)$. 
\end{remark}

\subsubsection{Discrete-to-continuum nonlocal rates} 

Here we loosely follow \cite[Section 6.2.2]{ElBouchairi}. By using \cite[Lemma 6.1]{ElBouchairi}, we easily check that the time interpolated version of the injected discrete problem satisfies an evolution problem. The proof can be found in Section \ref{sec:supplementary:rates}.

\begin{lemma}[Evolution problem for $\uTimeInter$] 
\label{lem:proofs:rates:NonFullyDiscreteLocal:evolutionTimeInterpolation}

Assume that Assumptions \ref{ass:main:assumptions:S1}, \ref{ass:main:assumptions:O1}, \ref{ass:main:assumptions:O2}, \ref{ass:main:assumptions:O3}, \ref{ass:main:assumptions:O4} and \ref{ass:main:assumptions:K1} hold. Let $p \geq 2$, $\mu > 0$, $T>0$, $u_0 \in \Lp{p}(\Omega)$, $\ell \in \Lp{2}(\Omega)$ and assume that $\cA^* \ell \in \Lp{p}(\Omega)$. Furthermore, let $n \in \bbN$ and define $\bar{K} = \cP_n \tilde{K}$, $\bar{f} = \cP_n \cA^*\ell $, $\bar{u}_0 = \cP_n u_0$.
Then, for any partition $0 = t^0 < t^1 < \dots < t^N = T$, the sequence $\{\bar{u}_n^k\}_{k=0}^N$ is unique and well-defined by \eqref{eq:main:notation:discrete:nonlocalProblemFully} with the above parameters. Furthermore, $\uTimeInter$ solves the following evolution problem:
\[
\begin{cases}
\frac{\partial}{\partial t } \uTimeInter + \mu\Delta_{p}^{\cI_n \bar{K}}(\uTimeInject) + \cA^*\cA(\uTimeInject) = \cI_n \bar{f} &\text{ in } (0,T) \times \Omega \\
\uTimeInter(0,\cdot) = \cI_n \bar{u}_0.
\end{cases}
\]
\end{lemma}

While our final results in Corollaries \ref{cor:proofs:rates:discreteNonlocalContinuumLocal:simplified} and \ref{cor:proofs:rates:discreteNonlocalContinuum:final} will be concerned with letting $T\to \infty$, $\tau_n \to 0$ and $\eps_n \to 0$, the next proposition presents more general non-asymptotic results that are valid for any kernel $K$, time-discretization $\tau$ and time $T$.

\begin{proposition}[Discrete-to-continuum nonlocal rates] 
\label{prop:proofs:rates:NonFullyDiscreteLocal:L2rates}

Assume that Assumptions \ref{ass:main:assumptions:S1}, \ref{ass:main:assumptions:O1}, \ref{ass:main:assumptions:O2}, \ref{ass:main:assumptions:O3}, \ref{ass:main:assumptions:O4} and \ref{ass:main:assumptions:K1} hold. Let $p \geq 2$, $\mu > 0$, $T>0$, $u_0 \in \Lp{p}(\Omega)$, $\ell \in \Lp{2}(\Omega)$ and assume that $\cA^* \ell \in \Lp{p}(\Omega)$. Furthermore, let $n \in \bbN$ and $\bar{K} = \cP_n \tilde{K}$, $\bar{f} = \cP_n \cA^*\ell $, $\bar{u}_0 = \cP_n u_0$.
Then, for any partition $0 = t^0 < t^1 < \dots < t^N = T$, there exists a sequence $\{\bar{u}_n^k\}_{k=0}^N$ satisfying \eqref{eq:main:notation:discrete:nonlocalProblemFully} with the above parameters. Furthermore, there exists a solution $u_K$ to \eqref{eq:main:notation:nonlocal:nonlocalProblem}. We also have the following rates for some $C > 0$ dependent on $\Omega$, $u_0$ and $\cA^*\ell$:
\begin{enumerate}
\vspace{-2mm}
    \setlength\itemsep{-1mm}
    \item if $u_0 \in \Lp{2p- 2/(p-1)}(\Omega)$ and $\cA^* \ell \in \Lp{2p- 2/(p-1)}(\Omega)$:
    \begin{align}
        &\sup_{1 \leq k \leq N } \sup_{t \in (t^{k-1},t^k]} \Vert \cI_n \bar{u}_n^k - u_K(t,\cdot) \Vert_{\Lp{2}} \leq C e^{\l 1 +  C^4_{\mathrm{op}} \r T} \Biggl( \Vert \cI_n \cP_n u_0 - u_0 \Vert_{\Lp{2}} + \Vert \cI_n \cP_n \cA^*\ell - \cA^*\ell \Vert_{\Lp{2}} \notag\\
        &+ (1 + T^{p-1}) \sup_{x \in \Omega} \Vert \cI_n \cP_n K(\vert x - \cdot \vert) - K(\vert x - \cdot \vert) \Vert_{\Lp{2}}  +  \tau (1 + \Vert K \Vert_{\Lp{\infty}}) (1 + T^{p-1}) \label{eq:proofs:rates:NonFullyDiscreteLocal:L2rates:LPrates}\\
    &+    \tau^{p/(2p-1)} \Vert K \Vert_{\Lp{\infty}}^{p/(2p-1)} ( 1 + \Vert K \Vert_{\Lp{\infty}})^{p/(2p-1)} (1 + T^{p-1})^{p/(2p-1)} ( 1 + T^{p-1 - 1/p})^{p/(2p-1)} \notag\\
    &+  \tau^{(p+1)/(2p)} \Vert K \Vert_{\Lp{\infty}}^{1/2} \l 1 + T^{p-1 - 1/p}  \r^{1/2} ( 1 + \Vert K \Vert_{\Lp{\infty}})^{(p+1)/(2p)} (1 + T^{p-1})^{(p+1)/(2p)} \Biggr); \notag
    \end{align}
    \item if $u_0 \in \Lp{\infty}(\Omega)$ and $\cA^* \ell \in \Lp{\infty}(\Omega)$:
    \begin{align}
        \sup_{1 \leq k \leq N } &\sup_{t \in (t^{k-1},t^k]} \Vert \cI_n \bar{u}_n^k - u_K(t,\cdot) \Vert_{\Lp{2}} \leq C e^{ \l 1 +  C^4_{\mathrm{op}} \r T} \Biggl( \Vert \cI_n \cP_n u_0 - u_0 \Vert_{\Lp{2}} + \Vert \cI_n \cP_n \cA^*\ell - \cA^*\ell \Vert_{\Lp{2}} \notag \\
        &+ (1 + T^{p-1}) \Vert \cI_n \cP_n \tilde{K} - \tilde{K} \Vert_{\Lp{2}(\Omega \times \Omega)} \label{eq:proofs:rates:NonFullyDiscreteLocal:L2rates:Linftyrates} \\
        &+ \tau (1 + \Vert K \Vert_{\Lp{\infty}}) (1 + T^{p-1}) \ls 1 + \Vert K \Vert_{\infty}(1+T^{p-2}) + \l \Vert K \Vert_{\Lp{\infty}} (1 + T^{p-2}) \r^{1/2} \rs \biggr). \notag
    \end{align}
\end{enumerate}
\end{proposition}

\begin{proof}
In the proof $C>0$ will denote a constant that can be arbitrarily large, (which might be) dependent on $\Omega$, $u_0$ or/and $\cA^*\ell$, that may change from line to line.

The existence and well-posedness of $\{\bar{u}_n^k\}_{k=0}^N$ and $u_K$ follow from Corollary \ref{cor:proofs:wellPosedness:discreteProblem:existence} and Theorem \ref{thm:proofs:wellPosedness:nonlocalProblem:existenceUniqueness} respectively. Also, note that for $p \geq 2$, we have $2p - 2/(p-1) \geq 2(p-1) \geq p$. 
Let $\alpha_i > 0$ for $1 \leq i \leq 3$ be such that $\alpha_1 + \mu (\alpha_2 + \alpha_3) = \frac{1}{2}$.  

We start by noticing that for $t \in (t^{k-1},t^k]$, we have
\[
\Vert \uTimeInter(t,\cdot) \Vert_{\Lp{p}} \leq \frac{t^k - t}{\tau^{k-1}} \Vert \cI_n \bar{u}_n^{k-1} \Vert_{\Lp{p}} + \frac{t-t^{k-1}}{\tau^{k-1}} \Vert \cI_n \bar{u}^{k}_n \Vert_{\Lp{p}} \leq C \l \Vert u_0 \Vert_{\Lp{p}} + T \Vert \cA^*\ell \Vert_{\Lp{p}} \r   < C(C + TC)
\]
where we used \eqref{eq:proofs:wellPosedness:discreteProblem:uniformBound} for the second inequality and also:
\begin{equation} \label{eq:proofs:rates:nonFullyDiscreteLocal:L2Rates:uTimeInjectBound}
\Vert \uTimeInject(t,\cdot) \Vert_{\Lp{p}} = \Vert \cI_n \bar{u}_n^{k-1} \Vert_{\Lp{p}} \leq \Vert u_0 \Vert_{\Lp{p}} + T \Vert \cA^*\ell \Vert_{\Lp{p}}   < C+TC
\end{equation}
again by \eqref{eq:proofs:wellPosedness:discreteProblem:uniformBound}. This, together with $\Vert \cI_n \bar{u}_0 \Vert_{\Lp{2}} \leq \Vert u_0 \Vert_{\Lp{p}}$ by \cite[Lemma 2.1]{ElBouchairi}, implies that for any $0 \leq t < T$, we have that $\uTimeInject(t,\cdot),\uTimeInter(t,\cdot) \in \Lp{p}(\Omega)$ uniformly in $t$.

Define $\zeta_{\text{TimeInt}}(t,x) = \uTimeInter(t,x) - u_K(t,x)$ and compute as follows:
\begin{align}
    \frac{1}{2}\frac{\partial}{\partial t} \Vert \zeta_{\text{TimeInt}}(t,\cdot) \Vert_{\Lp{2}}^2 &= - \mu \int_{\Omega}  \l \Delta_{p}^{\cI_n \bar{K}} \uTimeInject(t,x) - \Delta_p^{\cI_n \bar{K}} u_K(t,x) \r (\uTimeInject(t,x) - u_K(t,x) ) \, \dd x \notag \\
    &- \mu \int_{\Omega}  \l \Delta_{p}^{\cI_n \bar{K}} \uTimeInject(t,x) - \Delta_p^{\cI_n \bar{K}} u_K(t,x) \r (\uTimeInter(t,x) - \uTimeInject(t,x) ) \, \dd x \notag \\
    &- \mu \int_{\Omega}  \l \Delta_{p}^{\cI_n \bar{K}} u_K(t,x) - \Delta_p^{K} u_K(t,x) \r \zeta_{\text{TimeInt}}(t,x) \, \dd x \label{eq:proofs:rates:NonFullyDiscreteLocal:L2Rates:equation1}\\
    &- \int_{\Omega}  \cA^*\cA( \uTimeInject(t,x) - u_K(t,x) ) \zeta_{\text{TimeInt}}(t,x) \, \dd x  \notag \\
    &+ \int_{\Omega}  (\cI_n \bar{f}(x) - \cA^*\ell(x)) \zeta_{\text{TimeInt}}(t,x) \, \dd x \notag \\
    &=: T_1 + \mu T_2 + \mu T_3 + T_4 + T_5 \notag
\end{align}
where we used Lemma \ref{lem:proofs:rates:NonFullyDiscreteLocal:evolutionTimeInterpolation}, \eqref{eq:main:notation:nonlocal:nonlocalProblem} and the fact that $\uTimeInject(t,\cdot),\uTimeInter(t,\cdot) \in \Lp{p}(\Omega)$ for \eqref{eq:proofs:rates:NonFullyDiscreteLocal:L2Rates:equation1}. Arguing as in Proposition \ref{prop:proofs:wellPosedness:nonlocalProblem:completeAccretivityRangeCondition} or Proposition \ref{prop:proofs:nonlocalDiscreteContinuum:stability} (which relies on \cite[Lemma 2.3]{ANDREU2008201}), we obtain that $T_1 \leq 0$. Furthermore, by Young's inequality for products,
\begin{equation}\label{eq:proofs:rates:NonFullyDiscreteLocal:L2Rates:T5}
    T_5 \leq C \Vert \cI_n \cP_n \cA^*\ell - \cA^*\ell \Vert_{2}^2 + \alpha_1 \Vert \zeta_{\text{TimeInt}}(t) \Vert_{\Lp{2}}^2.  
\end{equation}

We continue our estimates by showing some auxiliary results first. For $t \in (t^{k-1},t^k]$:
\begin{align}
    \Vert \uTimeInter(t,\cdot) - \uTimeInject(t,\cdot) \Vert_{\Lp{2}} &= \vert t^k - t \vert \left\Vert \frac{\cI_n \bar{u}_n^{k} - \cI_n \bar{u}_n^{k-1}}{\tau^{k-1}} \right\Vert_{\Lp{2}} \notag \\
    &= \vert t^k - t \vert \left\Vert \mu \Delta_p^{\cI_n \bar{K}} \cI_n \bar{u}_n^k + \cA^*\cA \cI_n \bar{u}_n^k - \cI_n \bar{f} \right\Vert_{\Lp{2}} \label{eq:proofs:rates:NonFullyDiscreteLocal:L2Rates:resolvent} \\
    &\leq \tau \ls \mu \Vert \Delta_p^{\cI_n \bar{K}} \cI_n \bar{u}_n^k \Vert_{\Lp{2}} + \Vert  \cA^*\cA \cI_n \bar{u}_n^k \Vert_{\Lp{2}}  + \Vert  \cI_n \bar{f} \Vert_{\Lp{2}} \rs \label{eq:proofs:rates:NonFullyDiscreteLocal:L2Rates:equation2}
\end{align}
where we used the proof of Corollary \ref{cor:proofs:wellPosedness:discreteProblem:existence} for \eqref{eq:proofs:rates:NonFullyDiscreteLocal:L2Rates:resolvent}. Now, 
by Assumption \ref{ass:main:assumptions:O1} and 
\eqref{eq:proofs:rates:nonFullyDiscreteLocal:L2Rates:uTimeInjectBound}
we have that
\[
\Vert  \cA^*\cA \cI_n \bar{u}_n^k \Vert_{\Lp{2}} =  \Vert  \cA^*\cA \uTimeInject(t,\cdot) \Vert_{\Lp{2}} \leq C \Vert \uTimeInject(t,\cdot) \Vert_{\Lp{2}} \leq C \Vert \uTimeInject(t,\cdot) \Vert_{\Lp{p}} \leq C(C + TC)
\]
as well as 
$\Vert  \cI_n \bar{f} \Vert_{\Lp{2}} \leq \Vert \cA^* \ell \Vert_{\Lp{2}} \leq C \Vert \cA^*\ell \Vert_{\Lp{p}} \leq C$
by \cite[Lemma 2.1]{ElBouchairi}. Furthermore, re-using \cite[Lemma 2.1]{ElBouchairi} and Assumption \ref{ass:main:assumptions:S1} we have
\begin{align}
    \Vert \Delta_p^{\cI_n \bar{K}} \cI_n \bar{u}_n^k \Vert_{\Lp{2}} &= \Vert \Delta_p^{\cI_n \bar{K}} \uTimeInject(t,\cdot) \Vert_{\Lp{2}} \leq C \Vert K \Vert_{\Lp{\infty}} \Vert \uTimeInject(t,\cdot) \Vert_{\Lp{2(p-1)}}^{p-1}
    \label{eq:proofs:rates:NonFullyDiscreteLocal:L2Rates:equation3}
\end{align}
Now, in either case, $u_0 \in \Lp{2p - 2/(p-1)}(\Omega)$ and $\cA^*\ell \in \Lp{2p - 2/(p-1)}(\Omega)$ or $u_0 \in \Lp{\infty}(\Omega)$ and $\cA^*\ell \in \Lp{\infty}(\Omega)$,  we get from \eqref{eq:proofs:wellPosedness:discreteProblem:uniformBound} that 
$\Vert \uTimeInject(t,\cdot) \Vert_{\Lp{2(p-1)}}^{p-1} \leq C + T^{p-1} C$
as was shown in \eqref{eq:proofs:rates:nonFullyDiscreteLocal:L2Rates:uTimeInjectBound}. Injecting the above observation in \eqref{eq:proofs:rates:NonFullyDiscreteLocal:L2Rates:equation3} and then starting from \eqref{eq:proofs:rates:NonFullyDiscreteLocal:L2Rates:equation2}, we obtain that:
\begin{align} 
    \Vert \uTimeInter(t,\cdot) - \uTimeInject(t,\cdot) \Vert_{\Lp{2}}
    &\leq \tau (C + TC + \Vert K \Vert_{\Lp{\infty}} (C + T^{p-1}C))
    \label{eq:proofs:rates:NonFullyDiscreteLocal:L2Rates:differenceUs}
\end{align}

For the $T_4$ term, by Young's inequality of products: 
\begin{align}
\vert T_4 \vert &\leq \frac{1}{2} \Vert \zeta_{\text{TimeInt}}(t) \Vert_{\Lp{2}}^2 + \frac{1}{2} \Vert \cA^* \cA (\uTimeInject(t) - u_K(t)) \Vert_{\Lp{2}}^2 \notag \\
&\leq \frac{1}{2} \Vert \zeta_{\text{TimeInt}}(t) \Vert_{\Lp{2}}^2 + \frac{C^4_{\mathrm{op}}}{2} \Vert \uTimeInject(t) - \uTimeInter + \uTimeInter - u_K(t) \Vert_{\Lp{2}}^2 \label{eq:proofs:rates:NonFullyDiscreteLocal:L2Rates:T4:2} \\
&\leq \l \frac{1}{2} + C^4_{\mathrm{op}} \r \Vert \zeta_{\text{TimeInt}}(t) \Vert_{\Lp{2}}^2 + C \Vert \uTimeInject(t) - \uTimeInter \Vert_{\Lp{2}}^2 \notag \\
&\leq \l \frac{1}{2} + C^4_{\mathrm{op}} \r \Vert \zeta_{\text{TimeInt}}(t) \Vert_{\Lp{2}}^2 + \ls \tau (C + TC + \Vert K \Vert_{\Lp{\infty}} (C + T^{p-1}C)) \rs^2 \label{eq:proofs:rates:NonFullyDiscreteLocal:L2Rates:T4:4}
\end{align}
where we used Assumption \ref{ass:main:assumptions:O1}
for \eqref{eq:proofs:rates:NonFullyDiscreteLocal:L2Rates:T4:2} and \eqref{eq:proofs:rates:NonFullyDiscreteLocal:L2Rates:differenceUs} for \eqref{eq:proofs:rates:NonFullyDiscreteLocal:L2Rates:T4:4}.

We will now tackle the $T_2$ and $T_3$ terms. First, assume that $u_0, \, \cA^*\ell \in \Lp{2p - 2/(p-1)}(\Omega)$.
We want to estimate 
$T_6 := \Vert \Delta_{p}^{\cI_n \bar{K}} \uTimeInject(t,\cdot) - \Delta_p^{\cI_n \bar{K}} u_K(t,\cdot) \Vert_{\Lp{2}}$:
\begin{align}
    T_6^2 &\leq C\Vert \cI_n \bar{K} \Vert_{\Lp{\infty}}^2 \int_\Omega \int_{\Omega} \bigl\vert \vert \uTimeInject(y) - \uTimeInject(x)\vert^{p-2}(\uTimeInject(y) - \uTimeInject(x)) \notag \\
    & \qquad \qquad - \vert u_K(y) - u_K(x) \vert^{p-2}(u_K(y) - u_K(x))  \bigr\vert^2 \, \dd y \dd x \notag \\
    &\leq C\Vert K \Vert_{\Lp{\infty}}^2  \int_\Omega \int_{\Omega} \vert  \uTimeInject(y) - \uTimeInject(x)  - u_K(y) + u_K(x)  \vert^{2/p} \notag \\
    &\qquad \qquad \times ( \vert \uTimeInject(y) - \uTimeInject(x) \vert + \vert u_K(y) - u_K(x) \vert )^{2(p-1) - 2/p} \, \dd y \dd x \label{eq:proofs:rates:NonFullyDiscreteLocal:L2Rates:continuity} \\
    &\leq C\Vert K \Vert_{\Lp{\infty}}^2 \ls \int_\Omega \int_\Omega \vert  \uTimeInject(y) - \uTimeInject(x)  - u_K(y) + u_K(x)  \vert^{2} \, \dd y \dd x \rs^{1/p} \notag\\
    & \qquad \qquad \times \ls \int_\Omega \int_\Omega ( \vert \uTimeInject(y) - \uTimeInject(x) \vert + \vert u_K(y) - u_K(x) \vert )^{2p - 2/(p-1)} \, \dd y \dd x \rs^{1/q} \label{eq:proofs:rates:NonFullyDiscreteLocal:L2Rates:holder}\\
    &\leq C \Vert K \Vert_{\Lp{\infty}}^2 \ls \int_\Omega  \vert  \uTimeInject(x)  - u_K(x) \vert^{2} \, \dd x \rs^{1/p} \notag \\
    & \qquad \qquad \times \ls \int_\Omega \int_\Omega ( \vert \uTimeInject(y) \vert + \vert \uTimeInject(x) \vert + \vert u_K(y) \vert  + \vert u_K(x) \vert )^{2p - 2/(p-1)} \, \dd y \dd x \rs^{1/q} \label{eq:proofs:rates:NonFullyDiscreteLocal:L2Rates:domain1} \\
    &\leq C \Vert K \Vert_{\Lp{\infty}}^2 \Vert \uTimeInject(t,\cdot)  - u_K(t,\cdot) \Vert_{\Lp{2}}^{2/p} \Vert \vert \uTimeInject(t,\cdot) \vert + \vert u_K(t,\cdot) \vert \Vert_{\Lp{2p - 2/(p-1)}}^{2(p-1) - 2/p} \label{eq:proofs:rates:NonFullyDiscreteLocal:L2Rates:domain2} \\
    &\leq C \Vert K \Vert_{\Lp{\infty}}^2 \Vert \uTimeInject(t,\cdot)  - u_K(t,\cdot) \Vert_{\Lp{2}}^{2/p} \l C + T^{2(p-1) - 2/p} C  \r \label{eq:proofs:rates:NonFullyDiscreteLocal:L2Rates:contraction}
\end{align}
where we used \cite[Lemma 4.1, (ii)]{ElBouchairi} with $\alpha = 1/p$ for \eqref{eq:proofs:rates:NonFullyDiscreteLocal:L2Rates:continuity}, H\"older's inequality for \eqref{eq:proofs:rates:NonFullyDiscreteLocal:L2Rates:holder}, Assumption \ref{ass:main:assumptions:S1} for \eqref{eq:proofs:rates:NonFullyDiscreteLocal:L2Rates:domain1} and \eqref{eq:proofs:rates:NonFullyDiscreteLocal:L2Rates:domain2}, and \eqref{eq:proofs:wellPosedness:discreteProblem:uniformBound} as well as \eqref{eq:proofs:wellPosedness:nonlocalProblem:existenceUniqueness:bound} with $r = 2p - 2/(p-1)$ for \eqref{eq:proofs:rates:NonFullyDiscreteLocal:L2Rates:contraction}.

We now return to $T_2$ and estimate as follows:
\begin{align}
    \vert T_2 \vert &\leq \Vert \Delta_{p}^{\cI_n \bar{K}} \uTimeInject(t,\cdot) - \Delta_p^{\cI_n \bar{K}} u_K(t,\cdot) \Vert_{\Lp{2}}  \Vert \uTimeInter(t,\cdot) - \uTimeInject(t,\cdot) \Vert_{\Lp{2}} \notag \\
    &\leq \tau \Vert K \Vert_{\Lp{\infty}} \l C + TC + \Vert K \Vert_{\Lp{\infty}} (C + T^{p-1}C) \r \l C + T^{(p-1) - 1/p} C \r \Vert \uTimeInject(t,\cdot)  - u_K(t,\cdot) \Vert_{\Lp{2}}^{1/p}  \label{eq:proofs:rates:NonFullyDiscreteLocal:L2Rates:equation4} \\
    &\leq \tau \Vert K \Vert_{\Lp{\infty}} \l C + TC + \Vert K \Vert_{\Lp{\infty}} (C + T^{p-1}C) \r \l C + T^{(p-1) - 1/p} C \r \biggl[ \Vert \uTimeInter(t,\cdot)  - u_K(t,\cdot)  \Vert_{\Lp{2}}^{1/p} \notag \\ 
    &+  \l \tau (C + TC + \Vert K \Vert_{\Lp{\infty}} (C + T^{p-1}C)) \r^{1/p} \biggr] \label{eq:proofs:rates:NonFullyDiscreteLocal:L2Rates:equation5} \\
    &\leq \alpha_2 \Vert \zeta_{\text{TimeInt}}(t)  \Vert_{\Lp{2}}^{2} +   \ls   \tau \Vert K \Vert_{\Lp{\infty}} \l C + TC + \Vert K \Vert_{\Lp{\infty}} (C + T^{p-1}C) \r \l C + T^{(p-1) - 1/p} C \r \rs^{2p/(2p-1)} \notag \\
    &+   \ls \tau \l \Vert K \Vert_{\Lp{\infty}}\l C + T^{(p-1) - 1/p} C \r \r ^{p/(p+1)} \l C + TC + \Vert K \Vert_{\Lp{\infty}} (C + T^{p-1}C) \r \rs^{(p+1)/p} \label{eq:proofs:rates:NonFullyDiscreteLocal:L2Rates:T2}
\end{align}
where we used \eqref{eq:proofs:rates:NonFullyDiscreteLocal:L2Rates:differenceUs} and \eqref{eq:proofs:rates:NonFullyDiscreteLocal:L2Rates:contraction} for \eqref{eq:proofs:rates:NonFullyDiscreteLocal:L2Rates:equation4}, \eqref{eq:proofs:rates:NonFullyDiscreteLocal:L2Rates:differenceUs} again for \eqref{eq:proofs:rates:NonFullyDiscreteLocal:L2Rates:equation5} and Young's inequality for products for \eqref{eq:proofs:rates:NonFullyDiscreteLocal:L2Rates:T2}.

For $T_3$, relying on the fact that for $p \geq 2$ we have $2p - 2/(p-1) \geq 2(p-1)$, we proceed as in Proposition \ref{prop:proofs:nonlocalDiscreteContinuum:stability} to obtain \eqref{eq:proofs:nonlocalDiscreteContinuum:stability:contraction2}:
\begin{align}
    \vert T_3 \vert &\leq C(C + T^{p-1}C) \ls \sup_{x \in \Omega} \Vert \cI_n \bar{K}(\vert x - \cdot \vert) - K(\vert x - \cdot \vert) \Vert_{\Lp{2}} \rs \Vert \zeta_{\text{TimeInt}}(t) \Vert_{\Lp{2}} \notag \\
    &\leq (C + T^{2(p-1)} C)  \sup_{x \in \Omega} \Vert \cI_n \bar{K}(\vert x - \cdot \vert) - K(\vert x - \cdot \vert) \Vert_{\Lp{2}}^2 + \alpha_3 \Vert \zeta_{\text{TimeInt}}(t) \Vert_{\Lp{2}}^2 \label{eq:proofs:rates:NonFullyDiscreteLocal:L2Rates:T3}
\end{align}
using Young's inequality for products for \eqref{eq:proofs:rates:NonFullyDiscreteLocal:L2Rates:T3}.

Combining \eqref{eq:proofs:rates:NonFullyDiscreteLocal:L2Rates:T2}, \eqref{eq:proofs:rates:NonFullyDiscreteLocal:L2Rates:T3}, \eqref{eq:proofs:rates:NonFullyDiscreteLocal:L2Rates:T4:4} and \eqref{eq:proofs:rates:NonFullyDiscreteLocal:L2Rates:T5}, we obtain:
\begin{align}
    \frac{\partial}{\partial t} \Vert \zeta_{\text{TimeInt}}(t) \Vert_{\Lp{2}}^2 &\leq 2 \l \alpha_1 + \mu(\alpha_2 + \alpha_3) + \frac{1}{2} + C^4_{\mathrm{op}} \r \Vert \zeta_{\text{TimeInt}}(t)  \Vert_{\Lp{2}}^{2} \notag \\
    &+ \ls \tau (C + TC + \Vert K \Vert_{\Lp{\infty}} (C + T^{p-1}C)) \rs^2 \notag \\
    &+  \ls   \tau \Vert K \Vert_{\Lp{\infty}} \l C + TC + \Vert K \Vert_{\Lp{\infty}} (C + T^{p-1}C) \r \l C + T^{(p-1) - 1/p} C \r \rs^{2p/(2p-1)} \notag \\
    &+ \ls \tau \l \Vert K \Vert_{\Lp{\infty}}\l C + T^{(p-1) - 1/p} C \r \r ^{p/(p+1)} \l C + TC + \Vert K \Vert_{\Lp{\infty}} (C + T^{p-1}C) \r \rs^{(p+1)/p} \notag \\
    &+ (C + T^{2(p-1)}C) \sup_{x \in \Omega} \Vert \cI_n \bar{K}(\vert x - \cdot \vert) - K(\vert x - \cdot \vert) \Vert_{\Lp{2}}^2 + C \Vert \cI_n \cP_n \cA^*\ell - \cA^*\ell \Vert_{\Lp{2}}^2  \notag \\
    &=:  \l 2 +  2C^4_{\mathrm{op}} \r \Vert \zeta_{\text{TimeInt}}(t) \Vert_{\Lp{2}}^2 + T_7. \notag
\end{align}
We continue by applying Gronwall's lemma on the latter to deduce:
\begin{equation} \label{eq:proofs:rates:NonFullyDiscreteLocal:L2Rates:gronwall}
    \Vert \zeta_{\text{TimeInt}}(t) \Vert_{\Lp{2}} \leq e^{\l 1 +  C^4_{\mathrm{op}} \r T} \l \Vert \cI_n \cP_n u_0 - u_0 \Vert_{\Lp{2}} + C \cdot T_7^{1/2} \r.
\end{equation}

We conclude that
\begin{align}
    \sup_{1 \leq k \leq N } \sup_{t \in (t^{k-1},t^k]} \Vert \cI_n \bar{u}_n
    ^k - u_K(t,\cdot) \Vert_{\Lp{2}} &= \sup_{0 < t \leq T} \Vert \uTimeInject(t,\cdot) - u_K(t,\cdot) \Vert_{\Lp{2}} \notag \\
    &\leq \sup_{0 < t \leq T} \Vert \uTimeInter - u_K(t,\cdot) \Vert_{\Lp{2}} + \sup_{0 < t \leq T} \Vert \uTimeInject - \uTimeInter \Vert_{\Lp{2}} \notag \\
    &\leq e^{ \l 1 +  C^4_{\mathrm{op}} \r T} \l \Vert \cI_n \cP_n u_0 - u_0 \Vert_{\Lp{2}} + C \cdot T_7^{1/2} \r \notag \\ 
    &+ \tau \l C + TC + \Vert K \Vert_{\Lp{\infty}} (C + T^{p-1}C) \r \label{eq:proofs:rates:NonFullyDiscreteLocal:L2Rates:rates}
\end{align}
where we used \eqref{eq:proofs:rates:NonFullyDiscreteLocal:L2Rates:gronwall} and \eqref{eq:proofs:rates:NonFullyDiscreteLocal:L2Rates:differenceUs} for \eqref{eq:proofs:rates:NonFullyDiscreteLocal:L2Rates:rates}.

Let us now assume that $u_0 \in \Lp{\infty}(\Omega)$ and $\cA^*\ell \in \Lp{\infty}(\infty)$. We will slightly change the estimates for $T_6$ and $T_3$.
\begin{align}
    T_6^2
    &\leq C\Vert K \Vert_{\Lp{\infty}}^2  \int_\Omega \vert \int_{\Omega}  \vert \uTimeInject(y) - \uTimeInject(x)  - u_K(y) + u_K(x)  \vert \notag \\
    & \quad \quad \quad \quad \quad \quad \quad \times ( \vert \uTimeInject(y) - \uTimeInject(x) \vert + \vert u_K(y) - u_K(x) \vert )^{p-2} \, \dd y \vert^2 \dd x \label{eq:proofs:rates:NonFullyDiscreteLocal:L2Rates:continuity2} \\
    &\leq C \Vert K \Vert_{\Lp{\infty}}^2 \l \Vert u_0 \Vert_{\Lp{\infty}} + T \Vert \cA^*\ell \Vert_{\Lp{\infty}} \r^{2(p-2)} \notag \\
    &\times \int_\Omega \int_{\Omega}  \vert \uTimeInject(y) - \uTimeInject(x)  - u_K(y) + u_K(x)  \vert^2 \, \dd y \dd x \label{eq:proofs:rates:NonFullyDiscreteLocal:L2Rates:contraction2} \\
    &\leq (C + T^{2(p-2)}C) \Vert K \Vert_{\Lp{\infty}}^2 \Vert \uTimeInject(t,\cdot) - u_K(t,\cdot) \Vert_{\Lp{2}}^2 \label{eq:proofs:rates:NonFullyDiscreteLocal:L2Rates:domain3}
\end{align}
where we used \cite[Lemma 4.1 (ii)]{ElBouchairi} with $\alpha = 1$ and \cite[Lemma 2.1]{ElBouchairi} for \eqref{eq:proofs:rates:NonFullyDiscreteLocal:L2Rates:continuity2}, \eqref{eq:proofs:wellPosedness:nonlocalProblem:existenceUniqueness:bound} and \eqref{eq:proofs:wellPosedness:discreteProblem:uniformBound} for \eqref{eq:proofs:rates:NonFullyDiscreteLocal:L2Rates:contraction2} and Assumption \ref{ass:main:assumptions:S1} for \eqref{eq:proofs:rates:NonFullyDiscreteLocal:L2Rates:domain3}.
We can then return to $T_2$ and estimate as follows:
\begin{align}
    \vert T_2 \vert &\leq \tau \Vert K \Vert_{\Lp{\infty}} \l C + TC + \Vert K \Vert_{\Lp{\infty}} (C + T^{p-1}C) \r (C + T^{(p-2)}C) \Vert \uTimeInject(t,\cdot) - u_K(t,\cdot) \Vert_{\Lp{2}} \label{eq:proofs:rates:NonFullyDiscreteLocal:L2Rates:equation6} \\
    &\leq \tau \Vert K \Vert_{\Lp{\infty}} \l C + TC + \Vert K \Vert_{\Lp{\infty}} (C + T^{p-1}C) \r (C + T^{(p-2)}C) \biggl[ \Vert \uTimeInter(t,\cdot) - u_K(t,\cdot) \Vert_{\Lp{2}} \notag \\
    &+ \tau (C + TC + \Vert K \Vert_{\Lp{\infty}} (C + T^{p-1}C)) \biggr] \label{eq:proofs:rates:NonFullyDiscreteLocal:L2Rates:equation7} \\
    &\leq \alpha_2 \Vert \uTimeInter(t,\cdot) - u_K(t,\cdot) \Vert_{\Lp{2}}^2 + \ls \tau \Vert K \Vert_{\Lp{\infty}} \l C + TC + \Vert K \Vert_{\Lp{\infty}} (C + T^{p-1}C) \r (C + T^{(p-2)}C) \rs^2 \notag \\
    &+ \ls \l \Vert K \Vert_{\Lp{\infty}} (C + T^{(p-2)}C) \r^{1/2} \tau \l C + TC + \Vert K \Vert_{\Lp{\infty}} (C + T^{p-1}C) \r \rs^{2} \label{eq:proofs:rates:NonFullyDiscreteLocal:L2Rates:T22}
\end{align}
where we used \eqref{eq:proofs:rates:NonFullyDiscreteLocal:L2Rates:differenceUs} and \eqref{eq:proofs:rates:NonFullyDiscreteLocal:L2Rates:domain3} for \eqref{eq:proofs:rates:NonFullyDiscreteLocal:L2Rates:equation6}, \eqref{eq:proofs:rates:NonFullyDiscreteLocal:L2Rates:differenceUs} again for \eqref{eq:proofs:rates:NonFullyDiscreteLocal:L2Rates:equation7} and Young's inequality for products for \eqref{eq:proofs:rates:NonFullyDiscreteLocal:L2Rates:T22}.

For $T_3$, we proceed as in Proposition \ref{prop:proofs:nonlocalDiscreteContinuum:stability} in order to obtain \eqref{eq:proofs:nonlocalDiscreteContinuum:stability:Linfty}:
\begin{align}
    \vert T_3 \vert &\leq C(C + T^{p-1}C) \Vert \cI_n \bar{K} - \tilde{K} \Vert_{\Lp{2}(\Omega \times \Omega)} \Vert \zeta_{\text{TimeInt}}(t) \Vert_{\Lp{2}}\notag \\
    &\leq (C + T^{2(p-1)}C) \Vert \cI_n \bar{K} - \tilde{K} \Vert_{\Lp{2}(\Omega \times \Omega)}^2 + \alpha_3 \Vert \zeta_{\text{TimeInt}}(t) \Vert_{\Lp{2}}^2 \label{eq:proofs:rates:NonFullyDiscreteLocal:L2Rates:T32}
\end{align}
where we used Young's inequality for products for \eqref{eq:proofs:rates:NonFullyDiscreteLocal:L2Rates:T32}. We proceed as above to conclude.
\end{proof}

\subsubsection{Discrete-to-continuum local rates} \label{subsec:ratesDiscreteLocal}

In this section we will derive rates between the fully discrete problem and the continuum problem. In particular, by combining the results of Proposition \ref{prop:proofs:rates:NonFullyDiscreteLocal:L2rates} and Theorem \ref{thm:proofs:continuumRates:continuumNonlocalLocal}, one obtains general rates for fixed $T>0$ and several classes of $u_0$, $\cA^*\ell$. 

Of greater interest is the question of the possibility of letting $T \to \infty$, as the solution of the gradient flow \eqref{eq:main:notation:localProblem:localProblem} solved for large $T$ converges to the minimizer of the original regularization problem \eqref{eq:intro:regularizationProblem}. Corollary \ref{cor:proofs:rates:discreteNonlocalContinuumLocal:simplified} answers that question in a positive way: indeed, this can be achieved by correctly choosing the partition, the functions $u_0$, $\cA^* \ell$ and the kernel $K$ as well as indexing the time $T(n)$ in a meaningful way and imposing conditions between the time and space discretization parameters $\tau_n$ and $\eps_n$.

\begin{proof}[Proof of Corollary \ref{cor:proofs:rates:discreteNonlocalContinuumLocal:simplified}]
In the proof $C>0$ will denote a constant that can be arbitrarily large, (which might be) dependent on $\Omega$, $u_0$ or/and $\cA^*\ell$, that may change from line to line. We also briefly comment on some notation: $u_{\eps_n}$ in \eqref{eq:proofs:continuumRates:l2convergence:rates} is a short-hand for $u_{K_{\eps_n}}$ appearing in \eqref{eq:proofs:rates:NonFullyDiscreteLocal:L2rates:Linftyrates} when using the kernel $K_{\eps_n}$.

We have 
\begin{equation}\label{eq:application:ratesDiscreteDeterministic:kHatBound}
    \Vert \cI_n \bar{K}_{\eps_n} \Vert_{\Lp{\infty}} \leq \Vert \tilde{K}_{\eps_n} \Vert_{\Lp{\infty}}  \leq \eps_{n}^{-(d+p)} \Vert K \Vert_{\Lp{\infty}}
\end{equation}
by \cite[Lemma 2.1]{ElBouchairi}. The latter is finite by Assumption \ref{ass:main:assumptions:K1} and the existence claims then follow from Proposition \ref{prop:proofs:rates:NonFullyDiscreteLocal:L2rates} and Theorem \ref{thm:proofs:continuumRates:continuumNonlocalLocal}.

When combining \eqref{eq:proofs:rates:NonFullyDiscreteLocal:L2rates:Linftyrates} and \eqref{eq:proofs:continuumRates:l2convergence:rates}, 
we note that the constant is independent of $T$ and therefore, we have that for large $T \geq 1$ and small $\eps_n \leq e^{-1/\kappa} \leq 1$:
\begin{align}
        \sup_{1 \leq k \leq N } &\sup_{t \in (t^{k-1},t^k]} \Vert \cI_n \bar{u}_n^k - u(t,\cdot) \Vert_{\Lp{2}(\Omega')} \leq C e^{ \l 1 +  C^4_{\mathrm{op}}  \r T} \Biggl( \Vert \cI_n \cP_n u_0 - u_0 \Vert_{\Lp{2}} \notag \\
        &+  \Vert \cI_n \cP_n \cA^*\ell - \cA^*\ell \Vert_{2} + \frac{T^{(p-1)}}{\eps_n^{d+p}} \Vert \cI_n \cP_n \tilde{K}(\cdot/\eps_n) - \tilde{K}(\cdot/\eps_n) \Vert_{\Lp{2}(\Omega \times \Omega)} \notag \\
        &+ \tau \frac{T^{p-1}}{\eps_n^{d+p}} \ls \frac{T^{p-2}}{\eps_n^{d+p}} + \frac{ T^{(p-2)/2} }{\eps_n^{(d+p)/2}} \rs \Biggr) + C T \eps_n \label{eq:proofs:rates:discreteNonlocalContinuumLocal:simplified:equation1}\\
    &\leq Ce^{\l 1 +  C^4_{\mathrm{op}} \r T} \Biggl( n^{-\alpha_1} + n^{-\alpha_2} + \frac{T^{(p-1)}}{\eps_n^{d+p + \alpha_3} n^{\alpha_3}} + \tau \frac{T^{(2p-3)}}{\eps_n^{2(d+p)}} \Biggr)+ C T \eps_n \label{eq:proofs:rates:discreteNonlocalContinuumLocal:simplified:equation2}\\
    &=: C (T_1 + T_2 + T_3 + T_4 + T_5) \label{eq:proofs:rates:discreteNonlocalContinuumLocal:simplified:equation6}
\end{align}
where we used Assumption \ref{ass:main:assumptions:K1} and Assumption \ref{ass:main:assumptions:R3} for \eqref{eq:proofs:rates:discreteNonlocalContinuumLocal:simplified:equation1}, Lemma \ref{lem:background:piecewise:equiRates} and the fact that for $p \geq 3$,
\[
2p - 3 \geq \begin{cases}
p-1 & \\
p/2  \\
3p/2 - 2
\end{cases}
\]
for \eqref{eq:proofs:rates:discreteNonlocalContinuumLocal:simplified:equation2}.

We now want all the terms in \eqref{eq:proofs:rates:discreteNonlocalContinuumLocal:simplified:equation6} to tend to $0$. By choosing $T = T(n) = \l 1+C_{\mathrm{op}}^4  \r^{-1} \log(\eps^{-\kappa})$ for some $\kappa >0$, we now derive sufficient conditions on $\eps_n$ and $\tau = \tau_n$ such that $T_i \to 0$ for $1 \leq i \leq 5$. 

We have that $\eps_n \gg n^{-\alpha_1/\kappa}$
implies that $T_1 \to 0$. Indeed, from the latter, $1 \gg \eps_n^{-\kappa} n^{-\alpha_1} = e^{\l 1 +  C^4_{\mathrm{op}} \r T } n^{-\alpha_1}$. Analogously, $\eps_n \gg n^{-\alpha_2/\kappa}$ implies that $T_2 \to 0$.

We have that
$\eps_n \gg  \ls \expW \l n^{\frac{\alpha_3}{\max\l 1 + (d+p+\alpha_3)/\kappa, p - 1 \r}}\r \rs^{-1/\kappa}$
implies that $T_3 \to 0$. Indeed, similarly to the above, the latter is equivalent to 
\begin{align}
    1 &\gg \ls \eps_n^{-\kappa} \log(\eps_n^{-\kappa}) \rs^{\max\l 1 + (d+p+\alpha_3)/\kappa, p - 1 \r} \l 1+C_{\mathrm{op}}^4  \r^{(1-p)} n^{-\alpha_3} \notag \\
    &\geq \frac{1}{\eps_n^{\kappa + d+p + \alpha_3}} \log(\eps_n^{-\kappa})^{(p-1)} \l 1+C_{\mathrm{op}}^4  \r^{(1-p)} n^{-\alpha_3} \notag \\
    &= T_3. \notag
\end{align}

We have that
\[
    \tau \ll  \l \frac{1}{1+C_{\mathrm{op}}^4}  \r^{(3-2p)} \frac{\eps_n^{2(d+p) + \kappa }}{\log(\eps_n^{-\kappa})^{(2p-3)}} = e^{- \l 1 +  C^4_{\mathrm{op}} \r T} \frac{\eps_n^{2(d+p)}}{T^{(2p-3)}}
\]
which implies that $T_4 \to 0$. 

Lastly, we have that $\lim_{n \to \infty} \eps_n \log(\eps_n^{-\kappa}) =0$ implying that $T_5 \to 0$ which concludes the proof.
\end{proof}

\begin{remark}[Generality of Corollary \ref{cor:proofs:rates:discreteNonlocalContinuumLocal:simplified}] \label{rem:proofs:rates:generality}
    The choice of $T(n)$ in Corollary \ref{cor:proofs:rates:discreteNonlocalContinuumLocal:simplified} is arbitrary: by considering the general rates obtained by combining the results of Proposition \ref{prop:proofs:rates:NonFullyDiscreteLocal:L2rates} and Theorem \ref{thm:proofs:continuumRates:continuumNonlocalLocal}, 
    one could derive similar results to \eqref{eq:proofs:rates:simplified:final} with other $T(n)$. Furthermore, by combining \eqref{eq:proofs:rates:NonFullyDiscreteLocal:L2rates:LPrates} with \eqref{eq:proofs:continuumRates:l2convergence:rates},
    one could extend the results to more general $u_0$ and $\cA^*\ell$.
\end{remark}

\begin{proof}[Proof of Corollary \ref{cor:proofs:rates:discreteNonlocalContinuum:final}]
In the proof $C>0$ will denote a constant that can be arbitrarily large, (which might be) dependent on $\Omega$, $u_0$ or/and $\cA^*\ell$, that may change from line to line.

We have:
\begin{align}
    \Vert \cI_n \bar{u}_n^N - u_{\infty} \Vert_{\Lp{2}(\Omega')} &\leq \Vert \cI_n \bar{u}_n^N - u(T,\cdot) \Vert_{\Lp{2}(\Omega')} + \Vert u(T,\cdot) - u_{\infty} \Vert_{\Lp{2}} \notag \\
    &\leq \sup_{1 \leq k \leq N } \sup_{t \in (t^{k-1},t^k]} \Vert \cI_n \bar{u}_{n}^k - u(t,\cdot) \Vert_{\Lp{2}(\Omega')} + \Vert u(T,\cdot) - u_{\infty} \Vert_{\Lp{2}} \notag \\
    &=: T_1 + T_2. \notag
\end{align}
We note that $C_{\mathrm{op}} = 1$ since $\cA = \Id$. For the $T_1$ term, we can apply Corollary \ref{cor:proofs:rates:discreteNonlocalContinuumLocal:simplified}. For the $T_2$ term, we note from Lemma \ref{lem:proofs:rates:lambdaConvexity} and first order conditions that:
\[
    C\Vert u(T,\cdot) - u_{\infty} \Vert_{\Lp{2}}^2 + \langle \nabla \cF(u_{\infty}), u(T,\cdot) - u_\infty \rangle = \Vert u(T,\cdot) - u_{\infty} \Vert_{\Lp{2}}^2 \leq \cF(u(T,\cdot)) - \cF(u_\infty).
\]
Furthermore, by standard considerations of gradient flows \cite{santambrogio2015optimal} we obtain that 
\[
\cF(u(T,\cdot)) - \cF(u_\infty) \leq e^{- T}(\cF(u_0) - \cF(u_\infty)) = \eps_n^{\kappa/2}(\cF(u_0) - \cF(u_\infty))
\]
so that combining the latter yields the claim.
\end{proof}

\begin{remark}[Asymptotic rates in Corollaries \ref{cor:proofs:rates:discreteNonlocalContinuumLocal:simplified} and \ref{cor:proofs:rates:discreteNonlocalContinuum:final}.] \label{rem:proofs:rates:discreteNonLocalContinuumLocal:asymptotics}

The asymptotic aspect in Corollaries \ref{cor:proofs:rates:discreteNonlocalContinuumLocal:simplified} and \ref{cor:proofs:rates:discreteNonlocalContinuum:final} are not very restrictive: indeed, one part comes from Theorem \ref{thm:proofs:continuumRates:continuumNonlocalLocal} and discussed in Remark \ref{rem:proofs:rates:continuumRates:asymptotic} while the second part comes from our estimates in Corollary \ref{cor:proofs:rates:discreteNonlocalContinuumLocal:simplified}. The latter are related to finding the smallest $n$ such that $\eps_n \leq e^{-1/\kappa}$ and $T \geq 1$ which for all reasonable choices of $\eps_n$ should occur for a very small $n$. 

\end{remark}

\begin{remark}[$g_n$ when $\cA = \Id$] 
In Corollary \ref{cor:proofs:rates:discreteNonlocalContinuum:final}, we pick $\cA = \Id$ so that Assumption \ref{ass:main:assumptions:O4} simplifies greatly: indeed, we can just pick $\bar{G}_n = \Id$ in this case.
\end{remark}

\subsection{Application to random graph models} \label{subsec:application}

Let us now consider the evolution problem \eqref{eq:main:notation:random:evolutionProblem}. We recall that we will now be working on $(0,1)$ with the uniform partition so that Assumptions \ref{ass:main:assumptions:S1} and \ref{ass:main:assumptions:S2} will always be satisfied and we have $d = 1$. 
 
\subsubsection{Well-posedness}

\begin{corollary}[Well-posedness of \eqref{eq:main:notation:random:evolutionProblem}]
\label{cor:application:wellPosedness}

Assume that Assumptions \ref{ass:main:assumptions:S1}, \ref{ass:main:assumptions:O1}, \ref{ass:main:assumptions:O2}, \ref{ass:main:assumptions:O3} and \ref{ass:main:assumptions:O4} hold. Let $K:[0,\infty) \mapsto [0,\infty)$, $p \geq 2$, $\mu > 0$, $T>0$, $u_0 \in \Lp{p}(\Omega)$, $\ell \in \Lp{2}(\Omega)$ and assume that $\cA^* \ell \in \Lp{p}(\Omega)$. Furthermore, let $n \in \bbN$ and define $\bar{K} = \cP_n \tilde{K}$, $\tilde{K}(x,y) = K(|x-y|)$, $\bar{f} = \cP_n \cA^*\ell $, $\bar{u}_0 = \cP_n u_0$.
Then, $\bbP$-a.e., for any partition $0 = t^0 < t^1 < \dots < t^N = T$, there exist a sequence $\{\bar{u}_n^k\}_{k=0}^N$ satisfying \eqref{eq:main:notation:random:evolutionProblem} with parameters  $\bar{\Lambda}_n$, $\bar{f}$ and $\bar{u}_0$ that is well-defined and unique. We also have
\[
\Vert \cI_n \bar{u}_n^k \Vert_{\Lp{r}} \leq \Vert u_0 \Vert_{\Lp{r}} + T \Vert \cA^*\ell \Vert_{\Lp{r}}
\]
for $1 \leq r \leq \infty$.
\end{corollary}

\begin{proof}
Since $\Vert \cI_n \bar{\Lambda}_n \Vert_{\Lp{\infty}}$ is bounded we proceed exactly as in Corollary \ref{cor:proofs:wellPosedness:discreteProblem:existence} by considering \eqref{eq:proofs:wellPosedness:discreteProblem:iterative} with kernel $\cI_n \bar{\Lambda}_n$.
\end{proof}

\subsubsection{Rates}

As an intermediate step in establishing the rates between the random discrete and continuum local problems, we will have to compare random and deterministic discrete solutions which is what we discuss in the next proposition. The proof of the latter is similar to the proof of Proposition \ref{prop:proofs:rates:NonFullyDiscreteLocal:L2rates} but requires a few additional probabilistic estimates. Furthermore, on the results side, some of terms in the error bounds of Proposition \ref{prop:application:ratesDiscreteDeterministics} only differ from terms in \eqref{eq:proofs:rates:NonFullyDiscreteLocal:L2rates:LPrates} and \eqref{eq:proofs:rates:NonFullyDiscreteLocal:L2rates:Linftyrates} by having $\sup_{x \in \Omega}\Vert \cI_n \bar{\Lambda}_n(x,\cdot) \Vert_{\Lp{1}} + \Vert \tilde{K}^{\eps_n} \Vert_{\Lp{\infty}}$ instead of $\Vert \tilde{K}^{\eps_n} \Vert_{\Lp{\infty}}$.  For completeness, the full proof can be found in Section \ref{sec:supplementary:rates}.

\begin{proposition}[Random-to-deterministic rates]
\label{prop:application:ratesDiscreteDeterministics}

Assume that Assumptions \ref{ass:main:assumptions:O1}, \ref{ass:main:assumptions:O2}, \ref{ass:main:assumptions:O3}, \ref{ass:main:assumptions:O4} and \ref{ass:main:assumptions:K1} hold. Let $p \geq 2$, $\mu > 0$, $T>0$, $u_0 \in \Lp{p}(\Omega)$, $\ell \in \Lp{2}(\Omega)$ and assume that $\cA^* \ell \in \Lp{p}(\Omega)$. Furthermore, let $n \in \bbN$ and define $\bar{K}^{\eps_n} = \cP_n \tilde{K}^{\eps_n}$, $\bar{f} = \cP_n \cA^*\ell $, $\bar{u}_0 = \cP_n u_0$.
We also suppose that $\rho_n$ is a positive sequence with $\rho_n \to 0$ and $\rho_n \ll \eps_n^{1+p}$. Let $\Lambda_n\in\bbR^{n\times n}$ be the weight matrix defined as in Definition~\ref{def:randomGraphModels} with $\bar{K}=\bar{K}^{\eps_n}$.

Then, for any partition $0 = t^0 < t^1 < \dots < t^N = T$, there exists a sequence $\{\bar{v}_n^k\}_{k=0}^N$ satisfying \eqref{eq:main:notation:discrete:nonlocalProblemFully} with parameters $\bar{K}^{\eps_n}$, $\bar{f}$ and $\bar{u}_0$. In addition, $\bbP$-a.e., there exists a sequence $\{\bar{u}_n^k\}_{k=0}^N$ solving \eqref{eq:main:notation:random:evolutionProblem} with parameters $\bar{\Lambda}_n$, $\bar{f}$ and $\bar{u}_0$.

For any $\theta > 0$, we have the following rates with probability larger than $1 - \frac{(C + T^{2(p-1)}C)} {\theta^2 n \rho_n } \Vert \tilde{K}^{\eps_n} \Vert_{\Lp{\infty}}$ for some $C >0$ dependent on $\Omega$, $u_0$ and $\cA^*\ell$: \begin{enumerate}[leftmargin=*]
\vspace{-2mm}
    \setlength\itemsep{-1mm}
    \item if $u_0 \in \Lp{2p-2/(2p-1)}(\Omega)$ and $\cA^*\ell \in \Lp{2p-2/(2p-1)}(\Omega)$:
    \begin{align}
        &\sup_{0 \leq t \leq T} \Vert \uTimeInter - \vTimeInter \Vert_{\Lp{2}} \leq C e^{\l \frac{2 + 3C_{\mathrm{op}}^4}{2} \r T} \Biggl( \theta \notag \\
        &+ \tau^{p/(2p-1)} \sup_{x \in \Omega}\Vert \cI_n \bar{\Lambda}_n(x,\cdot) \Vert_{\Lp{1}}^{p/(2p-1)} (1 + T^{p-1-1/p})^{p/(2p-1)} \notag \\
        &\times (1+T^{p-1})^{p/(2p-1)}\l 1 + \sup_{x \in \Omega}\Vert \cI_n \bar{\Lambda}_n(x,\cdot) \Vert_{\Lp{1}} + \Vert \tilde{K}^{\eps_n} \Vert_{\Lp{\infty}}  \r^{p/(2p-1)} \notag \\
        &+ \tau^{(p+1)/(2p)} (1 + T^{p-1-1/p})^{1/2} \sup_{x \in \Omega}\Vert \cI_n \bar{\Lambda}_n(x,\cdot) \Vert_{\Lp{1}}^{1/2} \notag \\
        &\times (1+T^{p-1})^{(p+1)/(2p)} \l 1 + \sup_{x \in \Omega}\Vert \cI_n \bar{\Lambda}_n(x,\cdot) \Vert_{\Lp{1}} + \Vert \tilde{K}^{\eps_n} \Vert_{\Lp{\infty}} \r^{(p+1)/(2p)} \notag \\
        &+  \tau (1+T^{p-1})\l 1 + \sup_{x \in \Omega}\Vert \cI_n \bar{\Lambda}_n(x,\cdot) \Vert_{\Lp{1}} + \Vert \tilde{K}^{\eps_n} \Vert_{\Lp{\infty}}  \r  \Biggr);\notag
    \end{align}
    
    \item if $u_0 \in \Lp{\infty}(\Omega)$ and $\cA^*\ell \in \Lp{\infty}(\Omega)$:
    \begin{align}
        &\sup_{0 \leq t \leq T} \Vert \uTimeInter - \vTimeInter \Vert_{\Lp{2}} \leq C e^{\l \frac{2 + 3C_{\mathrm{op}}^4}{2} \r T} \Biggl( \theta \notag \\
        &+\tau (1+T^{p-1})\l 1 + \sup_{x \in \Omega}\Vert \cI_n \bar{\Lambda}_n(x,\cdot) \Vert_{\Lp{1}} + \Vert \tilde{K}^{\eps_n} \Vert_{\Lp{\infty}}  \r \notag \\
        &\times \Biggl[ 1 + \l \sup_{x \in \Omega}\Vert \cI_n \bar{\Lambda}_n(x,\cdot) \Vert_{\Lp{1}} (1 + T^{p-2})\r^{1/2} + \sup_{x \in \Omega}\Vert \cI_n \bar{\Lambda}_n(x,\cdot) \Vert_{\Lp{1}} (1 + T^{p-2}) \Biggr] \Biggr). \notag
    \end{align}
\end{enumerate}
\end{proposition}

\begin{corollary}[Discrete random-to-continuum nonlocal rates]
\label{cor:application:rates}
Assume that Assumptions \ref{ass:main:assumptions:O1}, \ref{ass:main:assumptions:O2}, \ref{ass:main:assumptions:O3}, \ref{ass:main:assumptions:O4} and \ref{ass:main:assumptions:K1} hold. Let $p \geq 2$, $\mu > 0$, $T>0$, $u_0 \in \Lp{p}(\Omega)$, $\ell \in \Lp{2}(\Omega)$ and assume that $\cA^* \ell \in \Lp{p}(\Omega)$.  Furthermore, let $n \in \bbN$ and define $\bar{K}^{\eps_n} = \cP_n \tilde{K}^{\eps_n}$, $\bar{f} = \cP_n \cA^*\ell $, $\bar{u}_0 = \cP_n u_0$. We also suppose that $\rho_n$ is a positive sequence with $\rho_n \to 0$ and $\rho_n \ll \eps_n^{1+p}$. Let $\Lambda_n\in\bbR^{n\times n}$ be the weight matrix defined as in Definition~\ref{def:randomGraphModels} with $\bar{K}=\bar{K}^{\eps_n}$.

Then, for any partition $0 = t^0 < t^1 < \dots < t^N = T$, there exists a sequence $\{\bar{u}_n^k\}_{k=0}^N$ solving \eqref{eq:main:notation:random:evolutionProblem} with parameters $\bar{\Lambda}_n$, $\bar{f}$ and $\bar{u}_0$, and a solution $u_{\eps_n}$ to \eqref{eq:main:notation:nonlocal:nonlocalProblem} with kernel $K^{\eps_n}$.

For any $\theta > 0$, we have the following rates with probability larger than $1 - \frac{(C + T^{2(p-1)}C)} {\theta^2 n \rho_n } \Vert \tilde{K}^{\eps_n} \Vert_{\Lp{\infty}}$ for some $C >0$ dependent on $\Omega$, $u_0$ and $\cA^*\ell$: \begin{enumerate}[leftmargin=*]
\vspace{-2mm}
    \setlength\itemsep{-1mm}
    \item if $u_0 \in \Lp{2p-2/(2p-1)}(\Omega)$ and $\cA^*\ell \in \Lp{2p-2/(2p-1)}(\Omega)$, then:
        \begin{align}
    &\sup_{1 \leq k \leq N} \sup_{t \in (t^{k-1},t^k]} \Vert \cI_n \bar{u}_n^k - u_{\eps_n} \Vert_{\Lp{2}} \leq \tau C (1+T^{p-1})\l 1 + \sup_{x \in \Omega}\Vert \cI_n \bar{\Lambda}_n(x,\cdot) \Vert_{\Lp{1}} + \frac{\Vert K \Vert_{\Lp{\infty}}}{{\eps_n^{1+p}}}  \r \notag \\
    &+ C e^{\l \frac{2 + 3C_{\mathrm{op}}^4}{2} \r T} \Biggl( \theta \notag \\
        &+ \tau^{p/(2p-1)} \sup_{x \in \Omega}\Vert \cI_n \bar{\Lambda}_n(x,\cdot) \Vert_{\Lp{1}}^{p/(2p-1)} (1 + T^{p-1-1/p})^{p/(2p-1)} \notag \\
        &\times (1+T^{p-1})^{p/(2p-1)}\l 1 + \sup_{x \in \Omega}\Vert \cI_n \bar{\Lambda}_n(x,\cdot) \Vert_{\Lp{1}} + \frac{\Vert K \Vert_{\Lp{\infty}}}{{\eps_n^{1+p}}}  \r^{p/(2p-1)} \notag \\
        &+  \tau^{(p+1)/(2p)} (1 + T^{p-1-1/p})^{1/2} \sup_{x \in \Omega}\Vert \cI_n \bar{\Lambda}_n(x,\cdot) \Vert_{\Lp{1}}^{1/2} \notag \\
        &\times (1+T^{p-1})^{(p+1)/(2p)} \l 1 + \sup_{x \in \Omega}\Vert \cI_n \bar{\Lambda}_n(x,\cdot) \Vert_{\Lp{1}} + \frac{\Vert K \Vert_{\Lp{\infty}}}{{\eps_n^{1+p}}} \r^{(p+1)/(2p)} \notag \\
        &+  \tau (1+T^{p-1})\l 1 + \sup_{x \in \Omega}\Vert \cI_n \bar{\Lambda}_n(x,\cdot) \Vert_{\Lp{1}} + \frac{\Vert K \Vert_{\Lp{\infty}}}{{\eps_n^{1+p}}}  \r  \Biggr)\notag \\
    &+ C e^{\l 1 +  C^4_{\mathrm{op}} \r T} \Biggl( \Vert \cI_n \cP_n u_0 - u_0 \Vert_{\Lp{2}} + \Vert \cI_n \cP_n \cA^*\ell - \cA^*\ell \Vert_{\Lp{2}} \label{eq:application:randomContinuumNonlocal:Lp}\\
        &+ \frac{(1 + T^{p-1})}{\eps_n^{1+p}} \sup_{x \in \Omega} \Vert \cI_n \cP_n K(\vert x - \cdot \vert/\eps_n) - K(\vert x - \cdot \vert/\eps_n) \Vert_{\Lp{2}}  +  \tau \l 1 + \frac{\Vert K \Vert_{\Lp{\infty}}}{{\eps_n^{1+p}}} \r (1 + T^{p-1}) \notag\\
    &+    \tau^{p/(2p-1)} \l \frac{\Vert K \Vert_{\Lp{\infty}}}{{\eps_n^{1+p}}} \r^{p/(2p-1)} \l 1 + \frac{\Vert K \Vert_{\Lp{\infty}}}{{\eps_n^{1+p}}}\r^{p/(2p-1)} (1 + T^{p-1})^{p/(2p-1)} ( 1 + T^{p-1 - 1/p})^{p/(2p-1)} \notag\\
    &+  \tau^{(p+1)/(2p)} \l \frac{\Vert K \Vert_{\Lp{\infty}}}{{\eps_n^{1+p}}} \r^{1/2} \l 1 + T^{p-1 - 1/p}  \r^{1/2} \l 1 + \frac{\Vert K \Vert_{\Lp{\infty}}}{{\eps_n^{1+p}}} \r^{(p+1)/(2p)} (1 + T^{p-1})^{(p+1)/(2p)} \Biggr); \notag
\end{align}
    
    \item if $u_0 \in \Lp{\infty}(\Omega)$ and $\cA^*\ell \in \Lp{\infty}(\Omega)$, then:
    \begin{align}
    &\sup_{1 \leq k \leq N} \sup_{t \in (t^{k-1},t^k]} \Vert \cI_n \bar{u}_n^k - u_{\eps_n} \Vert_{\Lp{2}} \leq C e^{\l \frac{2 + 3C_{\mathrm{op}}^4}{2} \r T} \Biggl( \theta \notag \\
        &+\tau (1+T^{p-1})\l 1 + \sup_{x \in \Omega}\Vert \cI_n \bar{\Lambda}_n(x,\cdot) \Vert_{\Lp{1}} + \frac{\Vert K \Vert_{\Lp{\infty}}}{{\eps_n^{1+p}}}  \r \notag \\
        &\times \Biggl[ 1 + \l \sup_{x \in \Omega}\Vert \cI_n \bar{\Lambda}_n(x,\cdot) \Vert_{\Lp{1}} (1 + T^{p-2})\r^{1/2} + \sup_{x \in \Omega}\Vert \cI_n \bar{\Lambda}_n(x,\cdot) \Vert_{\Lp{1}} (1 + T^{p-2}) \Biggr] \Biggr) \notag  \\
    &+ C e^{ \l 1 +  C^4_{\mathrm{op}} \r T} \Biggl( \Vert \cI_n \cP_n u_0 - u_0 \Vert_{\Lp{2}} + \Vert \cI_n \cP_n \cA^*\ell - \cA^*\ell \Vert_{\Lp{2}} \label{eq:application:randomContinuumNonlocal:Linfty} \\
        &+ \frac{(1 + T^{p-1})}{\eps_n^{d+1}} \Vert \cI_n \cP_n \tilde{K}(\cdot/\eps_n) - \tilde{K}(\cdot/\eps_n) \Vert_{\Lp{2}(\Omega \times \Omega)} \notag \\
        &+ \tau \l 1 + \frac{\Vert K \Vert_{\Lp{\infty}}}{{\eps_n^{1+p}}} \r (1 + T^{p-1}) \ls 1 + \frac{\Vert K \Vert_{\Lp{\infty}}}{{\eps_n^{1+p}}}(1+T^{p-2}) + \l \frac{\Vert K \Vert_{\Lp{\infty}}}{{\eps_n^{1+p}}} (1 + T^{p-2}) \r^{1/2} \rs \biggr).\notag
\end{align}
\end{enumerate}
\end{corollary}

\begin{proof}
In the proof $C>0$ will denote a constant that can be arbitrarily large, (which might be) dependent on $\Omega$, $u_0$ or/and $\cA^*\ell$, that may change from line to line.

The existence of $\{\bar{u}_n^k\}_{k=0}^N$ follows from Proposition \ref{prop:application:ratesDiscreteDeterministics}. Let $\{\bar{v}_n^k\}_{k=0}^N$ be as in Proposition \ref{prop:application:ratesDiscreteDeterministics} as well. The existence of $u_{\eps_n}$ follows from Corollary \ref{cor:proofs:rates:discreteNonlocalContinuumLocal:simplified}. 
We note that
\begin{align}
    \sup_{1 \leq k \leq N} \sup_{t \in (t^{k-1},t^k]} \Vert \cI_n \bar{u}_n^k - u_{\eps_n} \Vert_{\Lp{2}} &= \sup_{0 < t \leq T} \Vert \uTimeInject - u_{\eps_n} \Vert_{\Lp{2}} \notag \\
    &\leq \sup_{0 < t \leq T} \Vert \uTimeInject - \uTimeInter \Vert_{\Lp{2}} \notag + \sup_{0 < t \leq T} \Vert \uTimeInter -  \vTimeInter \Vert_{\Lp{2}}
    \notag\\
    &+ \sup_{0 < t \leq T} \Vert \vTimeInter -  \vTimeInject \Vert_{\Lp{2}} + \sup_{0 < t \leq T} \Vert \vTimeInject -  u_{\eps_n} \Vert_{\Lp{2}} \notag \\
    &\leq \tau \l C + CT + (C+T^{p-1}C)\ls \sup_{x \in \Omega}\Vert \cI_n \bar{\Lambda}_n(x,\cdot) \Vert_{\Lp{1}} + \Vert \tilde{K}^{\eps_n} \Vert_{\Lp{\infty}}  \rs \r \notag \\
    &+ \sup_{0 < t \leq T} \Vert \uTimeInter -  \vTimeInter \Vert_{\Lp{2}} + \sup_{0 < t \leq T} \Vert \vTimeInject -  u_{\eps_n} \Vert_{\Lp{2}} \label{eq:application:randomContinuumNonlocal:equation1}
\end{align}
where we used \eqref{eq:application:ratesDiscreteDeterminisitic:equation1} for \eqref{eq:application:randomContinuumNonlocal:equation1}.

Relying on \eqref{eq:application:ratesDiscreteDeterministic:kHatBound}, it is possible to apply Proposition \ref{prop:proofs:rates:NonFullyDiscreteLocal:L2rates} with the kernel $\bar{K}^{\eps_n}$ and obtain the same bounds (with the scaling factor $\eps_n^{-(1+p)}$). Combining the latter with Proposition \ref{prop:application:ratesDiscreteDeterministics} we obtain   \eqref{eq:application:randomContinuumNonlocal:Lp} and \eqref{eq:application:randomContinuumNonlocal:Linfty}.
\end{proof}

Similarly to the setting in Section \ref{subsec:ratesDiscreteLocal}, we can combine the results of Corollary \ref{cor:application:rates} and Theorem \ref{thm:proofs:continuumRates:continuumNonlocalLocal} to obtain precise rates for fixed $T > 0$. 
Recalling the discussion in Remark \ref{rem:proofs:rates:generality}, the analogous results to Corollaries \ref{cor:proofs:rates:discreteNonlocalContinuumLocal:simplified} and \ref{cor:proofs:rates:discreteNonlocalContinuum:final} are Corollaries \ref{cor:application:simplified} and \ref{cor:application:final}. Before proceeding to the proof of the latter two results, we present two probabilistic lemmas that are essential for the establishment of the rates.

 First, we start a variant of \cite[Lemma 3.1]{ElBouchairi}, by explicitly writing out the estimates for the scaled kernels. In this section, we recall that we will be working on $(0,1)$ with the uniform partition so that Assumptions \ref{ass:main:assumptions:S1} and \ref{ass:main:assumptions:S2} will always be satisfied and we have $d = 1$. 
 In this setting the operators $\cI_n$ and $\cP_n$ can be defined
\begin{align*}
(\cI_nu)(x) & = \sum_{i\in[n]} u_i \chi_{\Omega_{n,i}}(x) 
 & (\cI_nu)(x,y) & = \sum_{i\in[n],j\in[n]} u_{ij} \chi_{\Omega_{n,i}}(x) \chi_{\Omega_{n,j}}(y) \\
(\cP_n u)_i & = n\int_{\Omega_{n,i}} u(x) \, \dd x & (\cP_n u)_{ij} & = n^2 \int_{\Omega_{n,i}\times \Omega_{n,j}} u(x,y) \, \dd x \, \dd y
\end{align*}
where $\Omega_{n,i}=(\frac{i-1}{n},\frac{i}{n})$.

\begin{lemma}[Convergence of the random weight matrix] \label{lem:application:convergenceLambda}
Assume $K$ satisfies Assumption \ref{ass:main:assumptions:K1}, $p > 1$, $d=1$, $\rho_n$ and $\eps_n$ are positive sequences with $\eps_n \to 0$ and $\rho_n \to 0$ and define $\tilde{K}_{\eps_n}(x,y) = \frac{2}{c(p,1)\eps_n^{p+1}} K\l\frac{|x-y|}{\eps}\r$, where $c(p,1)$ is defined by~\eqref{eq:main:notation:nonlocal:cpd}.
Define $\bar{K}^{\eps_n}=\cP_n\tilde{K}_{\eps_n}$ and let $\Lambda_n\in\bbR^{n\times n}$ be the weight matrix defined as in Definition~\ref{def:randomGraphModels} with $\bar{K}=\bar{K}^{\eps_n}$.
Assume that
\[
\frac{\log(n) \eps_n^{2p}}{n} \ll \rho_n \ll \eps_n^{p+1}.
\]
Then, with probability one, for $n$ large enough, we have
\begin{equation} \label{eq:application:convergenceLambda:rate}
\vert \sup_{x\in \Omega} \Vert \cI_n \bar{\Lambda}_n(x,\cdot) \Vert_{\Lp{1}}  - \sup_{x \in \Omega} \Vert \cI_n \bar{K}_{\eps_n}(x,\cdot)  \Vert_{\Lp{1}} \vert < \eps_n^{-p}.
\end{equation}
Furthermore, with probability one, for $n$ large enough, we have
\[ \sup_{x\in\Omega} \| \cI_n \bar{\Lambda}_n(x,\cdot)\|_{\Lp{1}} \leq \frac{C}{\eps_n^p}. \]
\end{lemma}

\begin{proof}
We can estimate as follows:
\begin{align}
& \bbP \Biggl( \vert \sup_{x \in \Omega}\Vert \cI_n \bar{\Lambda}_n(x,\cdot) \Vert_{\Lp{1}}  - \sup_{x \in \Omega} \Vert \cI_n \bar{K}^{\eps_n}(x,\cdot)  \Vert_{\Lp{1}} \vert > \eps_n^{-p} \Biggr) \notag \\
& \qquad \qquad = \bbP \l \vert \max_{i\in [n]} \frac{1}{n} \sum_{j \in [n]} \bar{\Lambda}_{n,ij}   - \max_{i\in [n]} \frac{1}{n} \sum_{j \in [n]} \bar{K}^{\eps_n}_{ij} \vert > \eps_n^{-p} \r \notag \\
& \qquad \qquad \leq \bbP \l \max_{i\in [n]} \vert \frac{1}{n} \sum_{j \in [n]} \l \bar{\Lambda}_{n,ij}   - \bar{K}^{\eps_n}_{ij} \r \vert > \eps_n^{-p} \r \notag \\
& \qquad \qquad \leq \sum_{i\in [n]} \bbP \l \vert \frac{1}{n} \sum_{j\in [n]} \l \bar{\Lambda}_{n,ij}  - \bar{K}^{\eps_n}_{ij} \r \vert > \eps_n^{-p} \r \notag \\
& \qquad \qquad \leq 2 \sum_{i\in [n]} \exp\l-\frac{\frac12 n^2\eps_n^{-2p}}{\sum_{j\in [n]} \frac{\bar{K}_{ij}^{\eps_n}}{\rho_n} (1-\rho_n\bar{K}^{\eps_n}_{ij}) + \frac{n}{3\eps_n^p} (\frac{1}{\rho_n} + \frac{C}{\eps_n^{p+1}}) } \r  \label{eq:application:convergenceLambda:Bernstein} \\
& \qquad \qquad \leq  2 \sum_{i\in[n]} \exp\l - \frac{cn^2\eps_n^{-2p}}{\frac{n}{\rho_n^2} + \frac{n}{\eps_n^{p}}(\frac{1}{\rho_n}+\frac{1}{\eps_n^{p+1}})} \r \notag \\
& \qquad \qquad \leq 2 \sum_{i\in[n]} \exp\l - \frac{cn\eps_n^{-2p}\rho_n^2}{1+\frac{\rho_n}{\eps_n^{p}} + \frac{\rho_n^2}{\eps_n^{2p+1}}} \r \notag \\
& \qquad \qquad \leq 2\sum_{i\in[n]} \exp\l-\frac{cn\rho_n^2}{\eps_n^{2p}}\r \notag
\end{align}
where we used Bernstein's lemma for~\eqref{eq:application:convergenceLambda:Bernstein} after noticing that $\bbE[\bar{\Lambda}_{n,ij}-\bar{K}_{ij}^{\eps_n}]=0$, $\la\bar{\Lambda}_{n,ij}-\bar{K}^{\eps_n}_{ij}\ra\leq \frac{1}{\rho_n} + \frac{C}{\eps_n^{p+1}}$ and $\bbE[\bar{\Lambda}_{n,ij}-\bar{K}^{\eps_n}_{ij}]^2 = \frac{\bar{K}^{\eps_n}_{ij}}{\rho_n}(1-\rho_n\bar{K}^{\eps_n}_{ij})\leq \frac{1}{4\rho_n^2}$.
Choosing $\gamma>2$ such that $\frac{cn\rho_n^2}{\eps_n^{2p}} \geq \gamma \log(n)$ we have
\[ \bbP \Biggl( \vert \sup_{x \in \Omega}\Vert \cI_n \bar{\Lambda}_n(x,\cdot) \Vert_{\Lp{1}}  - \sup_{x \in \Omega} \Vert \cI_n \bar{K}^{\eps_n}(x,\cdot)  \Vert_{\Lp{1}} \vert > \eps_n^{-p} \Biggr) \leq 2 \sum_{i\in[n]} n^{-\gamma} \]
which is summable.
By the Borel-Cantelli Lemma, with probability one, for all but finitely many $n$,
\[
\vert \sup_{x \in \Omega}\Vert \cI_n \bar{\Lambda}_n(x,\cdot) \Vert_{\Lp{1}}  - \sup_{x \in \Omega} \Vert \cI_n \bar{K}^{\eps_n}(x,\cdot)  \Vert_{\Lp{1}} \vert < \eps_n^{-p}.
\]

For the furthermore part of the lemma we note that we can write
\begin{align*}
\lda \cI_n \bar{K}^{\eps_n}(x,\cdot) \rda_{\Lp{1}} & = \int_\Omega \cI_n\bar{K}^{\eps_n}(x,y) \, \dd y \\
 & = \int_\Omega \sum_{i.j=1}^n (\cP_n \tilde{K}^{\eps_n})_{i,j} \chi_{\Omega_{n,i}}(x) \chi_{\Omega_{n,j}}(y) \, \dd y \\
 & = \frac{n^2 C(p,1)}{\eps_n^{p+1}} \int_\Omega \sum_{i,j=1}^n \int_{\Omega_{n,i}\times \Omega_{n,j}} K\l\frac{|w-z|}{\eps_n}\r \, \dd w \, \dd z \, \chi_{\Omega_{n,i}}(x) \chi_{\Omega_{n,j}}(y) \, \dd y \\
 & = \frac{n C(p,1)}{\eps_n^{p+1}} \sum_{i,j=1}^n \int_{\Omega_{n,i}\times\Omega_{n,j}} K\l\frac{|w-z|}{\eps_n}\r \, \dd w \, \dd z \, \chi_{\Omega_{n,i}}(x).
\end{align*}
So,
\[ \sup_{x\in\Omega} \| \cI_n \bar{K}^{\eps_n}(x,\cdot) \|_{\Lp{1}} \leq \frac{C(p,1)}{\eps_n^{p+1}} \int_{\bbR} K\l\frac{|z|}{\eps}\r \, \dd z = \frac{C(p,1)}{\eps_n^p} \int_{\bbR} K(|z|) \, \dd z. \]
Combining with the first part of the lemma, this completes the proof.
\end{proof}

\begin{remark} \label{rem:application:tradeOff}
Asymptotic rates and Borel-Cantelli arguments. The asymptotic claim  in \eqref{eq:application:convergenceLambda:rate}, i.e. the existence of some $N$ such that for all $n\geq N$~\eqref{eq:application:convergenceLambda:rate} holds, comes from a Borel-Cantelli argument. 
All we know is that, with probability one, $N<\infty$.
This can be circumvented with the following trade-off: either one argues with a Borel-Cantelli Lemma and obtains an $\bbP$-a.e. statement with an asymptotic part or one does not and is then left with a claim holding with high probability. 
\end{remark}

Similarly to what was discussed in Section \ref{subsec:ratesDiscreteLocal} we want to find conditions under which we will be able to take the right-hand side of \eqref{eq:application:randomContinuumNonlocal:Linfty} to $0$. To that purpose, we explicit a choice of $\theta = \theta_n$ that is compatible with the conditions derived in Lemma \ref{lem:application:convergenceLambda}.

\begin{lemma}[Rates for $\theta_n$] \label{lem:application:ratesThetaSimplified}
Let $d = 1$, $p > 1$, $\rho_n$ and $\eps_n$ be positive sequences with $\eps_n \to 0$ and $\rho_n \to 0$. For some $\kappa > 0$, let $T(n) = \l \frac{2}{2 + 3C_{\mathrm{op}}^4} \r \log(\eps_n^{-\kappa})$. Assume that
\[ \eps_n \gg \ls \expW \l   \l \log(n)  \r^{1/\max(2(p-1),2+(1+3p)/\kappa)}   \r \rs^{-1/\kappa}.\]
Then, $\frac{[\log(\eps_n^{-\kappa})]^{2(p-1)}}{ \eps_n^{1+3p} \log(n)} \ll \eps_n^{2\kappa}$.
Moreover, for a positive sequence $\theta_n$ satisfying
\[
    \frac{[\log(\eps_n^{-\kappa})]^{2(p-1)}}{ \eps_n^{1+3p} \log(n)} \ll \theta_n^2 \ll \eps_n^{2\kappa},
\]
and assuming
\[ \frac{\log(n) \eps_n^{2p}}{n} \ll \rho_n \]
we have that $e^{\l \frac{2 + 3C_{\mathrm{op}}^4}{2} \r T}\theta_n \ll 1$ and $\frac{T^{2(p-1)}}{\theta_n^2 n \rho_n \eps_n^{1+p}} \ll 1$.
\end{lemma}

\begin{proof}
We start by assuming $\frac{[\log(\eps_n^{-\kappa})]^{2(p-1)}}{ \eps_n^{1+3p} \log(n)} \ll \eps_n^{2\kappa}$ holds, take $\eps_n$, $\rho_n$, $T=T(n)$ and $\rho_n$ satisfying the appropriate assumptions and show $e^{\l \frac{2 + 3C_{\mathrm{op}}^4}{2} \r T}\theta_n \ll 1$ and $\frac{T^{2(p-1)}}{\theta_n^2 n \rho_n \eps_n^{1+p}} \ll 1$.
The former is equivalent to 
$\eps_n^{-\kappa} \theta_n \ll 1$ 
or 
$
\theta_n^2 \ll \eps_n^{2\kappa}.
$

For the latter, by assumption, we have that $\rho_n^{-1} \ll n \log(n)^{-1} \eps_n^{-2p}$ so that
\[
\frac{T^{2(p-1)}}{\theta_n^2 n \rho_n \eps_n^{1+p}} \ll \frac{T^{2(p-1)}}{\theta_n^2 \eps_n^{1+3p} \log(n)} 
\]
and therefore, the lower bound 
\[
 \frac{T^{2(p-1)} }{\eps_n^{1+3p} \log(n)} = \frac{\l \frac{2}{2 + 3C_{\mathrm{op}}^4} \r^{2(p-1)}[\log(\eps_n^{-\kappa})]^{2(p-1)}}{ \eps_n^{1+3p} \log(n)} \ll \theta_n^2 
\]
is sufficient.

It remains to check under which conditions we have
\[
\frac{\log(\eps_n^{-\kappa})^{2(p-1)}}{ \eps_n^{1+3p} \log(n)} \ll  \eps_n^{2\kappa}.
\]
This is equivalent to 
\[
 \frac{\log(\eps_n^{-\kappa})^{2(p - 1)}}{\eps_n^{2\kappa + 1 + 3p}} \ll  \log(n)
\]
which, for $n$ large enough, is implied by
\[
\l \frac{1}{\eps_n^{\kappa}} \log(\eps_n^{-\kappa}) \r^{\max(2(p-1),2+(1+3p)/\kappa)} \ll \log(n).
\]
In turn this leads to
\[
\eps_n \gg \ls \expW \l   \l \log(n)  \r^{1/\max(2(p-1),2+(1+3p)/\kappa)}   \r \rs^{-1/\kappa}. \qedhere
\]
\end{proof}

\begin{proof}[Proof of Corollary \ref{cor:application:simplified}]
In the proof $C>0$ will denote a constant that can be arbitrarily large, (which might be) dependent on $\Omega$, $u_0$ or/and $\cA^*\ell$, that may change from line to line.

The existence claims follow from Corollary \ref{cor:application:rates} and Theorem \ref{thm:proofs:wellPosedness:localProblem:existenceUniqueness}.

In view of Corollary \ref{cor:proofs:rates:discreteNonlocalContinuumLocal:simplified}, let $T_1$ be the terms in the combination of \eqref{eq:proofs:continuumRates:l2convergence:rates} and  \eqref{eq:application:randomContinuumNonlocal:Linfty} that are not included in the combination of 
\eqref{eq:proofs:continuumRates:l2convergence:rates} and \eqref{eq:proofs:rates:NonFullyDiscreteLocal:L2rates:Linftyrates}.
Since the constant in the combination of \eqref{eq:proofs:continuumRates:l2convergence:rates} and  \eqref{eq:application:randomContinuumNonlocal:Linfty} 
is independent of $T$, we have that for large $T \geq 1$ and small $\eps_n \leq 1$:
\begin{align}
    T_1
    &\leq C e^{\l \frac{2 + 3C_{\mathrm{op}}^4}{2} \r T} \Biggl( \theta_n + \tau_n   T^{p-1}\l \sup_{x \in \Omega}\Vert \cI_n \bar{\Lambda}_n(x,\cdot) \Vert_{\Lp{1}} + \frac{\Vert K \Vert_{\Lp{\infty}}}{{\eps_n^{1+p}}}  \r \notag \\
        &\times \Biggl[\l \sup_{x \in \Omega}\Vert \cI_n \bar{\Lambda}_n(x,\cdot) \Vert_{\Lp{1}} T^{p-2}\r^{1/2} + \sup_{x \in \Omega}\Vert \cI_n \bar{\Lambda}_n(x,\cdot) \Vert_{\Lp{1}} T^{p-2} \Biggr] \Biggr) \notag \\
    &\leq  C e^{\l \frac{2 + 3C_{\mathrm{op}}^4}{2} \r T} \Biggl( \theta_n + \tau_n  \frac{T^{p-1}}{\eps_n^{1+p}}  \times \Biggl[ \frac{T^{(p-2)/2}}{\eps_n^{p/2}} + \frac{T^{p-2}}{\eps_n^{p}} \Biggr] \Biggr) \label{eq:application:simplified:T1:equation3}\\
    &\leq  C e^{\l \frac{2 + 3C_{\mathrm{op}}^4}{2} \r T} \Biggl( \theta_n + \tau_n  \frac{T^{2p-3}}{\eps_n^{1+2p}}  \Biggr) \label{eq:application:simplified:T1:equation4} \\
    &=: T_2 + T_3 \notag
\end{align}
where we used 
Lemma \ref{lem:application:convergenceLambda} for \eqref{eq:application:simplified:T1:equation3}
similar reasoning for $p \geq 3$ as in Corollary \ref{cor:proofs:rates:discreteNonlocalContinuumLocal:simplified} for \eqref{eq:application:simplified:T1:equation4}.

By assumption, we can apply Lemma \ref{lem:application:ratesThetaSimplified} to see that $T_2 \to 0$.
We have that $T_3 \to 0$ is equivalent to 
\[
\tau_n \ll \frac{ \eps_n^{1 + 2p + \kappa}}{\log(\eps_n^{-\kappa})^{2p - 3}}
\]
which holds by assumption.
We conclude the proof by combining \eqref{eq:application:simplified:T1:equation4} and \eqref{eq:proofs:rates:simplified:final} to obtain \eqref{eq:application:simplified:final}. The probability claim follows from Lemma \ref{lem:application:ratesThetaSimplified}.
\end{proof}

\begin{remark}[CFL condition for random-to-deterministic error convergence]
    Using the notation of Corollary \ref{cor:application:simplified}, we see from the proof of the latter that the requirement on $\tau_n$ that ensures $T_1 \to 0$ (note that $T_1$ here is the error bound arising from the comparison of the random solution to the deterministic one) is slightly better than the CFL condition for our complete problem. Indeed, for $T_1 \to 0$, we only need 
    \[
    \tau_n \ll \frac{ \eps_n^{1 + 2p + \kappa}}{\log(\eps_n^{-\kappa})^{2p - 3}} \qquad\text{as opposed to} \qquad \tau_n \ll \frac{ \eps_n^{2 + 2p + \kappa}}{\log(\eps_n^{-\kappa})^{2p - 3}}
    \]
    for the rest of the terms.
\end{remark}

\begin{proof}[Proof of Corollary \ref{cor:application:final}]
We proceed as in the proof of Corollary \ref{cor:proofs:rates:discreteNonlocalContinuum:final} with Corollary \ref{cor:application:simplified}.
\end{proof}

\begin{remark}[Asymptotics in Corollaries \ref{cor:application:simplified} and \ref{cor:application:final}]

The rates in Corollary \ref{cor:application:final} are formulated for large $n$. This is partly non-restrictive in practice as described in Remark \ref{rem:proofs:rates:discreteNonLocalContinuumLocal:asymptotics} and in the proof of the Corollary itself. The non-desirable part of this requirement stems from the application of Lemma \ref{lem:application:convergenceLambda} as discussed in Remark \ref{rem:application:tradeOff}. We conclude from these results that when considering random graph models, one loses the full traceability of the asymptotic aspect of the rates as opposed to the results in Corollary \ref{cor:proofs:rates:discreteNonlocalContinuum:final}.
\end{remark}

\section*{Acknowledgements}
The authors were supported by the European Research Council under the European Union's Horizon 2020 research and innovation programme grant agreement No 777826 (NoMADS). AW is particularly thankful for the warm hospitality during his stay at GREYC where the research idea behind the paper was developed.

\bibliographystyle{plain}
\bibliography{references}{}

\newpage

\section{Supplementary material}
 
\subsection{Technical lemmas} 

\subsubsection{Domain of integration}

\begin{lemma}[Asymptotics of domain of integration] \label{lem:proofs:rates:continuumRates:domainIntegration}

Assume that Assumptions \ref{ass:main:assumptions:S1} and \ref{ass:main:assumptions:L1} hold. Let $\Omega'$ be compactly contained in $\Omega$ and $C \subseteq \bbR^d$ be a compact set. Then, for $n$ large enough, for all $x \in \Omega'$, the set $S_{\eps_n}(x) = \{z \in \bbR^d \spaceBar x + \eps_n z \in \Omega \} \cap C$ is equal to $C$.
\end{lemma}

\begin{proof}
We only need to show that $C \subseteq S_{\eps_n}(x)$. Since by definition $\closure(\Omega') \subseteq \interior(\Omega)$, for all $x \in \Omega'$, there exists $r_x > 0$ such that $B(x,r_x) \subseteq \interior(\Omega)$ which implies that $d(x,\partial \Omega) > 0$. Now $d(\cdot,\partial\Omega):\closure(\Omega')\mapsto \bbR$ is continuous and $\closure(\Omega')$ is compact so that there exists $x_0 \in \closure(\Omega')$ with $r = \min_{y \in \closure(\Omega')} d(\cdot,\partial \Omega) = d(x_0,\partial \Omega) > 0$.

Since $C$ is a compact set, there exists $c > 0$ such that $\vert y \vert \leq c$ for all $y \in C$. By Assumption \ref{ass:main:assumptions:L1}, there exists $n_0$ such that for all $n \geq n_0$, $\eps_n \leq r/c$. 

Let $n \geq n_0$, $x \in \Omega'$ and $z' \in C$. Then, $x + \eps_n z' \in B(x,\eps_n \vert z' \vert) \subseteq B(x,r) \subseteq \interior(\Omega) \subseteq \Omega$, so $z^\prime \in S_{\eps_n}(x)$.
\end{proof}

\subsubsection{Convexity}

\begin{mydef}[$\lambda$-convexity]
Let $(X,\|\cdot\|_X)$ be a normed space.
We say that a function $f:X \mapsto \bbR$ is $\lambda$-convex if for all $t \in [0,1]$ and $x,y \in X$, we have:
\[
f((1-t)x + ty) \leq (1-t)f(x) + tf(y) - \frac{\lambda}{2}t(1-t)\Vert x - y \Vert_{X}^2.
\]
\end{mydef}

\begin{lemma}[$\lambda$-convexity of $\cF$] \label{lem:proofs:rates:lambdaConvexity}
Assume that Assumption \ref{ass:main:assumptions:S1} holds.
Further, assume that $\cA$ is such that $\cA^*\cA$ is coercive with $\langle \cA^*\cA v,v\rangle\geq c_{\cA}\|v\|_{\Lp{2}(\Omega)}$ for all $v\in \Lp{2}(\Omega)$.
Then, we have that the energy $\cF$ is $c_{\cA}$-convex.
\end{lemma}

\begin{proof}
Since $t\mapsto |t|^p$ is convex we have 
\begin{align}
    \cF(tu + (1-t)v) &= \frac{\mu}{p}\int_\Omega \vert t \nabla u + (1-t) \nabla v \vert^p \, \dd x + \frac{1}{2}\Vert t \cA u + (1-t) \cA v - \ell \Vert_{\Lp{2}}^2 \notag \\
    &\leq \frac{\mu}{p} \int_\Omega t |\nabla u|^p + (1-t)|\nabla v|^p \, \dd x + \frac{1}{2} [  t^2 \Vert \cA u - \ell \Vert_{\Lp{2}}^2 + (1-t)^2 \Vert \cA v - \ell \Vert_{\Lp{2}}^2 \notag \\
    & \qquad + 2t(1-t) \langle \cA v - \ell , \cA u - \ell \rangle \notag ] \\
    &= \frac{\mu}{p} \ls t \Vert \nabla u \Vert_{\Lp{p}}^p + (1-t) \Vert \nabla v \Vert_{\Lp{p}}^p \rs + \frac{1}{2} [  t \Vert \cA u - \ell \Vert_{\Lp{2}}^2 + (1-t) \Vert \cA v - \ell \Vert_{\Lp{2}}^2 \notag \\
    & \qquad + t(1-t) \ls \langle \ell - \cA u , \cA u - \ell \rangle + \langle \ell - \cA v , \cA v - \ell \rangle + 2 \langle \cA v - \ell , \cA u - \ell \rangle \rs ] \notag \\
    &= \frac{\mu}{p} \ls t \Vert \nabla u \Vert_{\Lp{p}}^p + (1-t) \Vert \nabla v \Vert_{\Lp{p}}^p \rs +  \frac{1}{2} [  t \Vert \cA u - \ell \Vert_{\Lp{2}}^2 + (1-t) \Vert \cA v - \ell \Vert_{\Lp{2}}^2 \notag \\
    & \qquad -  t(1-t) \Vert \cA (u-v) \Vert_{\Lp{2}}^2] \notag \\
    & = t\cF(u) + (1-t)\cF(v) - \frac{1}{2} t(1-t) \|\cA(u-v)\|_{\Lp{2}}^2. \notag
\end{align}
Now, $\|\cA(u-v)\|_{\Lp{2}}^2=\langle \cA^*\cA(u-v),u-v\rangle \geq c_{\cA} \|u-v\|_{\Lp{2}}^2$ and so
we have that $\cF$ is $c_{\cA}$-convex. 
\end{proof}

\subsection{Background results} \label{sec:supplementary:background}

\subsubsection{Nonlinear semigroup theory}

For accretive operators, $(\Id + \lambda A)^{-{1}}$ -- also called the resolvent of $A$ -- is well-behaved: it is single-valued and Lipschitz with Lipschitz constant $1$ for $\lambda \geq 0$ (see \cite{andreu2004parabolic} and \cite[Section 2]{BenilanCrandall}).

Equipped with the two notions of accretivity and $m$-accretivity, one is able to prove the existence of solutions to \eqref{eq:background:orderRelation:abstractCauchyProblem} through semigroups (see for example \cite[Crandall-Liggett Theorem, A.28]{andreu2004parabolic} or Theorem \ref{thm:background:solution}) and one also obtains useful contraction properties (we mean results similar to \eqref{eq:background:orderRelation:contraction}) as the resolvent of the operators is well-behaved. However, by considering completely accretive operators (see~\cite{BenilanCrandall} and Definition \ref{def:background:completelyAccretive}), one can observe even better contraction behaviour as given in Lemma \ref{lem:background:orderRelation:contraction}.

\begin{mydef}[Normal functional]
We say that a functional $N:M(\Omega) \mapsto \bbR$ is normal if $N(v) \leq N(w)$ whenever $v \llSgt w$.
\end{mydef}

By definition of $v \llSgt w$, we have that $A$ is completely accretive if and only if for all normal functionals $N$, $\lambda >0$ and $(v,\hat{v}),(w,\hat{w}) \in \graph(A)$, we have $N(v-w) \leq N(v-w + \lambda(\hat{v} - \hat{w}))$. 

Compared to Definition \ref{def:background:accretive}, we note that Definition \ref{def:background:completelyAccretive} is independent of any norms. However, if $A$ is completely accretive on $M(\Omega)$ and $\graph(A) \subseteq V \times V$ where $V \subseteq M(\Omega)$ is a Banach space whose norm is a normal functional, then $A$ is accretive in $V$. As an example, one could pick $V = \Lp{p}(\Omega)$ for $1 \leq p \leq \infty$. Indeed the $p$-norm of $\Lp{p}$-spaces is a normal functional as the next lemma shows.

\begin{lemma}[Properties of the order relation] \label{lem:background:orderRelation:properties}
Let $\Omega$ be an open bounded subset of $\bbR^d$. Then,
\begin{enumerate}
\item For $u,v \in \Lp{1}(\Omega)$, assume that
\begin{equation*}
    \int_\Omega u(x)h(u(x)) \, \dd x \leq \int_\Omega v(x) h(u(x)) \, \dd x
\end{equation*}
for all $h \in \cH$. Then, we have $u \llSgt v$.
\item Let $1 \leq p \leq \infty$, $u \in \Lp{p}(\Omega)$ and $v$ a measurable function. If $v \llSgt u$, then $\Vert v \Vert_{\Lp{p}} \leq \Vert u \Vert_{\Lp{p}}$.
\item Let $u \in \Lp{1}(\Omega)$. Then, the set $\{v:\Omega \mapsto \bbR \spaceBar v \text{ is measurable and } v \llSgt u \}$ is a weakly sequentially compact subset of $\Lp{1}(\Omega)$. 
\end{enumerate}
\end{lemma}

\begin{proof} 
The proof of 1, 2 and 3 can be found in~\cite[Proposition 2.7, eq. (2.8) and Proposition 2.11(ii)]{BenilanCrandall} respectively.
\end{proof}

We now prove Lemma~\ref{lem:background:orderRelation:contraction}.

\begin{proof}[Proof of Lemma \ref{lem:background:orderRelation:contraction}]
By definition and Lemma \ref{lem:background:orderRelation:properties}, we know that for $(u,\hat{u}),(v,\hat{v}) \in \graph(A)$ and $\lambda > 0$, 
\[
 \Vert u - v \Vert_{\Lp{r}} \llSgt \Vert u - v + \lambda(\hat{u} - \hat{v})\Vert_{\Lp{r}}
\]
for $1 \leq r \leq \infty$. Then, we can apply \cite[Theorem A.20]{andreu2004parabolic}
to obtain \eqref{eq:background:orderRelation:contraction}. 

Uniqueness follows by assuming the existence of another strong solution $v$ solving~\eqref{eq:background:orderRelation:abstractCauchyProblem}. 
Then, from \eqref{eq:background:orderRelation:contraction}, we get that $u = v$.
\end{proof}

\subsubsection{Piecewise constant approximations}

\begin{proof}[Proof of Lemma \ref{lem:background:piecewise:equiRates}]
By assumption, the cells $\Omega_{n,\bi}$ have sides of length $1/n$, $\lambda_x(\Omega_{n ,\bi}) = n^{-d}$ and they are all translations of each other which means that $\Omega_{n,\bi}=\bi+\Omega_n$, where $\Omega_n$ is a hypercube of side $1/n$. Let $V_d$ be the volume of the unit ball in $\bbR^d$. We then have
\begin{align}
\|g-\cI_{\Pi_{\uni,n}}\cP_{\Pi_{\uni,n}}g\|_{\Lp{q}(\Omega)}^q  
&= \sum_{\bi \in [n]^d} \int_{\Omega_{n,\bi}} \abs{n^d \int_{\Omega_{n,\bi}} \pa{g(z)-g(x)} \,\dd z}^q \,\dd x \notag \\
&\leq \sum_{\bi \in [n]^d} \int_{\Omega_{n,\bi}} n^d \int_{\Omega_{n,\bi}} \abs{g(z)-g(x)}^q \, \dd z \, \dd x  \notag \\
&\leq n^d 
\sum_{\bi \in [n]^d} \int_{\Omega_{n,\bi}} 
\int_{|z| \leq \sqrt{d}/n }
\abs{g(x+z)-g(x)}^q \, \dd z \, \dd x \notag \\ 
&=n^d 
\int_{\Omega} 
\int_{|z| \leq \sqrt{d}/n }
\abs{g(x+z)-g(x)}^q \,\dd z \, \dd x  \notag \\
&=n^d 
\int_{|z| \leq \sqrt{d}/n}
\int_{\Omega} 
\abs{g(x+z)-g(x)}^q \, \dd x \, \dd z  \notag \\
&\leq n^d V_d (\sqrt{d}/n)^d 
\sup_{z \in \bbR^d, |z| \leq \sqrt{d}/n}
\int_{\Omega} 
\abs{g(\bx+\bz)-g(\bx)}^q \, \dd x  \notag \\
&=V_d ~ d^{d/2}  \omega(g,\sqrt{d}/n)_q^q \notag \\ 
&\leq V_d d^{\frac{d+sq}{2}} |g|_{\Lip(s,\Lp{q}(\Omega))}^q n^{-sq}. \notag
\end{align}

If $g \in \Ck{0,\alpha}$, we have:
\begin{align}
    |g(\cdot/\eps)|_{\Lip(\alpha,\Lp{q}(\Omega))} &= \sup_{h > 0} h^{-\alpha} \sup_{z \in \bbR^d, |z| < h} \l \int_{x, x + z \in \Omega}\left\vert g\l \frac{x+z}{\eps}\r-g\l \frac{x}{\eps} \r \right\vert^q \, \dd x \r^{1/q} \notag \\
    &\leq C h^{-\alpha} \frac{h^\alpha}{\eps^\alpha} \notag \\
    &= \frac{C}{\eps^\alpha}. \notag
\end{align}
Combining the latter with the first statement of the lemma yields the result.
\end{proof}

\subsection{Supplementary proofs}

\subsubsection{Well-posedness} \label{sec:supplementary:wellPosedness}

\begin{proof}[Proof of Proposition \ref{prop:proofs:wellPosedness:nonLocalProblem:dirichlet}]
We start by showing the first claim. 
Let $u \in \Lp{p}(\Omega)$ satisfy \eqref{eq:proofs:wellPosedness:nonLocalProblem:dirichlet:equationN} and $v \in \Lp{p}(\Omega)$. 
First note
\begin{align}
& \int_\Omega (\Delta_p^K u)(u - v) \, \dd x \notag \\
& \quad = \frac{1}{2} \int_{\Omega\times\Omega} K(\vert x - y \vert) \la u(y)-u(x) \ra^{p-2} (u(x)-u(y)) (u(x)-u(y) + v(y) - v(x)) \, \dd y \dd x \notag \\
& \quad = \frac{1}{2} \int_{\Omega\times\Omega} K(\vert x - y \vert) \vert u(y) - u(x) \vert^p \, \dd y \dd x \notag\\
& \qquad \qquad \qquad - \frac{1}{2} \int_{\Omega\times\Omega} K(\vert x - y \vert) \la u(y) - u(x)\ra^{p-2} (u(y)-u(x)) (v(y) - v(x)) \, \dd y \dd x. \label{eq:proofs:wellPosedness:nonLocalProblem:dirichlet:identity1} 
\end{align}

By multiplying both sides of \eqref{eq:proofs:wellPosedness:nonLocalProblem:dirichlet:equationN} by $(u-v)$, integrating and using \eqref{eq:proofs:wellPosedness:nonLocalProblem:dirichlet:identity1}, we obtain:
    \begin{align}
         &\frac{1}{n} \int_\Omega \vert u \vert^p \, \dd x + \frac{\lambda \mu}{2} \int_{\Omega\times\Omega} K(\vert x-y \vert) \vert u(x) - u(y) \vert^p  \, \dd y \dd x + \int_\Omega [ \lambda \l \cA u\r^2  - \lambda f u + u^2 - \phi u] \, \dd x  \notag \\
         & \qquad = \frac{1}{n} \int_\Omega \vert u \vert^{p-2}uv \, \dd x + \frac{\lambda \mu}{2} \int_{\Omega\times\Omega} K(\vert x-y \vert) | u(x) - u(y)|^{p-2} (u(x)-u(y))(v(x)-v(y))  \, \dd y \dd x \notag \\
         & \qquad \qquad \qquad + \int_\Omega [\lambda\cA u\cA v - \lambda f v + uv - \phi v] \, \dd x \notag \\
         & \qquad \leq  \frac{1}{n} \int_\Omega \vert u \vert^{p-1} \vert v \vert \, \dd x + \frac{\lambda \mu}{2} \int_{\Omega\times\Omega} K(\vert x-y \vert) \vert u(x) - u(y)\vert^{p-1}\vert v(y)-v(x) \vert  \, \dd y \dd x \notag \\
         & \qquad \qquad \qquad + \int_\Omega [ \lambda\cA u \cA v - \lambda f v + uv - \phi v] \, \dd x \notag\\
         & \qquad \leq \frac{1}{qn} \int_\Omega \vert u \vert^p \, \dd x + \frac{1}{pn} \int_\Omega \vert v \vert^p \, \dd x + \frac{\lambda \mu}{2q} \int_{\Omega\times\Omega} K(\vert x-y \vert) \vert u(x) - u(y) \vert^p  \, \dd y \dd x \notag \\
         & \qquad \qquad \qquad + \frac{\lambda \mu}{2p} \int_{\Omega\times\Omega} K(\vert x-y \vert) \vert v(x) - v(y) \vert^p  \, \dd y \dd x + \frac{\lambda}{2} \int_\Omega \l \cA u \r^2 \, \dd x + \frac{\lambda}{2} \int_\Omega \l \cA v \r^2 \, \dd x \label{eq:proofs:wellPosedness:nonLocalProblem:dirichlet:identity2}\\ 
         & \qquad \qquad \qquad + \frac{1}{2} \int_\Omega u^2 \, \dd x + \frac{1}{2} \int_\Omega v^2 \, \dd x - \int_\Omega [  \lambda fv + \phi v] \, \dd x \notag
    \end{align}
where $q$ is such that $p^{-1} + q^{-1} = 1$ and we used Young's inequality for products for \eqref{eq:proofs:wellPosedness:nonLocalProblem:dirichlet:identity2}. Again, rearranging the latter and adding
$\frac{1}{2} \int_\Omega \phi^2 \, \dd x$
on each side yields $E_{n,\lambda,\cA,f}(u) \leq E_{n,\lambda,\cA,f}(v)$. 

For the second claim, suppose that $u \in \Lp{p}(\Omega)$ is such that $E_{\lambda,\cA,f}(u) \leq E_{\lambda,\cA,f}(v)$ for all $v \in \Lp{p}(\Omega)$. Then, the function $t \mapsto E_{\lambda,\cA,f}(u + tv)$ has a minimum at 0 so that $\frac{\partial}{\partial t}E_{\lambda,\cA,f}(u + tv)\vert_{t = 0} = 0$. We therefore compute
\begin{align}
\frac{\partial}{\partial t}E_\lambda(u + tv)\vert_{t = 0} &= \frac{\lambda \mu}{2} \int_\Omega \int_\Omega K(\vert x - y \vert) \la u(x) - u(y)\ra^{p-2} (u(x)-u(y))(v(x)-v(y)) \, \dd y \dd x \notag\\
& \qquad \qquad + \lambda\int_\Omega (\cA^*\cA u - f) v \, \dd x + \int_\Omega (u-\phi)v \, \dd x \notag\\
&= \lambda \mu\int_\Omega \l \Delta_p^K u \r v \, \dd x + \lambda\int_\Omega (\cA^*\cA u - f) v \, \dd x + \int_\Omega (u-\phi)v \, \dd x = 0. \label{eq:proofs:wellPosedness:nonLocalProblem:dirichlet:identity3}
\end{align}
Since~\eqref{eq:proofs:wellPosedness:nonLocalProblem:dirichlet:identity3} holds for all $v \in \Ckc{\infty}(\Omega)$, we deduce that $\lambda_x$-a.e.
\[
\lambda(\mu\Delta_p^K u + \cA^* \cA u - f) + u - \phi = 0,
\]
which means that $u$ satisfies \eqref{eq:proofs:wellPosedness:nonLocalProblem:dirichlet:equation}.
\end{proof}

\begin{proof}[Proof of Lemma \ref{lem:proofs:wellPosedness:nonlocalProblem:properties}]
In the proof $C>0$ will denote a constant that can be arbitrarily large, (which might be) dependent on the kernel $K$, $\lambda$, $\mu$, $p$, $\cA$ and/or $\Omega$ (but independent of $n$, $u$, $v$, $w$, $t_n$ and $t$), that may change from line to line.

By Assumption \ref{ass:main:assumptions:S1} and the fact that $2 \leq p$, we have $q \leq 2 \leq p$ and therefore $\Lp{p}(\Omega) \subseteq \Lp{2}(\Omega) \subseteq \Lp{q}(\Omega)$ as well as $\frac{1}{C}\Vert \cdot \Vert_{\Lp{q}} \leq \Vert \cdot \Vert_{\Lp{2}} \leq C \Vert \cdot \Vert_{\Lp{p}}$. Let $u \in \Lp{p}(\Omega)$.

We note that 
\[
\left\Vert \frac{\vert u \vert^{p-2}u}{n} \right\Vert_{\Lp{q}}^q = \frac{1}{n^q} \int_\Omega \vert u \vert^{(p-1)q} \, \dd x = \frac{1}{n^q} \Vert u \Vert_{\Lp{p}}^p,
\]
\begin{align}
   \left\Vert \Delta_p^K u\right\Vert_{\Lp{q}}^q &= \int_\Omega \left\vert \int_\Omega K(\vert x-y \vert) \vert u(y) - u(x) \vert^{p-2}(u(y) - u(x)) \, \dd y  \right\vert^q \, \dd x \notag \\
   &\leq C \int_\Omega \int_\Omega \vert u(y) - u(x) \vert^{q \cdot (p-1) } \, \dd y \, \dd x \label{eq:proofs:wellPosedness:nonlocalProblem:properties:jensen}\\
   &\leq C \Vert u \Vert_{\Lp{p}}^p \notag
\end{align}
where we used Assumption \ref{ass:main:assumptions:K1} and Jensen's inequality for the convex function $\vert \cdot \vert^{q}$ for \eqref{eq:proofs:wellPosedness:nonlocalProblem:properties:jensen}.
Since \[
\|\cA^*\cA u\|_{\Lp{q}}\leq C\|\cA^*\cA u\|_{\Lp{2}} \leq C C_{\op}^2 \|u\|_{\Lp{2}} \leq C^2C_{\op}^2 \|u\|_{\Lp{p}}
\]
by Assumption~\ref{ass:main:assumptions:O1} we can infer
\[
    \Vert \cE_{n,\lambda,\cA,f}(u) \Vert_{\Lp{q}} \leq C \ls \l\frac{1}{n} + 1 \r \Vert u \Vert_{\Lp{p}}^{p/q} + \Vert u \Vert_{\Lp{p}} + \Vert f \Vert_{\Lp{2}} \rs. 
\]

We will now proceed to verify hemicontinuity.
Let $u,v,w\in\Lp{p}(\Omega)$ and $t_n\to t$.
We have 
\[
\la u(x) + t_n v(x) \ra^{p-2} \l u(x) + t_n v(x) \r w(x) \leq \vert u(x) + t_n v(x) \vert^{p-1} \vert w(x) \vert \leq C ( \vert u(x) \vert^{p-1} + C \vert v(x) \vert^{p-1}) \vert w(x) \vert  
\]
and 
\[
\int_\Omega ( \vert u \vert^{p-1} + C \vert v \vert^{p-1}) \vert w \vert \, \dd x \leq \left[ \int_\Omega ( \vert u \vert^{p-1} + C \vert v \vert^{p-1})^q \, \dd x \right]^{1/q} \Vert w \Vert_{\Lp{p}} \leq C \left[ \int_\Omega \vert u \vert^{p} + \vert v \vert^{p} \, \dd x \right]^{1/q} \Vert w \Vert_{\Lp{p}}
\]
where we used H\"older's inequality for the above. 
The right-hand side is finite by assumption and we can therefore apply the dominated convergence theorem to infer that 
\[
\lim_{n \to \infty} \int_\Omega \la u + t_nv\ra^{p-2} (u+t_n v) w \, \dd x \to \int_\Omega |u+tv|^{p-2} (u + tv)w \, \dd x.
\]
Similarly:
\begin{align*}
& K(\vert x-y \vert)  \la u(x) - u(y) + t_n (v(x) - v(y))\ra^{p-2} \l u(x) - u(y) + t_n(v(x) - v(y) \r w(x)\\
& \qquad \leq C \vert u(x) - u(y) + t_n (v(x) - v(y)) \vert^{p-1} \vert w(x) \vert  \\
& \qquad \leq C \l \vert u(x) \vert^{p-1} + \vert u(y) \vert^{p-1} + \vert v(x) \vert^{p-1} + \vert v(y) \vert^{p-1} \r \vert w(x) \vert 
\end{align*}
and, using H\"older's inequality,
\begin{align*}
& \int_{\Omega\times\Omega} K(\vert x-y \vert)  (\vert u(y) \vert^{p-1} + C \vert v(y) \vert^{p-1} + \vert u(x) \vert^{p-1} + C\vert v(x) \vert^{p-1} ) \vert w(x) \vert \, \dd x \, \dd y \\
& \qquad \leq C \left[ \int_\Omega \int_\Omega (\vert u(y) \vert^{p-1} + \vert v(y) \vert^{p-1} + \vert u(x) \vert^{p-1} + \vert v(x) \vert^{p-1} )^{q} \, \dd y \, \dd x  \right]^{1/q} \Vert w \Vert_{\Lp{p}(\Omega \times \Omega)} \\
& \qquad \leq C \int_{\Omega\times\Omega} \l \vert u(y) \vert^{p} + \vert v(y) \vert^{p} + \vert u(x) \vert^{p} + \vert v(x) \vert^{p} \r \, \dd x \, \dd y \Vert w \Vert_{\Lp{p}(\Omega)}.
\end{align*}
The right-hand side is finite by assumption on the functions as well as Assumption \ref{ass:main:assumptions:S1}. This allows us to use the dominated convergence theorem to conclude that
\[
\lim_{n \to \infty} \int_\Omega \Delta_p^K(u + t_n v) w \, \dd x = \int_\Omega \Delta_p^K(u + t v) w \, \dd x.
\]
Furthermore, it is clear that by linearity and Assumption \ref{ass:main:assumptions:O1}, we have
\[
\lim_{n\to \infty} \int_{\Omega} (u + t_n v) w \, \dd x = \int_{\Omega} (u + t v) w \, \dd x
\]
as well as 
\[
\lim_{n\to \infty} \int_{\Omega} \cA^* \cA(u + t_n v) w \, \dd x = \int_{\Omega} \cA^* \cA(u + t v) w \, \dd x.
\]
Combining the above implies that $\cE_{n,\lambda,\cA,f}(u+t_nv)$ converges weakly to $\cE_{n,\lambda,\cA,f}(u+tv)$ in $\Lp{q}$. 

Regarding monotony, we note that
\begin{align}
& \l |\tilde{u}(x,y)|^{p-2} \tilde{u}(x,y) - |\tilde{v}(x,y)|^{p-2}\tilde{v}(x,y) \r \l \tilde{u}(x,y) - \tilde{v}(x,y) \r \notag \\
& \qquad \qquad \geq |\tilde{u}(x,y)|^p + |\tilde{v}(x,y)|^p - \frac{p-1}{p} |\tilde{u}(x,y)|^p - \frac{1}{p} |\tilde{v}(x,y)|^p - \frac{p-1}{p} |\tilde{v}(x,y)|^p - \frac{1}{p} |\tilde{u}(x,y)|^p \notag \\
& \qquad \qquad = 0 \label{eq:utildevtildeMon}
\end{align}
by Young's inequality, and therefore (choosing $\tilde{u}(x,y) = u(x) - u(y)$ and $\tilde{v}(x,y)=v(x)-v(y)$) we have
\begin{align}
& \int_\Omega \l \Delta_p^K u - \Delta_p^K v \r (u-v) \, \dd x \notag \\ 
& \qquad = \frac12 \int_{\Omega\times\Omega} K(|x-y|) \Big( |u(x)-u(y)|^{p-2}(u(x)-u(y)) \notag \\
& \qquad \qquad \qquad - |v(x)-v(y)|^{p-2}(v(x)-v(y)) \Big) \l u(x) - u(y) - v(x) + v(y) \r \, \dd x \, \dd y \notag \\
& \qquad \geq 0. \label{eq:LapMono}
\end{align}
In addition,
\[
\int_\Omega \l \cA^*\cA u - \cA^* \cA v  \r (u -v) \, \dd x = \Vert \cA(u-v) \Vert_{\Lp{2}} \geq 0
\]
by Assumption \ref{ass:main:assumptions:O1}.
The latter facts and substituting $\tilde{u}(x,y) = u(x)$, $\tilde{v}(x,y) = v(x)$ in~\eqref{eq:utildevtildeMon} allow us to deduce that 
\[
\int_\Omega \l \cE_{n,\lambda,\cA,f}(u) - \cE_{n,\lambda,\cA,f}(v) \r (u-v) \, \dd x \geq 0
\]
which proves monotony. 

Finally, we tackle coercivity. To this purpose, note that 
\[
\frac{\int_\Omega (\vert u \vert^{p-2}u)u \, \dd x }{\Vert u \Vert_{\Lp{p}}} = \Vert u \Vert_{\Lp{p}}^{p-1},
\]
\[
\int_\Omega (\Delta_p^K u)u \, \dd x \geq 0 
\]
(by choosing $v=0$ in~\eqref{eq:LapMono}),
\[
\int_\Omega (\cA^* \cA u)u \, \dd x = \Vert \cA u \Vert_{\Lp{2}}^2 \geq 0
\]
and 
\[
-\frac{1}{\Vert u \Vert_{\Lp{p}}} \int_\Omega f u \, \dd x \geq - \Vert f \Vert_{\Lp{q}} > -\infty 
\]
since $\Vert f \Vert_{\Lp{q}} \leq C \Vert f \Vert_{\Lp{2}} < \infty$. This allows us to conclude that 
\[
\lim_{\Vert u \Vert_{\Lp{p}} \to \infty} \frac{\int_\Omega \cE_{n,\lambda,\cA,f}(u) u \, \dd x }{\Vert u \Vert_{\Lp{p}}} \geq \lim_{\|u\|_{\Lp{p}}\to\infty} \l \frac{1}{n} \|u\|_{\Lp{p}}^{p-1} - \|f\|_{\Lp{q}} \r = \infty
\]
which proves coercivity.
\end{proof}

\begin{proof}[Proof of Corollary \ref{cor:proofs:wellPosedness:nonlocalProblem:specialCases}]
From the proof of Lemma \ref{lem:proofs:wellPosedness:nonlocalProblem:properties}, we have that the evolution operator $\cE^K_{\cA,f}$ is well-defined from $\Lp{p}(\Omega)$ to $\Lp{q}(\Omega)$ with $p^{-1} + q^{-1} = 1$.

For the first part, we note that since $\Lp{q}(\Omega) \subseteq \Lp{p}(\Omega)$, for any $g \in \Lp{p}(\Omega)$, $g + \lambda \cE^K_{\cA,f}(g) \in \Lp{p}(\Omega)$ which shows that $\range(\Id + \lambda \cE) = \Lp{p}(\Omega)$ by Proposition \ref{prop:proofs:wellPosedness:nonlocalProblem:completeAccretivityRangeCondition}. This implies that $\cE^K_{\cA,f}$ is $m$-completely accretive. For the second part, we use the first statement of the corollary and the fact that $0 \in \Lp{p}(\Omega)$. For the third part, we note that whenever Assumption \ref{ass:main:assumptions:O2} was used, if the operator is 0, then the inequalities required in the proof of Proposition \ref{prop:proofs:wellPosedness:nonlocalProblem:completeAccretivityRangeCondition} are trivially satisfied.
\end{proof}

\begin{proof}[Proof of Proposition \ref{prop:proofs:nonlocalDiscreteContinuum:stability}]
In the proof $C>0$ will denote a constant that can be arbitrarily large, (which might be) dependent on $K_i$, $\Omega$, $\lambda$, $\mu$, $p$, $u_0$ and/or $f_i$, that may change from line to line.

By Theorem \ref{thm:proofs:wellPosedness:nonlocalProblem:existenceUniqueness}, we get the existence of the solutions $u_i$ for $i=1,2$. Without loss of generality, let $u_{0,2}, f_2 \in \Lp{2(p-1)}(\Omega)$. 

We define $\zeta(t,\cdot) = u_1(t,\cdot) - u_2(t,\cdot) \in \Lp{p}(\Omega)$ for $t < T$ and differentiate its squared $\Lp{2}$-norm:
\begin{align}
    \frac{1}{2}\frac{\partial}{\partial t} \Vert \zeta(t,\cdot) \Vert_{\Lp{2}}^2 
    &=\int_{\Omega} \l \frac{\partial }{\partial t} u_1(t,x)  - \frac{\partial }{\partial t} u_2(t,x) \r \zeta(t,x) \, \dd x \notag \\
    &= - \mu \int_{\Omega} \l \Delta_p^{K_1} u_1(t,x)  - \Delta_p^{K_1} u_2(t,x) \r \zeta(t,x) \, \dd x \notag \\
    & \qquad \qquad - \mu \int_{\Omega} \l \Delta_p^{K_1} u_2(t,x)  - \Delta_p^{K_2} u_2(t,x) \r \zeta(t,x) \, \dd x \notag \\ 
    & \qquad \qquad - \int_\Omega \cA( u_1(t,x) - u_2(t,x) ) \cA \zeta(t,x) \, \dd x + \int_\Omega \l f_1(x) - f_2(x) \r \zeta(t,x) \, \dd x \notag \\
    &\leq - \mu \int_{\Omega} \l \Delta_p^{K_1} u_2(t,x)  - \Delta_p^{K_2} u_2(t,x) \r \zeta(t,x) \, \dd x + \Vert f_1 - f_2 \Vert_{\Lp{2}} \Vert \zeta(t,\cdot) \Vert_{\Lp{2}} \label{eq:proofs:nonlocalDiscreteContinuum:stability:monotony} \\
    &=: T_1 + T_2 \notag
\end{align}
where we used
the monotony of the nonlocal $p$-Laplacian from \cite[Lemma 2.3]{ANDREU2008201} (see also Lemma~\ref{lem:proofs:wellPosedness:nonlocalProblem:properties} and in particular~\eqref{eq:LapMono}) as well as H\"older's inequality for \eqref{eq:proofs:nonlocalDiscreteContinuum:stability:monotony}. We now proceed as follows for the first term of the above:
\begin{align}
T_1 & \leq \mu \int_{\Omega\times\Omega} \la K_2(\vert x-y\vert) - K_1(\vert x-y\vert) \ra |u_2(t,x) - u_2(t,y)|^{p-1} \vert\zeta(t,x) \vert \, \dd x \, \dd y \notag \\
 & \leq \mu \left[ \int_{\Omega\times\Omega} \vert u_2(t,y) - u_2(t,x) \vert^{2(p-1)} \, \dd x \, \dd y \right]^{1/2} \notag \\
 & \qquad \qquad \times \left[ \int_{\Omega\times\Omega} \vert K_2(\vert x-y\vert) - K_1(\vert x-y\vert) \vert^2 \vert \zeta(t,x) \vert^2 \, \dd x \, \dd y \right]^{1/2} \notag\\
 & \leq C \Vert u_2(t,\cdot) \Vert_{\Lp{2(p-1)}}^{p-1} \left[ \sup_{x \in \Omega} \Vert K_2(\vert x - \cdot \vert) - K_1(\vert x - \cdot \vert) \Vert_{\Lp{2}} \right] \Vert \zeta(t,\cdot) \Vert_{\Lp{2}} \notag\\
 & \leq C (\Vert u_{0,2} \Vert_{\Lp{2(p-1)}} + T \Vert f_2 \Vert_{\Lp{2(p-1)}})^{p-1} \left[ \sup_{x \in \Omega} \Vert K_2(\vert x - \cdot \vert) - K_1(\vert x - \cdot \vert) \Vert_{\Lp{2}} \right] \Vert \zeta(t,\cdot) \Vert_{\Lp{2}} \label{eq:proofs:nonlocalDiscreteContinuum:stability:contraction} \\
    &\leq C(1 + T^{p-1}) \left[ \sup_{x \in \Omega} \Vert K_2(\vert x - \cdot \vert) - K_1(\vert x - \cdot \vert) \Vert_{\Lp{2}} \right] \Vert \zeta(t,\cdot) \Vert_{\Lp{2}} \label{eq:proofs:nonlocalDiscreteContinuum:stability:contraction2} \\
    &\leq CT^{p-1} \left[ \sup_{x \in \Omega} \Vert K_2(\vert x - \cdot \vert) - K_1(\vert x - \cdot \vert) \Vert_{\Lp{2}} \right] \Vert \zeta(t,\cdot) \Vert_{\Lp{2}} \notag
\end{align}
where we used \eqref{eq:proofs:wellPosedness:nonlocalProblem:existenceUniqueness:bound} for \eqref{eq:proofs:nonlocalDiscreteContinuum:stability:contraction} and the fact that $u_{0,2} \in \Lp{2(p-1)}(\Omega)$ and $f_2 \in \Lp{2(p-1)}(\Omega)$ for \eqref{eq:proofs:nonlocalDiscreteContinuum:stability:contraction2}.
Inserting the latter in \eqref{eq:proofs:nonlocalDiscreteContinuum:stability:monotony}, we obtain
\begin{align*}
\frac{\partial}{\partial t} \Vert \zeta(t,\cdot) \Vert_{\Lp{2}} & = \frac{1}{2 \Vert \zeta(t,\cdot) \Vert_{\Lp{2}}} \frac{\partial}{\partial t} \Vert \zeta(t,\cdot) \Vert_{\Lp{2}}^2 \\
 & \leq C T^{p-1} \left[ \sup_{x \in \Omega} \Vert K_2(\vert x - \cdot \vert) - K_1(\vert x - \cdot \vert) \Vert_{\Lp{2}} \right] + \Vert f_1 - f_2 \Vert_{\Lp{2}} \notag
\end{align*}
which we integrate from $0$ to $T$ to conclude:
\begin{align*}
\Vert u_2(t,\cdot) - u_1(t,\cdot) \Vert_{\Lp{2}}
&\leq CT^p \l \left[ \sup_{x \in \Omega} \Vert K_2(x,\cdot) - K_1(x,\cdot) \Vert_{\Lp{2}} \right] + \Vert f_1 - f_2 \Vert_{\Lp{2}} \r + \Vert u_{0,2} - u_{0,1} \Vert_{\Lp{2}}.
\end{align*}

Now, assume that without loss of generality $u_{0,2},f_2 \in \Lp{\infty}(\Omega)$. 
We proceed analogously with the appropriate modification to the $T_1$ term: 
\begin{align}
    T_1
    &\leq C \Vert u_2(t,\cdot) \Vert_{\Lp{\infty}}^{p-1} \int_{\Omega\times\Omega} \vert \left[ K_2(\vert x-y\vert) - K_1(\vert x-y\vert) \right] \zeta(t,x) \vert \, \dd x \, \dd y \notag \\
    &\leq C (\Vert u_{0,2} \Vert_{\Lp{\infty}} + T \Vert f_2 \Vert_{\Lp{\infty}})^{p-1} \Vert K_2 - K_1 \Vert_{\Lp{2}(\Omega \times \Omega)} \Vert \zeta \Vert_{\Lp{2}}\label{eq:proofs:nonlocalDiscreteContinuum:stability:contraction3} \\
    &\leq C T^{p-1} \Vert K_2 - K_1 \Vert_{\Lp{2}(\bbR)} \Vert \zeta \Vert_{\Lp{2}} \label{eq:proofs:nonlocalDiscreteContinuum:stability:Linfty}
\end{align}
where we used \eqref{eq:proofs:wellPosedness:nonlocalProblem:existenceUniqueness:bound} for \eqref{eq:proofs:nonlocalDiscreteContinuum:stability:contraction3}. We then conclude as above. 
\end{proof}

\subsubsection{Rates} \label{sec:supplementary:rates}

\begin{remark}[Consistency]\label{rem:consistency}
    Assuming the same setting as in Theorem \ref{thm:proofs:continuumRates:continuumNonlocalLocal}, we are also able to show the following consistency result:
    for $h \geq 1$, $s>2+\frac{d}{p}$ and $r>1+\frac{d}{p}$, suppose that, instead of Assumption~\ref{ass:main:assumptions:R2}, the solution $u$ to~\eqref{eq:main:notation:localProblem:localProblem} satisfies the weaker assumption
\begin{equation} \label{eq:proofs:rates:continuumRates:l2convergence:C2}
u\in \Lp{h}(0,T;\Wkp{s}{p}(\Omega)) \cap \Lp{\infty}(0,T;\Wkp{r}{p}(\Omega)).
\end{equation}
Then, we have the consistency result: for every $t$, as $n\to\infty$
\begin{equation} \label{eq:proofs:continuumRates:l2convergence:consistency}
\Vert u_{\eps_n}(t,\cdot) - u(t,\cdot) \Vert_{\Lp{2}(\Omega')} \to 0.
\end{equation}
\end{remark}

\begin{proof}[Proof of Remark \ref{rem:consistency} and Theorem \ref{thm:proofs:continuumRates:continuumNonlocalLocal}]
For ease of notation in the below, we will write $u(x)$ for $u(t,x)$ if the statement applies to all $t \in (0,T)$.

The existence of $u_{\eps_n}$ and $u$ are just applications of Theorem \ref{thm:proofs:wellPosedness:nonlocalProblem:existenceUniqueness} and Theorem \ref{thm:proofs:wellPosedness:localProblem:existenceUniqueness} respectively. 

Define $\zeta_{\eps_n}(t) = u_{\eps_n}(t) - u(t) \in \Wkp{1}{p}(\Omega)$ for $t < T$. Using H\"older's inequality, we start with the following computation:
\begin{align}
    \frac{1}{2}\frac{\partial}{\partial t} \Vert \zeta_{\eps_n}  \Vert_{\Lp{2}(\Omega')}^2 &= \int_{\Omega'} \l \frac{\partial}{\partial t} u_{\eps_n} - \frac{\partial}{\partial t} u \r \zeta_{\eps_n} \, \dd x \notag \\
    &= - \mu \int_{\Omega'} \l \Delta^{K_{\eps_n}}_p u_{\eps_n}\r \zeta_{\eps_n} \, \dd x - \int_{\Omega'} \cA^*(\cA u_{\eps_n} - \ell) \zeta_{\eps_n} \, \dd x \label{eq:proofs:rates:continuumRates:l2convergence:differenceEvolutions}\\ 
    & + \mu \int_{\Omega'} \Delta_p u \zeta_{\eps_n} \, \dd x + \int_{\Omega'} \cA^*(\cA u - \ell) \zeta_{\eps_n} \, \dd x \notag \\
    &= \underbrace{-\mu \int_{\Omega'} \l \Delta_p^{K_{\eps_n}} u_{\eps_n} - \Delta_p^{K_{\eps_n}} u \r (u_{\eps_n} - u) \, \dd x}_{=:T_1} \notag \\
    &+ \mu \int_{\Omega'} \l \Delta_p u - \Delta_p^{K_{\eps_n}} u \r \zeta_{\eps_n} \, \dd x \underbrace{- \int_{\Omega'} \cA \l u_{\eps_n} - u \r \cA \zeta_{\eps_n} \, \dd x}_{=:T_3} \notag \\
    &\leq T_1 + \mu \underbrace{\left\Vert \Delta_p u - \Delta_p^{K_{\eps_n}} u \right\Vert_{\Lp{2}(\Omega')}}_{=:T_2} \left\Vert \zeta_{\eps_n} \right\Vert_{\Lp{2}(\Omega')} + T_3  \notag
\end{align}
where we used \eqref{eq:main:notation:nonlocal:nonlocalProblem} and \eqref{eq:main:notation:localProblem:localProblem} for \eqref{eq:proofs:rates:continuumRates:l2convergence:differenceEvolutions}.
 Arguing as in the proof of the monotony in Lemma \ref{lem:proofs:wellPosedness:nonlocalProblem:properties} (see~\eqref{eq:LapMono}), we observe that $T_1 \leq 0$ and $T_3 = - \Vert \cA(u_{\eps_n} - u) \Vert_{\Lp{2}(\Omega)}^2 \leq 0$.

For $x \in \Omega'$, let $S_{\eps_n}(x) = \{z \in \bbR^d \spaceBar x + \eps_n z \in \Omega \} \cap \closure(B(0,1))$ and let $n$ be large enough so that $S_{\eps_n}(x) = \closure(B(0,1))$ for all $x \in \Omega'$  by Lemma \ref{lem:proofs:rates:continuumRates:domainIntegration}. 
Using a change of variables, we estimate as follows for the $T_2$ term:
\begin{align}
    T_2^2 &= \int_{\Omega'} \left\vert  \Delta_p^{K_{\eps_n}} u -\Delta_p u  \right\vert^2 \, \dd x  \notag \\
    & = \int_{\Omega^\prime} \la \frac{2}{c(p,d) \eps_n^{d+p}} \int_\Omega K\l\frac{|x-y|}{\eps_n}\r |u(y)-u(x)|^{p-2} (u(x)-u(y)) \, \dd y - \Delta_pu(x) \ra^2 \, \dd x \notag \\
    &= \int_{\Omega'} \left\vert -\frac{C(p,d)}{\eps_n^{p}} \int_{B(0,1)} K(\vert z \vert) \la u(x + \eps_n z)-u(x) \ra^{p-2} \l u(x)-u(x-\eps_n z) \r \, \dd z - \Delta_p u(x)   \right\vert^2 \, \dd x. \label{eq:proofs:rates:continuumRates:l2convergence:T2}
\end{align}

By assumptions on $s$, $p$ and $d$ and \ref{ass:main:assumptions:S2}, we can use Sobolev regularity in order to obtain that
\[
\begin{cases}
u \in \Ck{2}(\Omega) \text{ with $\sup_{t \in (0,T)} \Vert \nabla u(t,\cdot) \Vert_{\Lp{\infty}} < \infty$} & \text{if $u$ satisfies \eqref{eq:proofs:rates:continuumRates:l2convergence:C2}},\\
u \in \Ck{3}(\Omega) \text{ with $\max\{\sup_{t \in (0,T)} \Vert \nabla u(t,\cdot) \Vert_{\Lp{\infty}}, \sup_{t \in (0,T)} \Vert \nabla^2 u(t,\cdot) \Vert_{\Lp{\infty}}\} < \infty$} & \parbox[t]{.2\textwidth}{if $u$ satisfies Assumption \ref{ass:main:assumptions:R2}.}
\end{cases}
\]

We define the finite constants $C_1(t) = \Vert \nabla u(t,\cdot) \Vert_{\Lp{\infty}}$ and $C_2(t) = \Vert \nabla^2u(t,\cdot) \Vert_{\Lp{\infty}}$ and consider the inner integrand of \eqref{eq:proofs:rates:continuumRates:l2convergence:T2}. 
Letting $\psi(t) = |t|^{p-2}t$ and using $\psi(t) = \psi(a) + \psi^\prime(a)(t-a) + \cO(C^{p-3}|t-a|^2)$ for $a,t\in [-C,C]$ we have, for $u \in \Ck{2}(\Omega)$, 
\begin{align}
& |\underbrace{u(x+\eps_nz) - u(x)}_{=t}|^{p-2} \l u(x) - u(x+\eps_n z) \r \notag \\
& \qquad \qquad = |\underbrace{\eps_n z\cdot \nabla u(x)}_{=a}|^{p-2} \eps_n z \cdot \nabla u(x) \notag \\
 & \qquad \qquad \qquad \qquad + (p-1)  |\eps_n z \cdot \nabla u(x)|^{p-2} \l u(x+\eps_n z) - u(x) - \eps_n z \cdot \nabla u(x) \r \notag \\
 & \qquad \qquad \qquad \qquad + \cO\l C^{p-3} \la u(x+\eps_n z) - u(x) - \eps_n z \cdot \nabla u(x) \ra^2 \r \notag \\
 & \qquad \qquad = \eps_n^{p-1} |z\cdot \nabla u(x)|^{p-2} z\cdot \nabla u(x) \notag \\
 & \qquad \qquad \qquad \qquad + \frac{(p-1)\eps_n^p}{2} \la z\cdot \nabla u(x) \ra^{p-2} \l z^\top \nabla^2 u(x) z + o(|z|^2) \r  \label{eq:proofs:rates:continuumRates:l2convergence:TaylorC2} \\
 & \qquad \qquad \qquad \qquad + \cO(\eps_n^4 |z|^4 C^{p-3} C_2^2) \notag
\end{align}
where $\max\{|u(x+\eps_n z) - u(x)|,|\eps_n z \cdot \nabla u(x)|\}\leq C$.
We can choose $C\sim \eps_n |z| C_1$ so that
\begin{align*}
& |u(x+\eps_nz) - u(x)|^{p-2} \l u(x) - u(x+\eps_n z) \r = \eps_n^{p-1} |z\cdot \nabla u(x)|^{p-2} z\cdot \nabla u(x) \\
& \qquad \qquad + \frac{(p-1)\eps_n^p}{2} \la z\cdot \nabla u(x) \ra^{p-2} z^\top \nabla^2 u(x) z + o(\eps_n^p |z|^{p+1} C_1^{p-2}) + \cO(\eps_n^{p+1} |z|^{p+1} C_1^{p-3} C_2^2).
\end{align*}
Multiplying by $K(\vert z \vert)$ and integrating the latter equation over $B(0,1)$, we obtain:
\begin{align*}
& \int_{B(0,1)} K(\vert z \vert)|u(x+\eps_nz) - u(x)|^{p-2} \l u(x) - u(x+\eps_n z) \r \, \dd x \\
& \qquad \qquad = \eps_n^{p-1} \int_{B(0,1)} K(\vert z \vert) |u(x+\eps_nz) - u(x)|^{p-2} \l u(x) - u(x+\eps_n z) \r \, \dd z \notag \\
& \qquad \qquad \qquad + \frac{(p-1)\eps_n^p}{2} \int_{B(0,1)} K(\vert z \vert) |z\cdot \nabla u(x)|^{p-2} z^\top \nabla^2 u(x) z \, \dd z \notag + o(C_1^{p-2}\eps_n^{p}). 
\end{align*}
The first integral on the right hand side of the latter vanishes since it is odd in $z$; for the second integral, we make an appropriate change of variables as in \cite[Theorem A.1]{Calder_2018}
which yields:
\[ \frac{(p-1)\eps_n^p}{2} \int_{B(0,1)} K(\vert z \vert) |z\cdot \nabla u(x)|^{p-2} z^\top \nabla^2 u(x) z \, \dd z = \frac{\eps_n^p}{2} c(p,d) \Delta_p u(x).
\]
Finally, inserting these results in \eqref{eq:proofs:rates:continuumRates:l2convergence:T2}, we have:
\begin{align}
    T_2^2 &= \int_{\Omega'} \left\vert -\frac{C(p,d)}{\eps_n^p} \left[ \frac{\eps_n^p}{2}c(p,d)\Delta_p u + o(C_1^{p-2}\eps_n^{p}) 
    \right] - \Delta_p u  \right\vert^2 \, \dd x \notag \\
    &= o \l \sup_{t \in (0,T)} C_1(t)^{2p-4} \r. \label{eq:proofs:rates:continuumRates:l2convergence:T2final}
\end{align}
Coming back to our initial computation and using \eqref{eq:proofs:rates:continuumRates:l2convergence:T2final}, we obtain:
\[
\frac{\partial}{\partial t} \Vert \zeta_{\eps_n}(t) \Vert_{\Lp{2}(\Omega')} \leq \mu T_2 = o \l \sup_{t \in (0,T)} C_1(t)^{p-2} \r.
\]
Integrating the latter from $0$ to $t$ and recalling that the initial conditions are the same in \eqref{eq:main:notation:nonlocal:nonlocalProblem} and \eqref{eq:main:notation:localProblem:localProblem} we conclude that
\[
\Vert \zeta_{\eps_n}(t) \Vert_{\Lp{2}(\Omega')} = o \l t \sup_{t \in (0,T)} C_1(t)^{p-2} \r
\]
which proves \eqref{eq:proofs:continuumRates:l2convergence:consistency}.

In the case of $u$ satisfying Assumption \ref{ass:main:assumptions:R2}, we use
a third order Taylor expansion
in \eqref{eq:proofs:rates:continuumRates:l2convergence:TaylorC2} to obtain 
\begin{align*}
& |u(x+\eps_nz) - u(x)|^{p-2} \l u(x) - u(x+\eps_n z) \r \\
& \qquad \qquad = \eps_n^{p-1} |z\cdot \nabla u(x)|^{p-2} z \cdot \nabla u(x) + (p-1)\eps_n^{p-2} |z\cdot \nabla u(x)|^{p-2} \l \frac{\eps_n^2}{2} z^\top \nabla^2 u(x) z + \cO(\eps_n^3 |z|^3)\r \\
 & \qquad \qquad \qquad \qquad + \cO(\eps_n^4 |z|^4 C^{p-3} C_2^2) \\
& \qquad \qquad = \eps_n^{p-1} |z\cdot \nabla u(x)|^{p-2} z \cdot \nabla u(x) + \frac{(p-1)\eps_n^p}{2} |z\cdot \nabla u(x)|^{p-2} z^\top \nabla^2 u(x) z \\
 & \qquad \qquad \qquad \qquad + \cO(\eps_n^{p+1} |z|^{p+1}C_1^{p-2}) + \cO(\eps_n^{p+1} |z|^{p+1} C_1^{p-3} C_2^2)
\end{align*}
and, inserting the latter in \eqref{eq:proofs:rates:continuumRates:l2convergence:T2},
\begin{align*}
T_2^2 &= \mathcal{O}\l \eps_n^2 \l C_1^{p-2} + C_1^{p-3}C_2^2 \r^2\r \\
&= \mathcal{O}\l \eps_n^2 \ls \sup_{t \in (0,T)} C_1(t)^{p-2} + \sup_{t \in (0,T)} C_1(t)^{p-3}\sup_{t \in (0,T)} C_2(t)^{2} \rs^2 \r
\end{align*}
Proceeding as above, we obtain \eqref{eq:proofs:continuumRates:l2convergence:rates} which concludes our proof.
\end{proof}

\begin{proof}[Proof of Lemma \ref{lem:proofs:rates:NonFullyDiscreteLocal:evolutionTimeInterpolation}]
The existence and well-posedness of $\{\bar{u}_n^k\}_{k=0}^N$ follows from Corollary \ref{cor:proofs:wellPosedness:discreteProblem:existence}. For the second claim, let $t \in (t^{k-1},t^k]$ and compute:
\begin{align}
    \frac{\partial}{\partial t} \uTimeInter(t,x) &= \frac{\cI_n \bar{u}_n^{k-1} - \cI_n \bar{u}_n^{k-1} }{\tau^{k-1}} \notag \\
    &= \sum_{i=1}^{\vert \Pi_n \vert} \chi_{\pi_i^n} \ls  \frac{(\bar{u}_n^{k-1})_i - (\bar{u}_n^{k-1})_i }{\tau^{k-1}}  \rs \notag \\
    &= \sum_{i=1}^{\vert \Pi_n \vert} \chi_{\pi_i^n} \ls -\mu (\Delta_{p,n}^{\bar{K}} \bar{u}_n^{k})_i - (\bar{G}_n (\bar{u}_n^{k}))_i + (\bar{f})_i  \rs \label{eq:proofs:rates:NonFullyDiscreteLocal:evolutionTimeInterpolation:equqation1} \\
    &= - \mu \Delta_{p}^{\cI_n \bar{K}}(\cI_n \bar{u}_n^{k}) - \cA^*\cA(\cI_n \bar{u}_n^k) + \cI_n \bar{f}    \label{eq:proofs:rates:NonFullyDiscreteLocal:evolutionTimeInterpolation:equqation2}
\end{align}
where we used \eqref{eq:main:notation:discrete:nonlocalProblemFully} for \eqref{eq:proofs:rates:NonFullyDiscreteLocal:evolutionTimeInterpolation:equqation1} and Lemma \cite[Lemma 6.1]{ElBouchairi} as well as Assumption \ref{ass:main:assumptions:O4} for \eqref{eq:proofs:rates:NonFullyDiscreteLocal:evolutionTimeInterpolation:equqation2}. By noting that 
\[
\l \Delta_p^{\cI_n \bar{K}} \uTimeInject  \r (t,x) = \Delta_p^{\cI_n \bar{K}} \cI_n \bar{u}_n^k,
\]
and recalling linearity of $\cA^* \cA$ by Assumption \ref{ass:main:assumptions:O1}, this concludes our proof.

\end{proof}

\begin{proof}[Proof of Proposition \ref{prop:application:ratesDiscreteDeterministics}]
In the proof $C>0$ will denote a constant that can be arbitrarily large, (which might be) dependent on $\Omega$, $u_0$ or/and $\cA^*\ell$, that may change from line to line.

The existence of of $\{\bar{v}_n^k\}_{k=0}^N$ solving \eqref{eq:main:notation:discrete:nonlocalProblemFully}  with parameters $\bar{K}^{\eps_n}$, $\bar{f}$ and $\bar{u}_0$ follows from Corollary \ref{cor:proofs:rates:discreteNonlocalContinuumLocal:simplified}. Also, Assumption \ref{ass:main:assumptions:K1} and $\rho_n \ll \eps_n^{d+1}$ imply that $\rho_n \bar{K}^{\eps_n} \to 0$ so that we may assume without loss of generality that $\rho_n \bar{K}^{\eps_n} \leq 1$ and hence $\bar{\Lambda}_n$ is well-defined. By Corollary \ref{cor:application:wellPosedness}, we obtain, $\bbP$-a.e., a solution $\{\bar{u}_n^k\}_{k=0}^N$ solving \eqref{eq:main:notation:random:evolutionProblem} with parameters $\bar{\Lambda}_n$, $\bar{f}$ and $\bar{u}_0$.

Let $\alpha_i > 0$ for $1 \leq i \leq 2$ be such that $\mu(\alpha_1 + \alpha_2) = 1/2$. Let us define $\zeta(t) = \vTimeInter(t) - \uTimeInter(t)$ and start by verifying that $\zeta(t) \in \Lp{p}(\Omega)$. For that purpose, we note from the proof of Proposition \ref{prop:proofs:rates:NonFullyDiscreteLocal:L2rates} that $\Vert \vTimeInter \Vert_{\Lp{p}} \leq C(C + TC)$ and, $\bbP$-a.e., $\Vert \uTimeInter \Vert_{\Lp{p}} \leq C(C + TC)$ so that $\zeta(t) \in \Lp{p}(\Omega)$ uniformly in $t$. In order to proceed as in Proposition \ref{prop:proofs:rates:NonFullyDiscreteLocal:L2rates}, we will use Lemma \ref{lem:proofs:rates:NonFullyDiscreteLocal:evolutionTimeInterpolation} and note that $\bbP$-a.e., the results of the latter also apply to $\uTimeInter$ to deduce that
\begin{align}
    \frac{1}{2} \frac{\partial}{\partial t} &\Vert \zeta(t) \Vert_{\Lp{2}}^2 = - \mu \int_{\Omega} \l \Delta_p^{\cI_n \bar{K}^{\eps_n}} \vTimeInject - \Delta_p^{\cI_n \bar{\Lambda}_n} \vTimeInject \r \zeta \, \dd x \notag \\
    &- \mu \int_{\Omega} \l \Delta_p^{\cI_n \bar{\Lambda}_n} \vTimeInject - \Delta_p^{\cI_n \bar{\Lambda}_n} \uTimeInject \r (\vTimeInject - \uTimeInject) \, \dd x \notag \\
    &- \mu \int_{\Omega} \l \Delta_p^{\cI_n \bar{\Lambda}_n} \vTimeInject - \Delta_p^{\cI_n \bar{\Lambda}_n} \uTimeInject\r \ls (\vTimeInter - \vTimeInject) - (\uTimeInter -  \uTimeInject) \rs \, \dd x \label{eq:application:ratesDiscreteDeterministic:evolution} \\
    &- \int_\Omega \cA^*\cA(\vTimeInject- \uTimeInject) \zeta \, \dd x \notag \\
    &=: \mu T_1 + T_2 + \mu T_3 + T_4 \notag
\end{align}
where we used \eqref{eq:main:notation:discrete:nonlocalProblemFully} and \eqref{eq:main:notation:random:evolutionProblem} for \eqref{eq:application:ratesDiscreteDeterministic:evolution}. Arguing as in Proposition \ref{prop:proofs:wellPosedness:nonlocalProblem:completeAccretivityRangeCondition}, Proposition \ref{prop:proofs:nonlocalDiscreteContinuum:stability} or Proposition \ref{prop:proofs:rates:NonFullyDiscreteLocal:L2rates} (which relies on \cite[Lemma 2.3]{ANDREU2008201}), we obtain that $T_2 \leq 0$.

We continue our estimates by starting with an auxiliary result. In order to obtain \eqref{eq:proofs:rates:NonFullyDiscreteLocal:L2Rates:equation3}, we used a crude bound of the form $\Vert \Delta_p^{\cI_n \bar{K}} u \Vert_{\Lp{2}} \leq C \Vert \cI_n \bar{K} \Vert_{\Lp{\infty}} \Vert u \Vert_{\Lp{2(p-1)}}^{p-1} \leq C \Vert K \Vert_{\Lp{\infty}} \Vert u \Vert_{\Lp{2(p-1)}}^{p-1}$. A slightly finer analysis with Jensen's inequality yields $\Vert \Delta_p^{\cI_n \bar{K}} u \Vert_{\Lp{2}} \leq C \sup_{x \in \Omega
}\Vert \cI_n \bar{K}(x,\cdot) \Vert_{\Lp{1}} \Vert u \Vert_{\Lp{2(p-1)}}^{p-1}$. We use the latter two bounds and reason as we did for \eqref{eq:proofs:rates:NonFullyDiscreteLocal:L2Rates:differenceUs}, so that $\bbP$-a.e.:   
\begin{align}
    \Vert (\vTimeInter - \vTimeInject) - &(\uTimeInter -  \uTimeInject) \Vert_{\Lp{2}} \leq \Vert   \vTimeInter - \vTimeInject \Vert_{\Lp{2}} \notag \\
    &+ \Vert   \uTimeInter - \uTimeInject \Vert_{\Lp{2}} \notag \\
    &\leq \tau \l C + CT + (C+T^{p-1}C)\ls \sup_{x \in \Omega}\Vert \cI_n \bar{\Lambda}_n(x,\cdot) \Vert_{\Lp{1}} + \Vert \tilde{K}^{\eps_n} \Vert_{\Lp{\infty}}  \rs \r \label{eq:application:ratesDiscreteDeterminisitic:equation1}
\end{align}
where we used \eqref{eq:application:ratesDiscreteDeterministic:kHatBound} for \eqref{eq:application:ratesDiscreteDeterminisitic:equation1}. 

Using Young's inequality for products, similarly to \eqref{eq:proofs:rates:NonFullyDiscreteLocal:L2Rates:T4:4}, the latter allows us to compute, $\bbP$-a.e.:
\begin{align}
    \vert T_4 \vert &\leq \frac{1}{2} \Vert \zeta \Vert_{\Lp{2}}^2 +  \frac{1}{2} \Vert \cA^*\cA (\vTimeInject - \uTimeInject) \Vert_{\Lp{2}}^2  \notag \\
    &\leq \l \frac{1 + 3 C_{\mathrm{op}}^4}{2} \r \Vert \zeta \Vert_{\Lp{2}}^2 + \ls \tau \l C + CT + (C+T^{p-1}C)\ls \sup_{x \in \Omega}\Vert \cI_n \bar{\Lambda}_n(x,\cdot) \Vert_{\Lp{1}} + \Vert \tilde{K}^{\eps_n} \Vert_{\Lp{\infty}}  \rs \r \rs^2. \label{eq:application:ratesDiscreteDeterministic:T4}
\end{align}

For the $T_3$ term, we have
\[
\vert T_3 \vert \leq \Vert  \Delta_p^{\cI_n \bar{\Lambda}_n} \vTimeInject - \Delta_p^{\cI_n \bar{\Lambda}_n} \uTimeInject \Vert_{\Lp{2}} \Vert (\vTimeInter - \vTimeInject) - (\uTimeInter -  \uTimeInject) \Vert_{\Lp{2}}.
\] 
We now assume that $u_0 \in \Lp{2p - 2/(p-1))}(\Omega)$ and $\cA^*\ell \in \Lp{2p - 2/(p-1))}(\Omega) $. Analogously to how we obtained \eqref{eq:proofs:rates:NonFullyDiscreteLocal:L2Rates:contraction} and using Jensen's inequality as described for \eqref{eq:application:ratesDiscreteDeterminisitic:equation1}, $\bbP$-a.e., we have that:
\begin{align}
    &\Vert  \Delta_p^{\cI_n \bar{\Lambda}_n} \vTimeInject - \Delta_p^{\cI_n \bar{\Lambda}_n} \uTimeInject \Vert_{\Lp{2}} \leq (C + T^{(p-1)-1/p}) \sup_{x \in \Omega}\Vert \cI_n \bar{\Lambda}_n(x,\cdot) \Vert_{\Lp{1}} \notag \\
    &\times \Vert \vTimeInject - \uTimeInject \Vert_{\Lp{2}}^{1/p} \notag \\
    &\leq (C + CT^{(p-1)-1/p}) \sup_{x \in \Omega}\Vert \cI_n \bar{\Lambda}_n(x,\cdot) \Vert_{\Lp{1}} \notag \\
    &\times \ls \Vert \vTimeInject - \vTimeInter \Vert_{\Lp{2}} + \Vert \vTimeInter - \uTimeInter \Vert_{\Lp{2}} + \Vert \uTimeInter - \uTimeInject \Vert_{\Lp{2}} \rs^{1/p} \notag \\
    &\leq (C + CT^{(p-1)-1/p}) \sup_{x \in \Omega}\Vert \cI_n \bar{\Lambda}_n(x,\cdot) \Vert_{\Lp{1}} \notag \\
    &\times \l \ls \tau \l C + CT + (C+T^{p-1}C)\ls \sup_{x \in \Omega}\Vert \cI_n \bar{\Lambda}_n(x,\cdot) \Vert_{\Lp{1}} + \Vert \tilde{K}^{\eps_n} \Vert_{\Lp{\infty}}  \rs \r \rs^{1/p} + \Vert \zeta \Vert_{\Lp{2}}^{1/p} \r \label{eq:application:ratesDiscreteDeterminisitic:equation2}
\end{align}
where we used \eqref{eq:application:ratesDiscreteDeterminisitic:equation1} for \eqref{eq:application:ratesDiscreteDeterminisitic:equation2}. Combining \eqref{eq:application:ratesDiscreteDeterminisitic:equation1} and \eqref{eq:application:ratesDiscreteDeterminisitic:equation2}, we return to $T_3$:
\begin{align}
    &\vert T_3 \vert \leq \tau \sup_{x \in \Omega}\Vert \cI_n \bar{\Lambda}_n(x,\cdot) \Vert_{\Lp{1}} (C + CT^{(p-1)-1/p}) \notag \\
    &\times \l C + CT + (C+T^{p-1}C)\ls \sup_{x \in \Omega}\Vert \cI_n \bar{\Lambda}_n(x,\cdot) \Vert_{\Lp{1}}  + \Vert \tilde{K}^{\eps_n} \Vert_{\Lp{\infty}}  \rs \r \Vert \zeta \Vert_{\Lp{2}}^{1/p} \notag \\
    & + \Biggl[  \tau \l (C + T^{(p-1)-1/p}) \sup_{x \in \Omega}\Vert \cI_n \bar{\Lambda}_n(x,\cdot) \Vert_{\Lp{1}} \r^{p/(p+1)} \notag \\
    &\times \l C + CT + (C+T^{p-1}C)\ls \sup_{x \in \Omega}\Vert \cI_n \bar{\Lambda}_n(x,\cdot) \Vert_{\Lp{1}}  + \Vert \tilde{K}^{\eps_n} \Vert_{\Lp{\infty}}  \rs \r \Biggr]^{(p+1)/p} \notag \\
    &\leq \alpha_1 \Vert \zeta \Vert_{\Lp{2}}^2 +  \Biggl[ \tau \sup_{x \in \Omega}\Vert \cI_n \bar{\Lambda}_n(x,\cdot) \Vert_{\Lp{1}}  (C + CT^{(p-1)-1/p}) \label{eq:application:ratesDiscreteDeterminisitic:T3} \\
    &\times \l C + CT + (C+T^{p-1}C)\ls \sup_{x \in \Omega}\Vert \cI_n \bar{\Lambda}_n(x,\cdot) \Vert_{\Lp{1}}  + \Vert \tilde{K}^{\eps_n} \Vert_{\Lp{\infty}}  \rs \r \Biggr]^{2p/(2p-1)} \notag \\
    &+ \Biggl[  \tau \l (C + CT^{(p-1)-1/p}) \sup_{x \in \Omega}\Vert \cI_n \bar{\Lambda}_n(x,\cdot) \Vert_{\Lp{1}} \r^{p/(p+1)} \notag \\
    &\times \l C + CT + (C+T^{p-1}C)\ls \sup_{x \in \Omega}\Vert \cI_n \bar{\Lambda}_n(x,\cdot) \Vert_{\Lp{1}}  + \Vert \tilde{K}^{\eps_n} \Vert_{\Lp{\infty}}  \rs \r \Biggr]^{(p+1)/p} \notag 
\end{align}
where we used Young's inequality for products for \eqref{eq:application:ratesDiscreteDeterminisitic:T3}.

We now define the random variables \[
Z_i^k = n^{-1} \sum_{j=1}^n \l \l\bar{\Lambda}_n\r_{ij} - \l\bar{K}^{\eps_n}\r_{ij} \r \left\vert \l \bar{v}_n^k \r_j - \l \bar{v}_n^k \r_i  \right\vert^{p-2} \l \l \bar{v}_n^k \r_j - \l \bar{v}_n^k \r_i \r \]
for $1 \leq i \leq n$ and $1 \leq k \leq N$. Let $Z^k = (Z_1^k,\dots,Z_n^k) \in \bbR^{n}$. Then, for $t \in (t^{k-1},t^k]$: 
\[
\vert T_1 \vert \leq \Vert \zeta \Vert_{\Lp{2}} \Vert \Delta_p^{\cI_n \bar{K}^{\eps_n}} \vTimeInject - \Delta_p^{\cI_n \bar{\Lambda}_n} \vTimeInject \Vert_{\Lp{2}} = \Vert \zeta \Vert_{\Lp{2}} \Vert \cI_n Z^k \Vert_{\Lp{2}}.
\]
For $\theta > 0$, we will now bound $\Vert \cI_n Z_i^k \Vert_{\Lp{2}}$. We start by estimating:
\begin{align}
    &\bbP \l \Vert \cI_n Z_i^k \Vert_{\Lp{2}} \geq \theta \r 
    \leq \frac{1}{\theta^2 n} \sum_{i=1}^n \bbE \l \ls Z_i^k \rs^2 \r \label{eq:application:ratesDiscreteDeterministic:Markov} \\
    &= \frac{1}{\theta^2 n^3 \rho_n^2} \sum_{i=1}^n \sum_{j=1}^n \rho_n \l\bar{K}^{\eps_n}\r_{ij}(1 - \rho_n \l\bar{K}^{\eps_n}\r_{ij}) \ls \left\vert \l \bar{v}_n^k \r_j - \l \bar{v}_n^k \r_i  \right\vert^{p-2} \l \l \bar{v}_n^k \r_j - \l \bar{v}_n^k \r_i \r \rs^2 \label{eq:application:ratesDiscreteDeterministic:independence} \\
    &\leq \frac{1}{\theta^2 n \rho_n} \int_\Omega \int_{\Omega} \cI_n \bar{K}^{\eps_n}(x,y) \vert (\cI_n \bar{v}_n^k)(y) - (\cI_n \bar{v}_n^k)(x)  \vert^{2(p-1)} \, \dd y \dd x \notag \\
    &\leq \frac{C}{\theta^2 n \rho_n} \Vert \tilde{K}^{\eps_n} \Vert_{\Lp{\infty}} \Vert \cI_n \bar{v}_n^k \Vert_{\Lp{2(p-1)}}^{2(p-1)}  \label{eq:application:ratesDiscreteDeterministic:KBoundLpNorm}\\
    &\leq \frac{(C + T^{2(p-1)}C)}{\theta^2 n \rho_n } \Vert \tilde{K}^{\eps_n} \Vert_{\Lp{\infty}} \label{eq:application:ratesDiscreteDeterministic:contraction}
\end{align}
where we used Markov's inequality for \eqref{eq:application:ratesDiscreteDeterministic:Markov}, 
the fact that $\rho_n \Lambda_{ij}^n$ are $\text{Ber}(\rho_n \l\bar{K}^{\eps_n}\r_{ij})$ random variables for \eqref{eq:application:ratesDiscreteDeterministic:independence},
\eqref{eq:application:ratesDiscreteDeterministic:kHatBound} for \eqref{eq:application:ratesDiscreteDeterministic:KBoundLpNorm} and \eqref{eq:proofs:wellPosedness:discreteProblem:uniformBound} for \eqref{eq:application:ratesDiscreteDeterministic:contraction}. Returning to $T_1$ and using Young's inequality for products, we obtain:
\begin{equation} \label{eq:application:ratesDiscreteDeterministicT1}
    \vert T_1 \vert \leq  C\theta^2 + \alpha_2 \Vert \zeta \Vert_{\Lp{2}}^2 
\end{equation}
with probability larger than $1 - \frac{(C + T^{2(p-1)}C)}{\theta^2 n \rho_n } \Vert \tilde{K}^{\eps_n} \Vert_{\Lp{\infty}}$. 

We finally combine \eqref{eq:application:ratesDiscreteDeterministicT1}, \eqref{eq:application:ratesDiscreteDeterminisitic:T3}, \eqref{eq:application:ratesDiscreteDeterministic:T4} to obtain:
\begin{align}
    &\frac{\partial}{\partial t} \Vert \zeta \Vert_{\Lp{2}}^2 \leq 2 \l \l \frac{1 + 3 C_{\mathrm{op}}^4}{2} \r  + \mu(\alpha_1 + \alpha_2) \r \Vert \zeta \Vert_{\Lp{2}}^2 + C \theta^2 \notag \\
    & + \Biggl[ \tau \sup_{x \in \Omega}\Vert \cI_n \bar{\Lambda}_n(x,\cdot) \Vert_{\Lp{1}}  (C + CT^{(p-1)-1/p}) \notag \\
    &\times \l C + CT + (C+T^{p-1}C)\ls \sup_{x \in \Omega}\Vert \cI_n \bar{\Lambda}_n(x,\cdot) \Vert_{\Lp{1}}  + \Vert \tilde{K}^{\eps_n} \Vert_{\Lp{\infty}}  \rs \r \Biggr]^{2p/(2p-1)} \notag \\
    & + \Biggl[  \tau \l (C + CT^{(p-1)-1/p}) \sup_{x \in \Omega}\Vert \cI_n \bar{\Lambda}_n(x,\cdot) \Vert_{\Lp{1}} \r^{p/p+1} \notag \\
    & \times \l C + CT + (C+T^{p-1}C)\ls \sup_{x \in \Omega}\Vert \cI_n \bar{\Lambda}_n(x,\cdot) \Vert_{\Lp{1}} + \Vert \tilde{K}^{\eps_n} \Vert_{\Lp{\infty}}  \rs \r \Biggr]^{(p+1)/p} \notag \\
    & + \ls \tau \l C + CT + (C+T^{p-1}C)\ls \sup_{x \in \Omega}\Vert \cI_n \bar{\Lambda}_n(x,\cdot) \Vert_{\Lp{1}}  + \Vert \tilde{K}^{\eps_n} \Vert_{\Lp{\infty}}  \rs \r \rs^2 \notag \\
    &:= \l 2 + 3C_{\mathrm{op}}^4 \r \Vert \zeta \Vert_{\Lp{2}}^2 + T_7. \notag
\end{align}
We proceed analogously to \eqref{eq:proofs:rates:NonFullyDiscreteLocal:L2Rates:gronwall} to conclude that
$\Vert \zeta \Vert_{\Lp{2}} \leq C e^{\l \frac{2 + 3C_{\mathrm{op}}^4}{2} \r T}  \cdot T_7^{1/2}.$

Let us now assume that $u_0 \in \Lp{\infty}(\Omega)$ and $\cA^*\ell \in \Lp{\infty}(\Omega)$. We be modifying our estimates of $T_3$ and $T_1$ slightly. By following the computation that lead to \eqref{eq:proofs:rates:NonFullyDiscreteLocal:L2Rates:domain3} and using Jensen's inequality as we did for \eqref{eq:application:ratesDiscreteDeterminisitic:equation2}, we obtain:
\begin{align}
    &\Vert  \Delta_p^{\cI_n \bar{\Lambda}_n} \vTimeInject - \Delta_p^{\cI_n \bar{\Lambda}_n} \uTimeInject \Vert_{\Lp{2}} \leq (C + T^{p-2}C) \sup_{x \in \Omega}\Vert \cI_n \bar{\Lambda}_n(x,\cdot) \Vert_{\Lp{1}} \Vert \vTimeInject - \uTimeInject \Vert_{\Lp{2}} \notag \\
    &\leq (C + T^{p-2}C) \sup_{x \in \Omega}\Vert \cI_n \bar{\Lambda}_n(x,\cdot) \Vert_{\Lp{1}} \Vert \zeta \Vert_{\Lp{2}} \notag \\
    &+ (C + T^{p-2}C) \sup_{x \in \Omega}\Vert \cI_n \bar{\Lambda}_n(x,\cdot) \Vert_{\Lp{1}} \ls \tau \l C + CT + (C+T^{p-1}C)\ls \sup_{x \in \Omega}\Vert \cI_n \bar{\Lambda}_n(x,\cdot) \Vert_{\Lp{1}} + \Vert \tilde{K}^{\eps_n} \Vert_{\Lp{\infty}}  \rs \r \rs \label{eq:application:ratesDiscreteDeterminisitic:equation22}
\end{align}
where we used \eqref{eq:application:ratesDiscreteDeterminisitic:equation1} for \eqref{eq:application:ratesDiscreteDeterminisitic:equation22}, which in turn, using Young's inequality for inequalities, yields:
\begin{align}
    &\vert T_3 \vert
    \leq \alpha_2 \Vert \zeta \Vert_{\Lp{2}} \notag \\
    &+ \ls \tau (C + T^{p-2}C) \sup_{x \in \Omega}\Vert \cI_n \bar{\Lambda}_n(x,\cdot) \Vert_{\Lp{1}} \l C + CT + (C+T^{p-1}C)\ls \sup_{x \in \Omega}\Vert \cI_n \bar{\Lambda}_n(x,\cdot) \Vert_{\Lp{1}} + \Vert \tilde{K}^{\eps_n} \Vert_{\Lp{\infty}}  \rs \r \rs^2 \notag \\
    &+ \Biggl[ \tau \l (C + T^{p-2}C) \sup_{x \in \Omega}\Vert \cI_n \bar{\Lambda}_n(x,\cdot) \Vert_{\Lp{1}} \r^{1/2} \notag \\
    &\times \l C + CT + (C+T^{p-1}C)\ls \sup_{x \in \Omega}\Vert \cI_n \bar{\Lambda}_n(x,\cdot) \Vert_{\Lp{1}} + \Vert \tilde{K}^{\eps_n} \Vert_{\Lp{\infty}}  \rs \r \Biggr]^2. \notag
\end{align}
For the modification of $T_1$, we just take the $\Lp{\infty}$-norm in \eqref{eq:application:ratesDiscreteDeterministic:KBoundLpNorm}, but again using \eqref{eq:proofs:wellPosedness:discreteProblem:uniformBound}, we obtain the same estimate as in \eqref{eq:application:ratesDiscreteDeterministic:contraction}. We conclude as above.

\end{proof}

\end{document}